\documentclass{amsart}
\usepackage[left=1.2in,right=1.2in,top=1.2in,bottom=1.2in]{geometry}
\usepackage[utf8]{inputenc}
\usepackage[english]{babel}

\usepackage{graphicx}
\graphicspath{ {./images/} }
\usepackage{amsmath, lipsum}
\usepackage{amssymb}
\usepackage{amsthm}
\usepackage{framed}
\usepackage{mathtools}
\usepackage[colorlinks]{hyperref}
\usepackage{enumerate}
\usepackage{enumitem}
\usepackage{booktabs}
\usepackage{array}
\usepackage{faktor}
\usepackage{cancel}
\usepackage{makecell}
\usepackage{pbox}
\usepackage{bm}
\usepackage[toc,page]{appendix}
\usepackage{etoolbox}
\usepackage{stmaryrd}
\usepackage{extarrows}
\usepackage{blindtext}
\usepackage{comment}
\usepackage{centernot}
\usepackage{floatrow}
\allowdisplaybreaks[1]
\DeclareRobustCommand\longtwoheadrightarrow
{\relbar\joinrel\twoheadrightarrow}
\usepackage{tikz}
\usepackage{tikz-cd}
\usepackage{quiver}
\usetikzlibrary{arrows.meta}
\usepackage[all,cmtip]{xy}
\newcommand{\Wedge}{\bigwedge}
\newcommand{\got}{\mathfrak}
\newcommand{\cali}{\mathcal}

\newcommand{\eqdef}{\coloneqq}

\newcommand{\mto}{\mapsto}
\newcommand{\mbb}{\mathbb}

\newcommand{\fracddtz}{\frac{d}{dt}\bigg{|}_{t=0}}
\newcommand{\fracddtzs}{\frac{d}{dt}\big{|}_{t=0}}
\newcommand{\bigleftpar}{\big{(}}
\newcommand{\bigrightpar}{\big{)}}
\newcommand{\biggleftpar}{\bigg{(}}
\newcommand{\biggrightpar}{\bigg{)}}
\newcommand{\mbf}{\mathbf}

\newcommand{\llbra}{\llbracket}
\newcommand{\rrbra}{\rrbracket}
\newcommand{\dr}{\mathbf{d}}

\newcommand{\biggleftbra}{\bigg{[}}
\newcommand{\biggrightbra}{\bigg{]}}

\newcommand{\simeqd}{\mathrel{\rotatebox[origin=c]{-90}{$\simeq$}}}
\newcommand{\Tau}{\mathrm{T}}

\DeclareMathOperator{\Ad}{Ad}
\DeclareMathOperator{\ad}{ad}

\DeclareMathOperator{\GL}{GL}
\DeclareMathOperator{\Gr}{Gr}
\DeclareMathOperator{\Hom}{Hom}

\DeclareMathOperator{\End}{End}
\DeclareMathOperator{\Aut}{Aut}
\DeclareMathOperator{\id}{id}

\DeclareMathOperator{\grap}{graph}
\DeclareMathOperator{\Id}{Id}
\DeclareMathOperator{\Coker}{Coker}
\DeclareMathOperator{\Ker}{Ker}

\DeclareMathOperator{\tr}{tr}

\DeclareMathOperator{\Kur}{Kur}

\DeclareMathOperator{\defor}{def}

\DeclareMathOperator{\ver}{vert}

\DeclareMathOperator{\pr}{pr}
\DeclareMathOperator{\MC}{MC}

\DeclareMathOperator{\Graph}{graph}

\DeclareMathOperator{\Bott}{Bott}

\DeclareMathOperator{\res}{res}

\DeclareMathOperator{\img}{Im}
\DeclareMathOperator{\Taut}{Taut}
\DeclareMathOperator{\Cent}{Cent}

\theoremstyle{plain}
\newtheorem{thm}{Theorem}[section]

\theoremstyle{definition}
\newtheorem{defin}[thm]{Definition}
\newtheorem{ex}[thm]{Example}
\newtheorem{lem}[thm]{Lemma}
\newtheorem{prop}[thm]{Proposition}
\newtheorem{cor}[thm]{Corollary}
\newtheorem{rem}[thm]{Remark}

\title{Deformations of ideals in Lie algebras}
\author{Ilias Ermeidis$^{*,\star,1}$}
\author{Madeleine Jotz$^\dagger$}
\address{$^\star$Mathematics Institute, Georg-August-University of Göttingen, Göttingen 37073, Germany}
\email{ilias.ermeidis@mathematik.uni-goettingen.de}
\address{$^1$Current address: Department of Mathematics, Aristotle University of Thessaloniki, Thessaloniki 54124, Greece}
\email{iermeida@math.auth.gr}
\address{$^\dagger$Institut für Mathematik, Julius-Maximilians-Universität Würzburg, Würzburg 97074, Germany}
\email{madeleine.jotz@uni-wuerzburg.de}
\date{}

\thanks{This project was partially funded by the RTG 2491 in G\"ottingen and by a GSSP-DAAD fellowship at the University of G\"ottingen.}

\thanks{$^*$Corresponding author}

\begin{document}

	\maketitle
	
	\begin{abstract}
		
		This paper develops the deformation theory of Lie ideals. It shows that the 
		smooth deformations  of an ideal $\got i$ in a Lie algebra $\got g$ differentiate to cohomology classes in the cohomology of $\got g$ with values in its adjoint representation on $\Hom(\got i, \got g/\got i)$. The cohomology associated with the ideal $\got i$ in $\got g$ is compared with other Lie algebra cohomologies defined by $\got i$, such as the cohomology defined by $\got i$ as a Lie subalgebra of $\got g$, and the cohomology defined by the Lie algebra morphism $\got g \to \got g/\got i$.

		After a choice of complement of the ideal $\got i$ in the Lie algebra $\got g$, its deformation complex is enriched to the differential graded Lie algebra that controls its deformations, in the sense that its Maurer-Cartan elements are in one-to-one correspondence with the (small) deformations of the ideal. These constructions are shown to hold independently of the choice of complement -- up to isomorphism. Furthermore, the $L_{\infty}$-algebra that simultaneously controls the deformations of $\got{i}$ and of the ambient Lie bracket is identified.

		Under appropriate assumptions on the low degrees
		of the deformation cohomology of a given Lie ideal, the (topological) rigidity and  stability of ideals are studied, as well as obstructions to deformations of ideals of Lie algebras.
	\end{abstract}

	\bigskip

	\tableofcontents
	
	\section{Introduction}
	Ideals in Lie algebras are crucial in the representation theory and classification of the latter.
	Therefore, for understanding a given Lie algebra, it is crucial to have a good
	comprehension of its Lie ideals — such as how many there are, i.e.~whether there are none, few,
	or many, and how much these essentially differ from one another. 
	This is  where
	deformation theory comes into play. 
	Deformation theory serves as a powerful tool for understanding how mathematical structures change under formal or smooth perturbations. Originating in algebraic geometry, where it was used to study the deformations of complex structures, the theory has since found wide-ranging applications in algebra, topology, and mathematical physics. In the context of Lie algebras, deformation theory investigates how the structure of a Lie algebra can be modified or preserved under small changes to its Lie bracket.
	
	The foundations of deformation theory 
	were
	laid by Kodaira and Spencer in their remarkable series of works on deformations of
	complex manifolds \cite{Kodaira-Spencer-On-deformations-of-complex-analytic-structures-Part1-and-Part2-58',Kodaira-Spencer-On-the-existence-of-deformations-of-complex-analytic-structures-58',Kodaira-Spencer-deformations-of-complex-structures-Part3-60',Kodaira-book-on-complex-manifolds-and-deformation-of-complex-structures-86'}. 
	Gerstenhaber then studied deformations of rings
	and associative algebras \cite{Gerstenhaber-the-cohomology-structure-of-an-associative-ring-63',Gerstenhaber-On-deformations-of-rings-and-algebras-part1-64',Gerstenhaber-on-deformations-of-rings-and-algebras-part2-66',Gerstenhaber-on-deformations-of-rings-and-algebras-part3-68',Gerstenhaber-on-deformations-of-rings-and-algebras-part4-74',Gerstenhaber-On-deformations-of-rings-and-algebras-part5-1999}. Foundational work in this area was conducted as well in the 1960's by Nijenhuis and Richardson, who developed key techniques for studying infinitesimal deformations of algebraic structures. In particular, they introduced the concept of deformation cohomology and provided criteria for the integrability of infinitesimal deformations into formal deformations. Their work established that the deformation theory of Lie algebras can be elegantly formulated using the language of differential graded Lie algebras (dgLa), which naturally encode both the deformation problem and its obstructions, see \cite{Nijenhuis-Richardson-Cohomology-and-Deformations-of-algebraic-structures-64',Nijenhuis-Richardson-cohomology-and-deformations-in-graded-Lie-algebras-66',Nijenhuis-Composition-Systems-and-Deformations-of-subalgebras-68',Nijenhuis-Richardson-Commutative-algebra-cohomology-and-deformations-68',Nijenhuis-Richardson-Deformations-of-Lie-algebra-structures-67',Nijenhuis-Richardson-Deformations-of-homomorphisms-of-Lie-groups-and-Lie-algebras-67',Richardson-A-Rigidity-theorem-for-subalgebras-of-Lie-and-associative-algebras-67',Richardson-Deformations-of-subalgebras-69',Richardson-Stanley-Stable-subalgebras-of-Lie-and-associative-algebras-67',Richardson-on-the-rigidity-of-semidirect-product-of-Lie-algebras-67'}. Despite these advancements, surprisingly enough,  the deformation theory of Lie ideals seems to be, so far, almost totally
	absent in the literature. 
	While the deformation theory of Lie subalgebras has been partially addressed in certain contexts, the specific case of ideals — a central and invariant feature of Lie algebras —remains largely unexplored. This gap in the literature is striking, given that the behavior of ideals under deformations can influence both the internal structure of the algebra and its external representations.

	This paper addresses this gap by associating to any ideal in a Lie algebra a dgLa that controls its deformations. Furthermore, it introduces an $L_\infty$-algebra that governs the simultaneous deformations of the Lie bracket in the ambient Lie algebra and the ideal itself. These constructions provide a new perspective on the deformation theory of Lie algebras, with significant implications for understanding rigidity, stability, and obstruction phenomena.
	
	\medskip

	A well-known principle in deformation theory states, roughly speaking, that behind
	every reasonable deformation problem of a mathematical structure of a specific type,
	there is a differential graded Lie algebra (dgLa for short) or, more generally, an $L_\infty$-algebra,
	which ``controls" the deformation problem --  in the sense that its Maurer-Cartan elements are in bijective correspondence with the deformations of the initial structure. This
	philosophy can be traced back to Deligne, Drinfeld, Kontsevich and many others \cite{Deligne-letter-to-Millson-and-Goldman-1987,Drinfeld-A-letter-from-Kharkov-to-Moscow-about-deformation-theory-2014-original-1988,Goldman-Millson-homotopy-invariance-of-Kuranishi-space-they-mention-the-principle-of-defor-theory-1990,Kontsevich-Soibelman-Deformation-theory-notes-2002,Hinich-dg-coalgebras-as-formal-stacks-2001,Manetti-dgLas-and-formal-deformation-theory-2009}. Around fifteen years ago, Lurie \cite{Lurie-derived-algebraic-geometry-formal-moduli-problems-2011} and
	Pridham \cite{Pridham-Unifying-derived-deformation-theories-2010} formalized this heuristic philosophy of deformation theory into an equivalence between formal moduli problems and differential graded Lie algebras in characteristic
	zero, using higher category theory.

	Returning to the main focus of this paper; given a Lie algebra
	$\got g$ and a Lie ideal $\got i \lhd\got g$, the following question arises where the adjective ``controlling" is
	meant in the aforementioned sense.
	
	\begin{center}
		\textbf{Question 1: What is the controlling differential graded Lie algebra of the deformation problem of a Lie ideal $\got i\lhd \got g$?}
	\end{center}
	
	This paper answers the above question by describing explicitly the deformation cochain complex
	together with its graded Lie algebra structure and proving that it controls the ``small" deformations of the Lie ideal $\got i\lhd \got g$.
	Moduli theory deals with the study of the geometry of moduli spaces, which are spaces
	whose points represent equivalence classes of algebro-geometric objects. To the best of the
	knowledge of the authors, the prototypical and perhaps the most elegant example of a moduli space is
	the (real) Grassmannian of a vector space. Given an $n$-dimensional vector space $V$, the set of isomorphism
	classes of $k$-dimensional vector subspaces, denoted $\Gr_k(V)$, carries a natural smooth manifold structure of dimension equal to $k(n-k)$ and is called the \emph{$k$-Grassmannian of $V$}. Except
	for this beautiful example, in almost all other cases it is difficult to understand (at least immediately) the global geometry of a moduli space, due to the presence of singularities and/or of
	infinite dimensionality. Hence the infinitesimal nature of a moduli space must first
	and foremost be understood.  

	Some of the questions in deformation theory
	concern when a mathematical structure of interest  is rigid or stable under deformations. Roughly speaking, rigidity questions are related to identifying isolated points of the
	moduli space of the studied structure, while stability questions concern its local smoothness around the distinguished point.
	Given a Lie algebra $\got g$, the space $\mbf I_k(\got g)$ of $k$-dimensional Lie ideals is a
	subspace of $\Gr_k(\got g)$. A Lie ideal $\got i\lhd \got g$ is \emph{rigid under the natural action of $\Aut(\got g)$ on $\mbf I_k(\got g)$}
	if the space of $k$-dimensional Lie ideals $\mbf I_k(\got g)$ coincides locally, in some open neighborhood
	of $\got i \in \mbf I_k(\got g)\subseteq \Gr_k(\got g)$, with the $\Aut(\got g)$-orbit of $\got i$.
	
	\begin{center}
		\textbf{
			Question 2: Under what assumption is a Lie ideal $\got i\lhd \got g$ $\Aut(\got g)$-rigid?
		}
	\end{center}
	
	An ideal $\got i\lhd \got g$ in a Lie algebra $(\got g, \mu_{\got{g}})$  is called \textbf{stable} if for any Lie bracket $\mu'$ on $\got{g}$ sufficiently close to $\mu_{\got{g}}$,
	there exists a Lie ideal $\got i'\lhd (\got g, \mu')$ sufficiently close to $\got i\in \Gr_k(\got g)$
	
	\begin{center}
		\textbf{
			Question 3: Under what assumption is a Lie ideal $\got i\lhd \got g$ stable?
		}
	\end{center}
	
	This paper provides sufficient criteria for both the rigidity and
	stability  of a Lie ideal $\got i$ in a Lie algebra $\got g$.
	By developing deformation cohomologies tailored to ideals, the authors not only establish theoretical foundations but also explore geometric applications, including the use of the Kuranishi map to study obstructions to deformations of ideals in Lie algebras. This work thus represents a novel contribution to the broader landscape of deformation theory in the context of Lie algebras and opens up new avenues for the long-standing, active research on their classification theory.

	\subsection*{Outline of the paper}
	Section \ref{prelim}
	introduces the mathematical framework necessary for the studies in this paper, including graded vector spaces, differential graded Lie algebras, and $L_\infty$-algebras. These structures form the algebraic foundation for the deformation theories discussed here.
	Then standard deformation complexes associated with Lie algebras are reviewed, such as the Chevalley-Eilenberg complex, and the notation and conventions used throughout the paper are set up.
	Section \ref{MC_section} in particular
	demonstrates how Lie algebras and Lie algebras with representations can be understood as Maurer-Cartan elements of special differential graded Lie algebras. This establishes the framework for associating dgLa structures to deformations.
	
	Section \ref{def_subalgebras_section}
	recalls existing results on the deformation theory of Lie subalgebras, including their description as Maurer-Cartan elements in dgLa frameworks. It lays the groundwork for extending these ideas to the case of Lie ideals.
	Section \ref{def_coh_ideal1} then defines smooth and infinitesimal deformations of ideals in Lie algebras. It provides examples and clarifies their distinctions from subalgebra deformations.
	Then it explores the connections between the deformation cohomology of an  ideal and those for the same ideal as a subalgebra, as well as those of morphisms associated to ideals in Lie algebras, highlighting novel differences and insights.

	Section \ref{voronov_ideals} 
	constructs a Voronov dataset associated to an ideal in a Lie algebra and uses it to define the unique, up to isomorphism, dgL[1]a that controls the deformations of the ideal. It then develops an $L_\infty[1]$-algebra that governs the simultaneous deformations of the Lie bracket and the ideal, providing a unified framework for understanding these processes.
	
	Section \ref{geometric_results_ideals}  introduces the Kuranishi map as a tool for identifying obstructions to deformations. It discusses its role in understanding when deformations can be extended or are blocked,
	and it examines then (cohomological) conditions under which ideals remain rigid under perturbations, before also establishing criteria for their stability.
	
	Appendix \ref{appendix_G_k} provides necessary supplementary material on Grassmannian manifolds for the considerations in this paper.

	\subsection*{Acknowledgement}
	This paper is included in the PhD thesis of the first author. The authors warmly thank Kalin Krishna, Stefano Ronchi, Ivan Struchiner, Marco Zambon, and Chenchang Zhu for interesting discussions, and in particular 
	Karandeep Jandu Singh for many useful comments on an early version of this paper, for interesting discussions and for pointing out to them that Theorem \ref{theorem stability ideals} follows as well from his result \cite[Theorem 3.20]{Singh25} with a weaker assumption, see Remark \ref{theorem stability ideals rem}.
	The first author thanks in particular Miquel Cueca for numerous fruitful and useful discussions during his PhD studies, and Luca Vitagliano for a very stimulating research stay, as well as many enlightening comments and discussions. Finally, the authors thank warmly the anonymous referee for very insightful, useful, and detailed suggestions.

	\section{Preliminaries}\label{prelim}
	This section collects necessary preliminaries for the contents of this paper.
	\subsection{Graded vector spaces and L$_\infty$-algebras}
	L$_\infty$-algebras are used to describe deformations of algebraic structures. Their definition and properties are summarized in this section.

	A \textbf{graded vector space} $\mathbf{V}$ is the 
	direct sum of a family of vector spaces $(V_i \mid i\in\mathbb Z)$, that comes equipped as follows with a grading.
	An element $v\in\mbf{V}$ is called \textbf{degree-homogeneous} if $v\in V_i$ for some $i\in\mbb{Z}$ and the \textbf{degree} of $v$ is then defined to be $|v|=i$. That is, for
	$i\in\mathbb Z$, the degree $i$ component of $\mathbf V$ (denoted with lower
	index $\mathbf V_i$) equals $V_{i}$. The component $V_i$ of $\mathbf V$ is then written $V_i[-i]$ for recording its degree, i.e.
	\[\mathbf{V}=\bigoplus_{i\in\mbb{Z}}V_i[-i].\] 
	That is, $V_i$, which as a classical vector space has elements of degree $0$, is shifted by $-i$ in the following sense.
	For $k\in\mbb{Z}$, the \textbf{degree $k$-shift} $\mbf{V}[k]$ of $\mathbf V$ is the graded vector space 
	\[\mathbf{V}[k]=\bigoplus_{i\in\mbb{Z}}V_i[-i+k]=\bigoplus_{j\in\mbb{Z}}V_{j+k}[-j],
	\]
	i.e.~with 
	$(\mbf{V}[k])_j=V_{j+k}$ for all $j\in\mathbb Z$. Unless specified otherwise, the vector space $\mathbf V$ has finite dimension. That is, only finitely many of its summands are non-trivial and have then finite dimension.

	The usual constructions with vector spaces can be similarly done in the graded setting.
	Let $\mbf{V}$ and $\mbf{W}$ be two graded vector spaces.
	\begin{enumerate}
		\item The direct sum $\mbf{V}\oplus\mbf{W}$ is given by $(\mbf{V}\oplus\mbf{W})_i=V_i\oplus W_i$ for all $i$, i.e. \[\mathbf{V}\oplus\mathbf W=\bigoplus_{i\in\mbb{Z}}(V_i+W_i)[-i].\] 
		\item The tensor product $\mbf{V}\otimes\mbf{W}$ is graded by $(\mbf{V}\otimes\mbf{W})_i=\oplus_{j+k=i}\ V_j\otimes W_k$, i.e.
		\[\mathbf{V}\otimes\mathbf W=\bigoplus_{i\in\mbb{Z}}\left(\bigoplus_{j+k=i}V_j\otimes W_k\right)[-i].\] 
		
		\item The dual $\mbf{V}^*$ of $\mathbf V$ is the graded vector space 
		\[ \mbf{V}^*=\bigoplus_{i\in\mbb{Z}}V_i^*[i].
		\]
		\item The graded vector space  $\textbf{Hom}(\mbf{V},\mbf{W})\simeq \mbf{V}^*\otimes\mbf{W}$ is then defined by \[\textbf{Hom}(\mbf{V},\mbf{W})=\oplus_{i\in\mathbb  Z}\left(\oplus_{j\in\mbb{Z}}\Hom(V_j,W_{i+j})\right)[-i]\] and, as usual, denoted by $\textbf{End}(\mbf{V})\simeq \mbf V^*\otimes\mbf V$ when $\mbf V=\mbf W$.
		
		In particular a (degree 0) \textbf{linear map} $f\colon \mbf{V}\to\mbf{W}$ between graded vector spaces is a collection of degree-preserving linear maps $\{f_i\colon V_i\to W_i\}$. A \textbf{degree $k$ linear map} from $\textbf V$ to $\textbf W$ is a linear map $f\colon \mbf{V}\to\mbf{W}[k]$ i.e.~$f$ is a collection of linear maps $f_i\colon V_i\to W_{i+k}$ for all $i\in\mathbb Z$ (with, by definition, all but finitely many of these maps being $0$).

		\item The tensor algebra $\mbf{T}(\mbf{V})=\oplus_{l\geq 0}\mbf{V}^{\otimes l}$ is graded by the total degree $$|v_1\otimes\cdots\otimes v_l|=|v_1|+\cdots|v_l|$$ 
		since according to the considerations above
		\[ (\mbf{T}(\mbf{V}))_i=\oplus_{l\geq 0}(\mbf{V}^{\otimes l})_i=\oplus_{l\geq 0} \oplus_{i_1+\ldots+i_l=i}V_{i_1}\otimes\ldots\otimes V_{i_l}
		\]
		for all $i\in\mathbb Z$. Here, by convention $\mbf V^{\otimes 0}=\mathbb R$ has degree $0$.
		The \textbf{graded symmetric algebra} $\mbf{S}(\mbf{V})$  of $\mbf V$ is the quotient of $\mbf{T}(\mbf{\mbf{V}})$ by the two-sided ideal generated by elements of the form $x\otimes y-(-1)^{|x||y|}y\otimes x$.  Similarly, the \textbf{graded exterior algebra} ${\bigwedge}(\mbf{V})$ of $\mbf V$ is the quotient of $\mbf T(\mbf V)$ by the two-sided ideal generated by elements of the form $x\otimes y+(-1)^{|x||y|}y\otimes x$.
	\end{enumerate}
	
	\begin{rem}
		\begin{enumerate}
			\item The graded vector spaces $\mbf{T}(\mbf{V})$, ${\bigwedge}(\mbf{V})$ and $\mbf{S}(\mbf{V})$ do not have finite dimension in general.

			\item		The graded tensor algebra in the last construction is in fact bigraded, since an element 
			$ x\in V_{i_1}\otimes\ldots\otimes V_{i_l}$ has  as well the \emph{polynomial degree} $l$. Its bidegree is hence $(l,i_1+\ldots+i_l)$.
			Precisely, the elements of $\mbf T^k\mbf V:=\mbf V^{\otimes k}$ have polynomial degree $k$.
			Analogously, the elements of $\mbf S^k(\mbf V)$ and $\bigwedge^k(\mbf V)$ have polynomial degree $k$.

		\end{enumerate}
	\end{rem}
	\bigskip

	A \textbf{graded Lie algebra} $(\mbf{V},[\cdot,\cdot])$ is a graded vector space $\mbf{V}$ equipped with a $\mathbb R$-bilinear bracket $[\cdot,\cdot]\colon \mbf{V}\otimes\mbf{V}\to\mbf{V}$ satisfying the following conditions:
	\begin{enumerate}
		\item the bracket is degree-preserving: $[V_i,V_j]\subset V_{i+j}$ for $i,j\in\mathbb Z$, 
	\end{enumerate}
	and 
	\begin{enumerate}\setcounter{enumi}{1}
		\item the bracket is \textbf{graded skew-symmetric}: $[x,y]=-(-1)^{|x||y|}[y,x]$ for all degree-homogeneous elements $x,y\in \mbf V$,
		\item the bracket satisfies the \textbf{graded Jacobi identity} $$[x,[y,z]]=[[x,y],z]+(-1)^{|x||y|}[y,[x,z]]$$
	\end{enumerate}
	for degree-homogeneous elements $x,y,z\in\mbf{V}$.
	\medskip

	A \textbf{differential graded Lie algebra (dgLa)} is then a triple $(\mbf{V},[\cdot,\cdot],\dr)$ where $(\mbf{V},[\cdot,\cdot])$ is a graded Lie algebra and $\dr\colon \mbf{V}\to\mbf{V}$ is a degree 1 linear map such that
	\begin{enumerate}
		\item $\dr[x,y]=[\dr(x),y]+(-1)^{|x|}[x,\dr(y)]$ (that is, $\dr$ is a degree $1$ derivation with respect to the bracket $[\cdot,\cdot]$)
		\item $\dr^2=0$.
	\end{enumerate}

	\bigskip

	For $n\in\mathbb N$ and $0\leq i\leq n$ a  permutation $\sigma\in S_n$ is called an \textbf{$(i,n-i)$-unshuffle} if it satisfies $\sigma(1)<\cdots<\sigma(i)$ and $\sigma(i+1)<\cdots<\sigma(n).$ The set of $(i,n-i)$-unshuffles is denoted by $S_{(i,n-i)}$. 
	
	Consider a graded vector space $\mbf V$ as above and two of its degree homogeneous elements $x_1$ and $x_2$. Then as elements of $\mbf S(\mbf V)$,
	$x_1$ and $x_2$ satisfy
	\[ x_1\cdot x_2=(-1)^{|x_1|\cdot|x_2|}x_2\cdot x_1=:\epsilon( (12); x_1,x_2)\cdot x_2\cdot x_1
	\]
	i.e.~with $\epsilon( (12); x_1,x_2)\in\{-1,1\}$ defined by this equation.
	More generally for $x_1,\ldots, x_n\in \mbf V$ degree-homogeneous, the  \textbf{Koszul sign}  $\epsilon(\sigma; x_1,\ldots, x_n)\in \{-1, 1\}$ of a permutation $\sigma$ and $x_1,\ldots, x_n$ is defined by 
	\[ x_1\cdot \ldots\cdot x_n=\epsilon(\sigma; x_1,\ldots, x_n)\cdot x_{\sigma(1)}\cdot \ldots\cdot x_{\sigma(n)}.
	\]
	\begin{defin}
		\begin{enumerate}
			\item		An \textbf{$L_{\infty}[1]$-algebra} is a graded vector space $\mbf{V}$  equipped with a
			collection \[\left\{m_k\colon S^k\mbf{V}\to\mbf{V}[1]\right\}_{k\geq 1}\] of linear maps, satisfying the following relations for all homogeneous elements $x_1,\dots,x_n\in\mbf{V}$:
			\begin{equation}\label{eq_L_infty_1}
				\sum_{i+j=n+1}\sum_{\sigma\in S_{(i,n-i)}}\epsilon(\sigma; x_1,\ldots, x_n)m_j(m_i(x_{\sigma(1)},\dots,x_{\sigma(i)}), x_{\sigma(i+1)},\dots, x_{\sigma(n)})=0.
			\end{equation}
			\item A \textbf{dgL$[1]$a} is an $L_\infty[1]$-algebra $(\mbf{V}, m_1, m_2)$, i.e.~with $m_k=0$ for all $k\geq 3$.
			\item The notion of \textbf{curved $L_{\infty}[1]$-algebra} generalizes the notion of $L_{\infty}[1]$-algebra by allowing the presence of an element $m_0\in V_1[-1]$, and so permitting the indices $i,j$ in (\ref{eq_L_infty_1}) to take the value zero. Hence, an extra relation arises, corresponding to the case $n=0$ in (\ref{eq_L_infty_1}).
		\end{enumerate}
	\end{defin}
	
	\begin{defin}
		\begin{enumerate}
			\item		A \textbf{Maurer-Cartan element} of a dgLa $(\mbf{V},[\cdot,\cdot],\dr)$ is an element $x\in\mbf V$ of degree $1$ satisfying the following equation:
			\begin{equation}\label{Maurer-Cartan equation defin}
				\dr(x)+\frac{1}{2}[x,x]=0.
			\end{equation}
			This equation \eqref{Maurer-Cartan equation defin} is called the  \textbf{Maurer-Cartan equation}.
			\item More generally, a \textbf{Maurer-Cartan element} of a $L_\infty[1]$-algebra $(\mbf V, m_1, m_2,\ldots, m_l)$ with finitely many multibrackets is an element $x$ of degree $0$ in $\mbf V$ such that
			\begin{equation}\label{MC-eq for L-infty}
				\sum_{k=1}^l\frac{1}{k!} m_k(x, \ldots, x)=0.
			\end{equation}
			If $\bigleftpar\mbf{V},m_0,m_1,\dots,m_l\bigrightpar$ is a curved $L_{\infty}[1]$-algebra, then the above sum starts from $k=0$.
		\end{enumerate}
	\end{defin}
	\begin{defin}\label{iso_L_infty}\cite{Schaetz09,BrZa22,Manetti-book-Lie-methods-in-deformation-theory-22',KrSc24}
		Let $\mbf V$ and $\mbf W$ be two graded vector spaces equipped with
		$L_\infty[1]$-algebra structures $(m_k)_{k \geq 1}$ and
		$(n_k)_{k \geq 1}$ respectively.
		
		A \textbf{morphism of $L_\infty[1]$-algebras} from
		\[
		(\mbf V,(m_k)_{k\geq 1})
		\]
		to
		\[
		(\mbf W,(n_k)_{k\geq 1})
		\]
		is a family of morphisms
		\[
		F := \bigl(f_k \colon S^{k}\mbf V \to \mbf W\bigr)_{k \geq 1}
		\]
		of graded vector spaces, 
		such that the two families of morphisms
		\[
		S^{k}\mbf V \to \mbf W[1]
		\]
		given by
		\[
		\sum_{r+s=k}
		\sum_{\tau \in S_{(r,s)}}
		\epsilon(\tau;x_1,\ldots, x_k)
		f_{s+1}\Bigl(
		m_r(x_{\tau(1)}, \cdots, x_{\tau(r)}),
		x_{\tau(r+1)}, \cdots, x_{\tau(k)}
		\Bigr)
		\]
		and
		\[
		\sum_{l=1}^k
		\sum_{j_1+\cdots+j_l=k}
		\sum_{\tau \in S_k}
		\frac{\epsilon(\tau;x_1,\ldots, x_k)}
		{l!\, j_1! \cdots j_l!}
		\,
		n_l\Bigl(
		f_{j_1}(x_{\tau(1)},\cdots, x_{\tau(j_1)}), \cdots,
		f_{j_l}(x_{\tau(k-(j_1+\cdots+j_{l-1}))}, \cdots,  x_{\tau(k)})
		\Bigr)
		\]
		on homogeneous elements $x_1,\dots,x_k \in \mbf V$ coincide.
		
		\medskip
		
		The morphism $F= \bigl(f_k \colon S^{k}\mbf V \to \mbf W\bigr)_{k \geq 1}$ of $L_\infty[1]$-algebras is an \textbf{isomorphism} of $L_\infty[1]$-algebras if $f_1$ is an isomorphism of vector spaces.
	\end{defin}
	
	\begin{ex}
		A dgLa structure $([\cdot,\cdot],\dr)$ on a graded vector space $\mbf V$ becomes a dgL$[1]$a structure $(m_1,m_2)$ on $\mbf V[1]$ by setting  $m_1=-\dr$, $m_k=0$ for $k>2$, and by defining $m_2$ by 
		\[m_2(x_1,x_2)=(-1)^{|x_1|}[x_1,x_2]
		\] for degree-homogeneous elements $x_1,x_2\in \mbf V$, where $|x_1|$ is the degree of $x_1$ as an element of $\mbf V$.
	\end{ex}

	\begin{ex}(Lie 2-algebra \cite{Baez-Crans-Lie2algebras-04'})\label{defin of a Lie 2-algebra}
		A \textbf{2-term $L_\infty[1]$-algebra} is a $L_\infty[1]$-algebra defined on 
		a graded\footnote{The peculiar choice of grading versus indices of the summands becomes clear at the end of this example.} vector bundle $\mathbf V= \got g_0[1]\oplus \got g_{-1}[2]$.
		This amounts to a 2-term cochain complex of vector spaces $\got{g}_{-1}\stackrel{m_1}{\longrightarrow}\got{g}_0$ together with
		\begin{itemize}
			\item a skew-symmetric bilinear map $m_2^{0}\colon \wedge^2\got{g}_0\to\got{g}_0$ (this map is also written $m_2^0=[\cdot,\cdot]$)
			\item a bilinear map $m_2^1\colon \got{g}_0\otimes\got{g}_{-1}\to\got{g}_{-1}$ (this map is also written $m_2^1=\nabla$)
			\item an alternating trilinear map $m_3\colon \wedge^3\got{g}_{0}\to\got{g}_{-1}$
		\end{itemize}
		satisfying, for all $x,x_1,x_2,x_3,x_4\in\got{g}_0$ and $y,y_1,y_2\in\got{g}_{-1}$, the following conditions:
		\begin{enumerate}[labelindent=\parindent,leftmargin=1.18cm,label=(\roman*)]
			\item $m_1(\nabla_{x}y)=[x,m_1(y)]$
			\item $\nabla_{m_1(y_1)}y_2+\nabla_{m_1(y_2)}y_1=0$
			\item $[x_1,[x_2,x_3]]+[x_2,[x_3,x_1]]+[x_3,[x_1,x_2]]=m_1m_3(x_1,x_2,x_3)$
			\item $\nabla_{x_1}\nabla_{x_2}y-\nabla_{x_2}\nabla_{x_1}y-\nabla_{[x_1,x_2]}y=m_3(x_1,x_2,m_1(y))$
			\item and 
			\begin{equation*}
				\begin{split}
					0=&\sum_{i=1}^4(-1)^{i+1}\nabla_{x_i}(m_3(x_1, \ldots, \widehat{x_i}, \ldots, x_4))+\sum_{i<j}(-1)^{i+j}m_3([x_i,x_j],x_1,\ldots, \widehat{x_i}, \ldots, \widehat{x_j}, \ldots, x_4).
				\end{split}
			\end{equation*}
		\end{enumerate}
		In other words, $ \got g_0[0]\oplus \got g_{-1}[1]$ with $(m_1, m_2=[\cdot,\cdot]+\nabla, m_3)$ is a \textbf{Lie 2-algebra}. The maps $m_i$, $i=1,2,3$, are usually written $l_i$ in this context.
	\end{ex}
	
	\begin{rem}
		\begin{enumerate}
			\item In general, a $L_\infty[1]$-algebra becomes a \textbf{$L_\infty$-algebra} \cite{Lada-Markl-StronglyHomotopy-95', Lada-Stasheff-Intro-Linftyalgebras-physicists-93'} when shifted appropriately by $-1$, see \cite{VoronovHigherDerivedBracketsAndHomotopyAlgebras, Manetti-Fiorenza-LinftyAlgebrasOnMappingCones-07'}. Because the setting of $L_\infty$-algebras is not used in this paper, only $L_\infty[1]$-algebras are defined and considered.
			\item	A Lie 2-algebra with $l_3=0$ is called a \textbf{strict Lie 2-algebra}. Then by the equations above, $[\cdot,\cdot]$ is a Lie bracket on $\mathfrak g_0$ and $\nabla$ is a representation of $\mathfrak g_0$ on $\mathfrak g_{-1}$. (ii) above defines a skew-symmetric map $[\cdot, \cdot]_{\mathfrak g_{-1}}\colon \wedge^2\mathfrak g_{-1}\to\mathfrak g_{-1}$, $[y_1,y_2]_{\mathfrak g_{-1}}=\nabla_{l_1(y_1)}y_2$ which,  by (iii),  satisfies the Jacobi identity. (i) shows that $l_1\colon \mathfrak g_{-1}\to\mathfrak g_0$ is a morphism of Lie algebras.
			(i) and (iii) then imply as well together that $\nabla\colon \mathfrak g_0\to \got{aut}(\got{g}_{-1})$, i.e.~$\nabla$ is a representation of $\mathfrak g_0$ by derivations of $\mathfrak g_{-1}$. Hence $\left(\mathfrak g_0, \mathfrak g_{-1}, l_1, \nabla\right)$ is a \textbf{crossed-module of Lie algebras} in the following sense.

			\item	A \textbf{crossed module of Lie algebras} is a pair of Lie algebras $(\got{g}_{-1},\got{g}_0)$ together with a Lie algebra morphism $\phi\colon \got{g}_{-1}\to\got{g}_0$ and a Lie algebra action by derivations $\psi\colon\got{g}_0\to\got{aut}(\got{g}_{-1})$ satisfying, for any $x\in\got{g}_0$ and $y_1,y_2\in\got{g}_{-1}$, the following two conditions:
			\begin{equation*}
				\phi(\psi_{x}(y_1))=[x,\phi(y_1)]_{\got{g}_{0}} \quad \text{and} \quad \psi_{\phi(y_1)}(y_2)=[y_1,y_2]_{\got{g}_{-1}}.
			\end{equation*}
		\end{enumerate}
		The above considerations establish an equivalence between crossed-modules of Lie algebras and strict Lie 2-algebras.
		
	\end{rem}
	
	\begin{ex}
		Let $\got g$ be a Lie algebra and let $\got i$ be an ideal in $\got g$. Then $\got g[0]\oplus \got i[-1]$ becomes a strict Lie 2-algebra with the inclusion $l_1\colon\got i\hookrightarrow \got g$, the Lie bracket $l_2^0=[\cdot\,,\cdot]$ on $\got g$ and the restriction to $\got i$ of the adjoint representation of $\got g$:  $l_2^1=\ad\colon \got g\otimes\got i\to \got i$, $l_2^1(x\otimes y)=[x,y]\in\got i$ for $x\in\got g$ and $y\in \got i$.
	\end{ex}

	\bigskip
	Voronov introduced in \cite{VoronovHigherDerivedBracketsAndHomotopyAlgebras} a construction of $L_\infty[1]$-algebras, see also \cite{Zambon-Fregier-SimultaneousDefOfAlgebraAndMorphisms-15', Fregier-Zambon-SimultaneousApplicationsInPoissonGeometry-2015} for the following approach to it.
	\begin{defin}
		A quadruple $(L,\got{a},P,\Theta)$ where
		\begin{enumerate}
			\item $L=\oplus_{i\in\mbb{Z}}L_i[-i]$ is a graded Lie algebra with Lie bracket $[\cdot,\cdot]$,
			\item $\got{a}$ is an abelian graded Lie subalgebra of $L$,
			\item $P\colon L\to\mathfrak{a}$ is a linear projection such that $\Ker P$ is a graded Lie subalgebra,
			\item $\Theta$ is an element of $L_1[-1]$ (i.e.~of degree $1$) such that $\Theta\in\Ker P$ and $[\Theta,\Theta]=0$
		\end{enumerate}
		is called here a \textbf{Voronov-dataset}.
		
		When $\Theta\in L_1[-1]$ is not assumed to live inside $\Ker P$, but only to commute with itself, i.e. $[\Theta,\Theta]=0$, the quadruple $(L,\got{a},P,\Theta)$ is called a \textbf{curved Voronov dataset}.
	\end{defin}
	
	Voronov proves in \cite{VoronovHigherDerivedBracketsAndHomotopyAlgebras, VoronovHigherDerivedBracketsForArbitraryDerivations-05'} the following two theorems. 
	\begin{thm}[\cite{VoronovHigherDerivedBracketsAndHomotopyAlgebras}]\label{theorem Voronov for L-infty algebra using higher derived brackets}
		Let $(L,\mathfrak{a},P,\Theta)$ be a Voronov-dataset.  The multibrackets $m_k\colon S^k\mathfrak{a}\to\mathfrak{a}[1]$ given by
		\begin{equation*}
			m_k(a_1, \ldots,  a_k)=P[[\ldots[[\Theta,a_1],a_2],\ldots], a_k],
		\end{equation*}
		for all $a_1,\ldots, a_k\in \got a$ determine an $L_\infty[1]$-algebra structure on $\mathfrak{a}$.
	\end{thm}

	Voronov's version of the following theorem in \cite{VoronovHigherDerivedBracketsForArbitraryDerivations-05'} is more general since it considers general derivations of the graded Lie algebra. The following version is stated in \cite{Zambon-Fregier-SimultaneousDefOfAlgebraAndMorphisms-15'} and focuses on inner derivations.
	\begin{thm}[\cite{VoronovHigherDerivedBracketsForArbitraryDerivations-05', Zambon-Fregier-SimultaneousDefOfAlgebraAndMorphisms-15'}]\label{theorem Voronov for L-infty algebra on the mapping cone}
		Given a Voronov-dataset $(L,\got{a},P,\Theta)$, set $D:=[\Theta,\cdot]$. The following brackets induce an $L_{\infty}[1]$-algebra structure on the space $L[1]\oplus\got{a}$:
		\begin{enumerate}
			\item The unary bracket is given by $m_1(x,a)=(-(Dx), P(x+Da))$ for $x\in L$ and $a\in \got a$.
			\item The binary bracket is given by\footnote{Here, $|x|$ is the degree of $x$ as an element of $L$, and so its degree as an element of $L[1]$ is $|x|-1$.}  $m_2(x,y)=(-1)^{|x|}[x,y]$ for $x,y\in L$.
			\item For $k\geq 2$ the $k$-ary bracket is given by \[ m_k(x,a_1,\dots,a_{k-1})=P[\dots[x,a_1],\dots,a_{k-1}]\]
			and 
			\[m_k(a_1,\dots,a_k)=P[\dots[Da_1,a_2],\dots,a_k]\]
			for $x\in L$ and $a_1,\dots,a_k\in\got{a}$, and it vanishes on any different combination of elements of $L[1]\oplus \got a$.
		\end{enumerate}
		
	\end{thm}

	\subsection{Deformation complexes}
	This section only considers finite-dimensional real Lie algebras.
	Given a Lie algebra $(\got{g},\mu\eqdef[\cdot,\cdot])$ and a representation $r\colon \got{g}\to\got{gl}(V)$ of $\got{g}$ on some vector space $V$, there is an associated complex $C^\bullet_r(\got{g}; V)\coloneqq \wedge^\bullet\got{g}^*\otimes V$ called the \textbf{Chevalley-Eilenberg complex with coefficients in $V$}. Its differential $\delta_{\got g}^r\colon C^\bullet (\got{g}; V)\to C^{\bullet+1}(\got{g}; V)$ is defined for all $k\geq0$ and $\omega\in C^k(\mathfrak g; V)$ by
	\begin{align}\label{Chevalley-Eilenberg differential formula with values in a representation}
		\delta_{\got g}^r\omega(x_1,\dots, x_{k+1})\coloneqq&\sum_{i=1}^{k+1} (-1)^{i+1}r(x_i)\left(\omega\left(x_1,\dots,\widehat{x_i},\dots, x_{k+1}\right)\right)\nonumber\\
		&+\sum_{i<j}(-1)^{i+j}\omega\left([x_i,x_j],x_1,\dots,\widehat{x_i},\dots,\widehat{x_j},\dots,x_{k+1}\right).
	\end{align}
	The associated cohomology is denoted by $H_r^\bullet(\got{g}; V)$.
	\begin{rem}\label{remark_functor_g-modules_cochain_complexes}
		The assignment
		\begin{equation*}
			\{\got{g}\text{-modules}\} \to \{\text{cochain complexes}\}, \qquad (V,r) \mto \bigleftpar C^{\bullet}(\got{g};V), \delta_{\got{g}}^r\bigrightpar,
		\end{equation*}
		extends to a functor from the category of $\got{g}$-modules to the category of cochain complexes: it sends a $\got{g}$-module morphism $\phi\colon (V,r_V)\to(W,r_W)$\footnote{I.e.~$\phi\colon V\to W$ is a linear map such that $\phi\circ r_V=r_W\circ\phi$.} to $\phi_*\colon\bigleftpar C^\bullet(\got{g}; V), \delta_{\got{g}}^{r_V}\bigrightpar\to\bigleftpar C^{\bullet}(\got{g}; W), \delta_{\got{g}}^{r_W}\bigrightpar$ defined by $\phi_*(\omega)=\phi\circ\omega$ for $\omega\in C^{\bullet}(\got{g}; V)$. It is easy to check that this functor is exact, i.e.~that it sends short exact sequences to short exact sequences.
	\end{rem}
	
	Section \ref{def_subalgebras_section} illustrates how strongly the complexes in the following list are linked with deformation theories in the context of Lie algebras (see \cite{Crainic-Ivan-Schaetz} and references therein).
	\begin{enumerate}
		\item The \textbf{deformation complex of a Lie algebra $\got{g}$} is the Chevalley-Eilenberg complex $C^\bullet_{\rm ad}(\got{g};\got{g}):=\wedge^\bullet\got{g}^*\otimes\got{g}$ with the representation $\ad\colon \got{g}\to\got{gl}(\got{g})$. The corresponding differential is accordingly denoted by $\delta^{\ad}_{\got g}$ 		and its cohomology is written $H^\bullet_{\rm ad}(\got{g};\got{g})$.
		\item The \textbf{deformation complex of a Lie algebra morphism $\phi\colon\got{g}\to\got{h}$} is given by the Chevalley-Eilenberg complex $C^\bullet_{\rm def}(\phi):=\wedge^\bullet\got{g}^*\otimes\got{h}$ with the representation $\phi^*\ad^{\got{h}}\colon \got{g}\to\got{gl}(\got{h})$. The corresponding differential is here denoted by $\delta_{\phi}$ for simplicity and its cohomology by $H^\bullet_\phi(\got{g};\got{h})$.
		\item The \textbf{deformation complex of a Lie subalgebra $\got{h}\subset\got{g}$} is given by the Chevalley-Eilenberg complex $C^\bullet_{\rm def}(\got{h}\subset\got{g}):=\wedge^\bullet\got{h}^*\otimes\got{g}/\got{h}$ with the \textbf{Bott representation}
		\begin{equation*}
			{\rm ad}^{\text{Bott}}\colon \got{h}\to\got{gl}(\got{g}/\got{h}), \quad {\rm ad}^{\text{Bott}}_u(\bar{y})=\overline{[u,y]},
		\end{equation*}
		where $\overline{(\cdot)}$ denotes the image under the canonical projection $\pi_{\got{g}/\got{h}}\colon \got{g}\to\got{g}/\got{h}, \ x\mto\pi_{\got{g}/\got{h}}(x)\eqqcolon\overline{x}$. The corresponding differential is denoted by $\delta^{\rm Bott}_{\got{h}}$ and its cohomology is written $H^\bullet_{\rm Bott}(\got{h};\got{g}/\got{h})=:H^\bullet_{\rm def}(\got h\subset \got g)$.
	\end{enumerate}
	
	The following result is an inspiration for similar results on the relations between the different cohomologies associated to ideal in Lie algebras, see Section \ref{rel_other_coh}.
	\begin{rem}
		\begin{enumerate}
			\item Let $\got g$ be a Lie algebra and let $\got h\subset \got g$ be a Lie subalgebra.
			The monomorphism $\iota\colon \got{h}\hookrightarrow\got{g}$ of Lie algebras induces the following short exact sequence of cochain complexes
			\begin{equation*}
				\begin{tikzcd}
					C^\bullet_{\rm ad}(\got{h};\got{h})\stackrel{\iota_*}{\hookrightarrow} C^\bullet_{\rm def}(\iota)\stackrel{(\pi_{\got{g}/\got{h}})_*}{\twoheadrightarrow}C^\bullet_{\rm def}(\got{h}\subset \got{g})
				\end{tikzcd}
			\end{equation*}
			where $\iota_*(\phi)(x_1,\dots,x_k)=\iota(\phi(x_1,\dots,x_k))$ and $(\pi_{\got{g}/\got{h}})_*(\psi)(x_1,\dots,x_k)=\pi_{\got{g}/\got{h}}(\psi(x_1,\dots,x_k))$ for $\phi\in \wedge^k\got{h}^*\otimes\got{h}$, $\psi\in \wedge^k\got h^*\otimes \got g$ and $x_1,\ldots, x_k\in\got{h}$.
			The equalities
			\[ \iota_*\circ\delta_{\got{g}}^{\ad}=\delta_\iota\circ \iota_* \quad \text{ and } \quad (\pi_{\got{g}/\got{h}})_*\circ \delta_{\iota}=\delta^{\text{Bott}}_{\got{h}}\circ (\pi_{\got{g}/\got{h}})_*
			\] are immediate. This observation is a direct consequence of Remark \ref{remark_functor_g-modules_cochain_complexes}.
			\item 	Consider the cokernel  complex 
			\[ \Coker^\bullet(\iota_*)\coloneqq \frac{C^\bullet_{\rm def}(\iota)}{C^\bullet(\got{h};\got{h})}=\frac{\wedge^\bullet\got{h}^*\otimes\got{g}}{\wedge^\bullet\got{h}^*\otimes\got{h}}
			\]
			with the differential $\overline{\delta_\iota}$ defined by the equation $\overline{\delta_\iota}\circ p=p\circ\delta_\iota$, where \[p\colon 
			C_{\rm def}^\bullet(\iota)\to \Coker^\bullet(\iota_*) 
			\] is the quotient map. The cokernel  complex is then canonically isomorphic to the deformation complex of Lie subalgebras  $C^\bullet_{\rm def}(\got{h}\subset\got{g})$, since the category of cochain complexes is an abelian category \cite[Theorem 1.2.3]{Weibel-book-Homological-Algebra} and so, in particular, each epimorphism, as $(\pi_{\got{g}/\got{h}})_*$, is canonically isomorphic to the cokernel of its kernel \cite[Definition 1.2.1 and 1.2.2]{Weibel-book-Homological-Algebra}.
		%
	\end{enumerate}
\end{rem}

\subsection{Lie algebraic structures as Maurer-Cartan elements }\label{MC_section}

This section collects 
two  situations, the ones of Lie algebras and of Lie algebras with representations, 
where the considered Lie structures are realised as Maurer-Cartan elements in appropriate graded Lie algebras, and their deformations then become Maurer-Cartan elements in the obtained differential graded Lie algebras.
\subsubsection{Lie algebras as Maurer-Cartan elements \cite{Nijenhuis-Richardson-Deformations-of-Lie-algebra-structures-67'}}\label{sec_LA_MC}
Consider   a vector space $V$ and the graded Lie algebra $L\eqdef C^\bullet(V;V)[1]=\wedge^{\bullet+1}V^*\otimes V$ (recall that this notation means that the elements of $\wedge^{r+1}V^*\otimes V$ have degree $r$) with the graded Lie algebra bracket given by 
\begin{align}\label{eq_bracket_explicit}\llbra \xi, \eta\rrbra=(-1)^{(r-1)\cdot (s-1)}\xi\circ\eta-\eta\circ\xi\in\wedge^{r+s-1}V^*\otimes V=(C^\bullet(V;V)[1])_{r+s-2}\end{align}
on tensors $\xi\in \wedge^r V^*\otimes V=(C^{\bullet}(V;V)[1])_{r-1}$ and $\eta\in \wedge^sV^*\otimes V=(C^{\bullet}(V;V)[1])_{s-1}$ 
(of degrees $r-1$, and $s-1$, respectively),
where:
\begin{align*}
	(\xi\circ\eta)(x_1,\dots,x_{r+s-1})&=\sum_{\tau\in S_{(s,r-1)}}(-1)^{\tau}\xi(\eta(x_{\tau(1)},\dots,x_{\tau(s))}),x_{\tau(s+1)},\dots,x_{\tau(r+s-1)})
\end{align*}
on all $x_1,\ldots, x_{r+s-1}\in V$.
A simple computation using \eqref{eq_bracket_explicit} shows that a degree $1$-element $\mu\in\wedge^2V^*\otimes V$ satisfies $\llbra\mu,\mu\rrbra=-2\cdot \operatorname{Jac}_\mu$, hence  \begin{equation}\label{Lie_Brackets_mu}
	\llbra \mu, \mu\rrbra= 0 \quad \text{ if and only if }\quad \mu \text{ is a Lie bracket on } V. 
\end{equation}
The operator $\llbra \mu,\cdot\rrbra\colon \wedge^{\bullet+1}V^*\otimes V\to \wedge^{\bullet+2}V^*\otimes V$ is then the Chevalley-Eilenberg differential $\delta^{\ad}_{\got g}$ defined by the Lie bracket $\mu$, and $(C^\bullet(V;V)[1], \llbra \cdot\,,\cdot\rrbra, \delta^{\ad}_{\got g})$ is a dgLa. This follows also in a straightforward manner from the formula \eqref{eq_bracket_explicit}.

It is then immediate that for $\widetilde\mu\in \wedge^2V^*\otimes V$, the sum $\mu+\widetilde\mu$ is a Lie algebra bracket on $V$ if and only if 
\[ \llbra \widetilde \mu, \widetilde \mu\rrbra+2\delta^{\ad}_{\got g}(\widetilde \mu)=0,
\]
i.e.~if and only if $\widetilde \mu$ is a Maurer-Cartan element of the differential graded Lie algebra $(C^\bullet(V;V)[1], \llbra\cdot\,,\cdot\rrbra, \delta^{\ad}_{\got g})$.

\subsubsection{Lie algebras and representations as Maurer-Cartan elements \cite{Zambon-Fregier-SimultaneousDefOfAlgebraAndMorphisms-15'}}\label{graded_LA_LA+rep}
Now, let $V$ and $W$ be two real and finite-dimensional vector spaces and consider 
\[L\eqdef C^\bullet(V, V)[1]\oplus C^{\bullet}(V, \operatorname{End}(W))
\]
That is, the space of elements of degree $p\geq 0$ of $L$ is
\[\left(\wedge^{p+1}V^*\otimes V \right)\oplus\left(\wedge^p V^*\otimes\operatorname{End}(W)\right)
\]
and 
\[ V\simeq\wedge^0V^*\otimes V
\]
is the space of elements of degree $-1$ of $L$.
Then $L$ is equipped with a graded Lie bracket defined as follows.  Choose $\xi, \xi'\in\wedge^{\bullet+1}V^*\otimes V$ of degrees $p$ and $p'$, and $\eta, \eta'\in\wedge^\bullet V^*\otimes\operatorname{End}(W)$ of degrees $q$ and $q'$, respectively. Then
\begin{align}
	\llbra \xi, \xi'\rrbra&=(-1)^{p\cdot p'}\xi\circ\xi'-\xi'\circ\xi\in\wedge^{p+p'+1}V^*\otimes V\label{eq1_LALArep}\\
	\llbra \xi, \eta\rrbra&=
	\left\{\begin{array}{ll}
		-\eta\circ \xi\in\wedge^{p+q}V^*\otimes\operatorname{End}(W) & \text{ if } q\geq 1\\
		\quad 0 & \text{ if } q=0\end{array}\right.
	\label{eq2_LALArep}\\
	\llbra \eta, \eta'\rrbra&=(-1)^{qq'}\eta\circ \eta'-\eta'\circ \eta\in\wedge^{q+q'}V^*\otimes\operatorname{End}(W)\label{eq3_LALArep}
\end{align}
where $\xi\circ\xi'$ is defined as in \eqref{eq_bracket_explicit} and 
\begin{align}
	(\eta\circ \xi)(x_1,\dots,x_{p+q})&=\sum_{\tau\in S_{(p+1,q-1)}}(-1)^{\tau}\eta(\xi(x_{\tau(1)},\dots,x_{\tau(p+1)}),x_{\tau(p+2)},\dots,x_{\tau(p+q)})\\
	(\eta\circ \eta')(x_1,\dots,x_{q+q'})&=\sum_{\tau\in S_{(q',q)}}(-1)^{\tau}\eta(x_{\tau(q'+1)},\dots,x_{\tau(q+q')})\circ \eta'(x_{\tau(1)},\dots,x_{\tau(q')})\label{comp_def}
\end{align}
on $x_i\in V$.
Consider an element $\mu+\rho\in \left(\wedge^{2}V^*\otimes V \right)\oplus\left(\wedge^1 V^*\otimes\operatorname{End}(W)\right)$ of degree $1$ in $L$.
Then 
\[ \left\llbra \mu+\rho, \mu+\rho\right\rrbra=(-2\mu\circ \mu)+(-2\rho\circ \mu-2\rho\circ\rho)
=(-2\operatorname{Jac}_\mu)+(-2\rho\circ \mu-2\rho\circ\rho).
\]
Compute on $x_1,x_2\in V$
\begin{equation*}
	\begin{split}
		(-\rho\circ\mu-\rho\circ\rho)(x_1,x_2)=-\rho(\mu(x_1,x_2))-\rho(x_2)\circ\rho(x_1)+\rho(x_1)\circ\rho(x_2).
	\end{split}
\end{equation*}
This shows that $ \left\llbra \mu+ \rho, \mu+ \rho\right\rrbra=0$ if and only if $\mu$ is a Lie bracket on $V$ and $\rho$ is a representation of $(V,\mu)$ on $W$.

As before, the operator $\delta=\llbra \mu+\rho,\cdot\rrbra$ on $L$ makes $(L, \llbra\cdot\,,\cdot\rrbra, \delta)$ into a differential graded Lie algebra.
It is then immediate that for $\widetilde\mu\in \wedge^2V^*\otimes V$ and $\widetilde\rho\in \wedge^1 V^*\otimes\operatorname{End}(W)$, the sum $\mu+\widetilde\mu$ is a Lie algebra bracket on $V$ such that $\rho+\widetilde\rho$ is a representation of $(V,\mu+\widetilde\mu)$ on $W$ if and only if 
\[ 0=\llbra (\mu+\widetilde \mu)+(\rho+\widetilde\rho),(\mu+\widetilde \mu)+(\rho+\widetilde\rho)\rrbra
=\llbra \widetilde \mu+\widetilde \rho, \widetilde \mu+\widetilde \rho\rrbra +2\delta_{\mu+\rho}(\widetilde\mu+\widetilde \rho),
\]
i.e.~if and only if $\widetilde \mu+\widetilde \rho$ is a Maurer-Cartan element of the differential graded Lie algebra $(L, \llbra \cdot\,,\cdot\rrbra, \delta_{\mu+\rho})$.

\medskip

The following remark about the composition defined in \eqref{comp_def} is useful later on in this paper.
\begin{rem}\label{remark_on_comp_forms}
	\begin{enumerate}
		\item For $\eta=\omega\otimes \phi$ and $\eta'=\omega'\otimes \phi'$ with $\omega, \omega'\in \wedge^\bullet V^*$ and $\phi,\phi'\in \End(W)$, \
		\[ \eta\circ\eta'=(\omega\otimes\phi)\circ(\omega'\otimes\phi')=(\omega\wedge\omega')\otimes(\phi\circ\phi').
		\]
		Hence the operation $\circ$ defined in \eqref{comp_def} is associative and for $\eta, \eta', \eta''\in\wedge^\bullet V^*\otimes\operatorname{End}(W)$ of degrees $q, q'$ and $q''$, respectively, the composition $(\eta\circ \eta'\circ \eta'')(x_1,\dots,x_{q+q'+q''})$ equals 
		\begin{equation}\label{comp_of_3}
			\sum_{\tau\in S_{(q'',q',q)}}(-1)^{\tau}\eta(x_{\tau(1+q'+q'')},\dots,x_{\tau(q+q'+q'')})\circ \eta'(x_{\tau(1+q'')},\dots,x_{\tau(q'+q'')})\circ\eta''(x_{\tau(1)},\dots,x_{\tau(q'')}),
		\end{equation}
		for all $x_1,\ldots, x_{q+q'+q''}\in V$, where $S_{(q'',q',q)}\subset S_{q+q'+q''}$ is the subgroup of $(q'',q',q)$-shuffles, i.e.~permutations $\tau\in S_{q+q'+q''}$ such that
		\[ \tau(1)<\ldots<\tau(q''), \quad \tau(1+q'')<\ldots<\tau(q'+q'')\quad \text{ and } \quad \tau(1+q'+q'')<\ldots<\tau(q+q'+q'').
		\]
		\item For vector spaces $W_1, W_2$ and $W_3$, the composition $\eta\circ\eta'\in \wedge^{q+q'}V^*\otimes \Hom(W_1,W_3)$ of a form $\eta\in\wedge^q\Hom(W_2,W_3)$ with $\eta'\in \wedge^{q'}V^*\otimes \Hom(W_1, W_2)$ is defined as in \eqref{comp_def}.
		\item Let $\got g$ be a Lie algebra and let $r_U\colon\got g\to\got{gl}(U)$, $r_V\colon\got g\to\got{gl}(V)$ and $r_W\colon\got g\to\got{gl}(W)$ be three representations of $\got g$.
		Equip the vector spaces $\Hom(U,V)$, $\Hom(V,W)$ and $\Hom(U,W)$ with the induced representations, all written $r^{\Hom}$ for simplicity.
		Then the linear composition map
		\[ \circ\colon \Hom(V,W)\otimes\Hom(U,V)\to \Hom(U,W), \qquad \phi\otimes\psi\mapsto \phi\circ \psi
		\]
		is a morphism of representations, i.e.
		\[ r^{\Hom}_x(\phi\circ \psi)=\left(r^{\Hom}_x\phi\right)\circ\psi+\phi\circ \left(r^{\Hom}_x\psi\right)
		\]
		for all $x\in\got g$ and all $\phi\in \Hom(V,W)$, $\psi\in \Hom(U,V)$.
		
		As a consequence,  the induced Chevalley-Eilenberg operators 
		\[\delta^{\Hom}\colon C^\bullet(\got g, \Hom(V,W))\to C^{\bullet+1}(\got g, \Hom(V,W)),\]
		\[\delta^{\Hom}\colon C^\bullet(\got g, \Hom(U,V))\to C^{\bullet+1}(\got g, \Hom(U,V)), \]
		and \[\delta^{\Hom}\colon C^\bullet(\got g, \Hom(U,W))\to C^{\bullet+1}(\got g, \Hom(U,W)),
		\]
		satisfy 
		\[ \delta^{\Hom}(\eta\circ\eta')=\left(\delta^{\Hom}\eta\right)\circ \eta'+ (-1)^{|\eta|}\eta\circ\left(\delta^{\Hom}\eta'\right)
		\]
		for all $\eta\in C^\bullet(\got g, \Hom(V,W))$ and $\eta'\in C^\bullet(\got g, \Hom(U,V))$. The proof is an easy computation that is left to the reader.
	\end{enumerate}
\end{rem}

\section{Deformation theory of Lie subalgebras -- recap}\label{def_subalgebras_section}

This section recalls, to the extent needed for the purposes of this paper, the deformation theory of Lie subalgebras \cite{Crainic-Ivan-Schaetz, Zambon-Fregier-SimultaneousDefOfAlgebraAndMorphisms-15'}.
Let $\got{h}$ be a $k$-dimensional Lie subalgebra of a Lie algebra $\got{g}$ and let $\Gr_k(\got{g})$ be the Grassmannian of $k$-dimensional subspaces of $\got{g}$.

\begin{defin}\label{def_def_subalgebras}
	A \textbf{smooth deformation $(\got{h}_t)_{t\in I}$ of a Lie subalgebra $\got{h}$ inside a Lie algebra $\got{g}$} is a smooth curve $\tilde{\got{h}}:[0,1]\to \Gr_k(\got{g})$ such that $\got{h}_t\coloneqq\tilde{\got{h}}(t)$ is a Lie subalgebra of $\got{g}$ for all $t\in I\coloneqq[0,1]$ and such that $\got{h}_0=\got{h}$.
	
	Two smooth deformations $(\got{h}_t)_{t\in I}$ and $(\got{h}'_t)_{t\in I}$ of $\got{h}$ inside $\got{g}$ are called \textbf{equivalent} if there exists a smooth curve $g\colon I\to G$ (the unique simply-connected integration of $\got{g}$), starting at the identity and such that $\got{h}'_t=\Ad_{g(t)}\got{h}_t$ for each $t\in I$.
\end{defin}

Appendix \ref{tangent_grk}
explains in detail how the tangent vectors to $\Gr_k(\got g)$ at $\got h$ are computed and seen as elements of $\got h^*\otimes \got g/\got h \simeq T_{\got h}\Gr_k(\got g)$.
In short, a  choice of linear complement $\got h^c\subseteq \mathfrak g$ for $\mathfrak h$ in $\mathfrak g$ defines a smooth chart of $ \Gr_k(\got{g})$ centered at $\mathfrak h$. Precisely, the map $\Psi\colon \operatorname{Hom}(\mathfrak h, \got h^c)\to \Gr_k(\got{g})$ sending $\phi$ to $\grap(\phi)$ is a diffeomorphism on its image $U_{\got h,\got h^c}:=\{W\subseteq \got{g}\mid W\oplus \got h^c=\got{g}\}\subseteq \Gr_k(\got{g})$. The inverse $\Psi^{-1}\colon U_{\got h,\got h^c}\to \operatorname{Hom}(\got{h},{\got h^c})$ sends $\mathfrak h$ to $0$.

Via this smooth chart, $I\ni t\mapsto \got h_t\in \Gr_k(\got{g})$ 
coincides with a  smooth curve $\phi\colon I\to \operatorname{Hom}(\got{h}, \got{h}^c)$, i.e.~$\got h_t=\grap(\phi(t))$ for all $t\in I$.
(Since only the values of $\tilde{\got{h}}$ in an arbitrarily small neighborhood of $0$ in $I$ are relevant, assume without loss of generality that $\tilde{\got{h}}$ has values in $U_{\got h, \got h^c}$.) Since $\operatorname{Hom}(\got{h}, \got{h}^c)$ is a vector space, the derivative $\dot{\phi}(0)$ is again an element of $\operatorname{Hom}(\got{h}, \got{h}^c)=\got{h}^*\otimes \got{h^c}$. A composition with $\pi_{\got g/\got h}\colon \got g\to\got g/\got h$ defines then 
\[\pi_{\got g/\got h}\circ \dot{\phi}(0)=\left.\frac{d}{dt}\right\arrowvert_{t=0}\got h_t\]
as an element of $T_{\got h}\Gr_k(\got g)\simeq \got h^*\otimes\got g/\got h=C^1_{\rm def}(\got{h}\subset\got{g})$.

\begin{prop}[\cite{Crainic-Ivan-Schaetz}]\label{ass_cocycle_deformations_subalgebras}
	Let $\got{h}$ be a $k$-dimensional Lie subalgebra of a Lie algebra $\got{g}$. If $(\got{h}_t)_{t\in I}$ is a smooth deformation of $\got{h}$ inside $\got{g}$, then $$\dot{\got{h}}_0\coloneqq\fracddtz\got{h}_t\in C^1_{\rm def}(\got{h}\subset\got{g})$$ is a cocycle. Moreover, the corresponding cohomology class only depends on the equivalence class of the deformation.
\end{prop}

\begin{proof}
	Consider as above $\tilde{\got{h}}$ to be a smooth curve in $\phi\colon I\to \operatorname{Hom}(\got{h}, {\got h^c})$ for a linear complement ${\got h^c}$ of $\got{h}$ in $\got{g}$ and set 
	\[ \alpha\colon I\to \GL(\got{g}), \qquad t\mapsto \begin{pmatrix} \id_{\got{h}} & 0\\ \phi(t)&  \id_{\got h^c}\end{pmatrix},
	\]
	where $\got{g}\simeq\got{h}\oplus {\got h^c}$. Then $\alpha(t)(\got{h})=\grap(\phi(t))=\got{h}_t$  and
	\[ \pi_{\got g/\got h}\circ\dot\alpha(0)\arrowvert_{\got h}=\pi_{\got g/\got h}\circ\dot{\phi}(0)=\dot{\got{h}}_0,
	\]
	see Appendix \ref{tangent_grk}.
	Set similarly $\pi_t\colon \got{g}\to\got{g}/\got{h_t}$ to be the canonical projection for each $t\in I$. Since $(\pi_t\circ\alpha(t))(\got{h})=\pi_t(\got{h}_t)=0$ for all $t\in I$, $\alpha(t)$ factors for each $t$ to $\overline{\alpha(t)}:=\overline{\alpha}(t)$ as in the following commutative diagram.
	\begin{equation}\label{commutative diagram in the proof of the cocycle in Lie subalgebras!}
		\begin{tikzcd}
			\got{g}\arrow[r,"\alpha(t)"] \arrow[d,swap,"\pi"] &
			\got{g} \arrow[d,"\pi_t"]\\
			\got{g}/\got{h} \arrow[r,"\overline{\alpha}(t)"] & \got{g}/\got{h}_t
		\end{tikzcd}
	\end{equation}
	The maps $\overline{\alpha}(t)$ are isomorphisms for all $t\in I$. In order to see this, note that 
	for $x\in\got{g}$, $\overline{\alpha}(t)(x+\got{h})=0$ if and only if $\alpha(t)(x)\in \got{h}_t$. Since $\alpha(t)$ is bijective with $\alpha(t)(\got{h})=\got{h}_t$, this is the case if and only if $x\in\got{h}$.

	Since $\got{h}_t\subset\got{g}$ is a Lie subalgebra for each $t\in I$ the map $\sigma\colon I\to \wedge^2\got{h}^*\otimes\got{g}/\got{h}$ defined for all $t\in I$ by		\begin{equation*}
		\sigma(t)(x,y)=\overline{\alpha}(t)^{-1}\circ\pi_t[\alpha(t)(x),\alpha(t)(y)]\stackrel{\eqref{commutative diagram in the proof of the cocycle in Lie subalgebras!}}{=}\pi\circ\alpha(t)^{-1}[\alpha(t)(x),\alpha(t)(y)]
	\end{equation*}
	for all $x,y\in\got{h}$, 
	vanishes identically on $I$. Hence, differentiating $t\mapsto \sigma(t)(x,y)$ at $t=0$ for $x, y\in \got h$ yields
	\begin{equation}\label{cocycle condition for Lie subalgebras}
		-\pi\circ\dot{\alpha}(0)[x,y]+\pi[\dot{\alpha}(0)(x),y]+\pi[x,\dot{\alpha}(0)(y)]=0,
	\end{equation}
	which is exactly the cocycle condition $\delta^{\text{Bott}}_{\got{h}}(\dot{\got{h}}_0)=\delta^{\text{Bott}}_{\got{h}}(\pi_{\got g/\got h}\circ\dot{\alpha}(0)\arrowvert_{\got{h}})=0$. 
	\medskip
	
	Let $(\got{h}_t)_{t\in I}$ and $(\got{h}'_t)_{t\in I}$ be two equivalent deformations of a Lie subalgebra $\got{h}$ of $\got{g}$ and let 
	$g\colon I\to G$ be the smooth curve with $g(0)=e$ and $\got{h}'_t=\Ad_{g(t)}\got{h}_t$ for all $t\in I$.
	Then $\Ad_{g(t)}\circ\alpha(t)=:\alpha'(t)$ defines a smooth curve $\alpha'\colon I \to \GL(\got{g})$ starting at the identity with $\alpha'(t)(\got{h})=\got{h}'_t$ for all $t\in I$. A differentiation at $t=0$ yields $$[\dot{g}(0),x]+\dot{\alpha}(0)(x)=\dot{\alpha'}(0)(x)$$
	for all $x\in\got{g}$.
	Applying the projection $\pi_{\got g/\got h}\colon \got{g}\to\got{g}/\got{h}$ then leads to $\dot{\got{h}}_0-\dot{\got{h}}'_0=\pi_{\got g/\got h}\circ\dot{\alpha}(0)\arrowvert_{\got{h}}-\pi_{\got g/\got h}\circ\dot{\alpha'}(0)\arrowvert_{\got{h}}=\delta^{\text{Bott}}_{\got{h}}(\dot{g}(0))$.
\end{proof}

\begin{rem}\label{rem_def_morphisms}
	Similarly, a deformation of a Lie algebra morphism gives as follows a \textbf{deformation cocycle}, see \cite{Crainic-Ivan-Schaetz}, and references therein. Let $\phi_0\colon \got g\to \got h$ be a morphism of Lie algebras and let $\phi\colon I \to \got g^*\otimes \got h$ be a smooth curve defined on an open interval $I$ containing $0$, such that $\phi(t)\colon \got g\to \got h$ is a morphism of Lie algebras for all $t\in I$, and such that $\phi(0)=\phi_0$.
	
	Then $\dot \phi(0)=\left.\frac{d}{dt}\right\arrowvert_{t=0}\phi(t)$ is again an element in the vector space $\got g^*\otimes \got h$ and for all $x,y\in \got g$
	\begin{equation*}
		\begin{split}
			\delta_{\phi_0}\left(\dot \phi(0)\right)(x,y)&=\left[\phi_0(x), \dot \phi(0)(y)\right]-\left[\phi_0(y), \dot \phi(0)(x)\right]- \dot \phi(0)([x,y])\\
			&=\left[\phi(0)(x), \dot \phi(0)(y)\right]+\left[\dot \phi(0)(x),\phi(0)(y)\right]- \dot \phi(0)([x,y])\\
			&=\left.\frac{d}{dt}\right\arrowvert_{t=0}\left(\left[\phi(t)(x), \phi(t)(y)\right]-\phi(t)([x,y])\right)
			=\left.\frac{d}{dt}\right\arrowvert_{t=0}0=0.
		\end{split}
	\end{equation*}
	That is, $\dot \phi(0)\in Z^1_{\phi_0}(\got g, \got h)$. The element $\dot{\phi}(0)$ is called the \textbf{deformation cocycle} of the Lie algebra morphism $\phi\colon \got{h}\to\got{g}$ associated to $(\phi_t)_{t\in I}$.
\end{rem}

\bigskip

The following Voronov-dataset due to \cite{Zambon-Fregier-SimultaneousDefOfAlgebraAndMorphisms-15'} is used for constructing the controlling L$_\infty[1]$-algebra of small deformations of a Lie subalgebra $\got{h}\subset\got{g}$. 
Here, instead of considering a linear complement 
${\got h^c}$ of $\got{h}$ in $\got{g}$, 
consider the equivalent splitting 
\[\sigma\colon \got g/\got h\rightarrow \got g
\]
of the short exact sequence 
\begin{equation}\label{ses_pi_h}
	0 \rightarrow \got h \rightarrow \got g\rightarrow \got g/\got h\rightarrow 0
\end{equation}
of vector spaces defined by the epimorphism $\pi_{\got g/\got h}\colon \got g\to \got g/\got h$ of vector spaces.
Denote by $p_\sigma\colon \got{g}\to\got{h}$, $x\mapsto x-\sigma(\pi_{\got g/\got h}(x))$ the projection associated to this choice of splitting. 
The inclusion of $\got h$ in $\got g$ is written $i_{\got h}\colon \got h\hookrightarrow \got g$ in order to distinguish it from the contraction operator $\iota$.

\begin{lem}
	In the situation above, the following quadruple defines a Voronov-dataset:
	\begin{itemize}
		\item $L\eqdef C^\bullet(\got{g};\got{g})[1]=(\wedge^{\bullet}\got{g}^*\otimes\got{g})[1]$  with the graded Lie algebra bracket given as in \eqref{eq_bracket_explicit}:
		\[ \llbra \phi\otimes v, \psi\otimes w\rrbra =(-1)^{(r-1)\cdot(s-1)}\psi\wedge\iota_w\phi\otimes v- \phi\wedge\iota_v\psi\otimes w
		\]
		on elementary tensors $\phi\otimes v\in \wedge^r\mathfrak g^*\otimes \mathfrak g$ and $\psi\otimes w\in \wedge^s\mathfrak g^*\otimes \mathfrak g$ 
		(of degrees $r-1$, and $s-1$, respectively).
		
		\item $\got{a}\eqdef C^\bullet(\got{h}; {\got g/\got h})[1]=\wedge^{\bullet+1}\got{h}^*\otimes {\got g/\got h}$, with the inclusion $I\colon \got{a}\to L$ given by $I(\psi\otimes v)=p_{\sigma}^*\psi\otimes \sigma(v)$.
		\item $P\colon L\to\got{a}$ the projection given by: $\phi\otimes x\mto i_{\got h}^*\phi\otimes \pi_{\got g/\got h}(x)$ with kernel given by $$\Ker P=\bigleftpar\wedge^{\bullet+1}\got{g}^*\otimes\got{h}\bigrightpar+\bigleftpar\got h^\circ\wedge(\wedge^\bullet\got g^*)\otimes {\got g}\bigrightpar.$$
		\item $\Theta=\mu_{\got g}$ is the Lie bracket on $\mathfrak g$, which, as an element of $\wedge^2\mathfrak g^*\otimes \mathfrak g$, has degree $1$, and satisfies $\llbra \mu_{\got g}, \mu_{\got g}\rrbra=0$ by \eqref{Lie_Brackets_mu}.
	\end{itemize}
\end{lem}

\begin{proof}
	The subspace $\got{a}$ is an abelian Lie subalgebra of $L$, because for 
	$\psi_1,\psi_2\in\wedge^{\bullet+1}\mathfrak h^* $ and $v_1,v_2\in {\got g/\got h}$	\begin{align}\label{IIzero}
		I(\psi_1\otimes v_1)\circ I(\psi_2\otimes v_2)=p_\sigma^*(\psi_2)\wedge\iota_{\sigma(v_2)}(p_\sigma^*\psi_1)\otimes \sigma(v_1)=0,
	\end{align}
	using $p_\sigma\circ \sigma=0$.
	Next verify that $\Ker P$ is a Lie subalgebra of $L$.  The following computations show that the bracket of any two elements in $\Ker P$ lies again in $\Ker P$.
	For $\phi_1,\phi_2\in\wedge^{\bullet+1}\mathfrak g^*$ and $u_1,u_2\in\mathfrak h$,
	\begin{equation*}
		\llbra\phi_1\otimes u_1,\phi_2\otimes u_2\rrbra=(-1)^{|\phi_1||\phi_2|}\phi_2\wedge\iota_{u_2}\phi_1\otimes u_1-\phi_1\wedge\iota_{u_1}{\phi_2}\otimes u_2\\
	\end{equation*}
	is again an element of $\wedge^{\bullet+1}\mathfrak g^*\otimes\mathfrak h$. For $\psi_1\in\wedge^{r+1}\mathfrak g^*$, $u\in \got h$, $\alpha\in\got h^\circ$, $\psi_2\in\wedge^s \got g^*$ and $v\in {\got g}$
	\begin{equation*}
		\begin{split}
			&\llbracket \psi_1\otimes u, \alpha\wedge \psi_2\otimes v\rrbra
			=(-1)^{rs}\alpha\wedge \psi_2\wedge\iota_v\psi_1\otimes u-\underset{=0}{\underbrace{\psi_1\wedge(\iota_{u}\alpha)\wedge\psi_2\otimes v}}+\psi_1\wedge\alpha\wedge\iota_{u}\psi_2\otimes v
		\end{split}
	\end{equation*}
	is an element of $\got h^\circ\wedge(\wedge^\bullet\got g^*)\otimes {\got g}\subseteq \ker P$.
	For $\psi_1,\psi_2\in\wedge^{\bullet}\mathfrak g^*$, $\alpha_1,\alpha_2\in\got h^\circ$ and $v_1,v_2\in\got g$,
	\begin{equation*}
		\begin{split}
			\llbra \alpha_1\wedge\psi_1\otimes v_1, \alpha_2\wedge\psi_2\otimes v_2\rrbra
			&=(-1)^{|\psi_1|\cdot|\psi_2|}\alpha_2\wedge\psi_2\wedge\iota_{v_2}(\alpha_1\wedge\psi_1)\otimes v_1
			-\alpha_1\wedge\psi_1\wedge\iota_{v_1}(\alpha_2\wedge\psi_2)\otimes v_2\end{split},
	\end{equation*}
	which is again an element of $\got h^\circ\wedge(\wedge^\bullet\got g^*)\otimes {\got g}$.

	Finally note that $P(\mu_{\got g})(u_1,u_2)=\pi_{\got g/\got h}[u_1,u_2]$ for all $u_1,u_2\in\mathfrak h$. Therefore, $\Theta=\mu_{\got g}\in\Ker P$ if and only if $\got{h}\subset\got{g}$ is a Lie subalgebra. As seen before, $\llbra \mu_{\got g},\mu_{\got g}\rrbra=0$ since $\mu_{\got g}$ is the Lie bracket on $\mathfrak g$.
\end{proof}

The multibrackets of the $L_\infty[1]$-algebra on $\got{a}$ induced by the above Voronov-dataset, are given by
\begin{equation*}
	m_l(a_1, \ldots, a_l)=P\llbra\llbra\ldots\llbra\llbra\mu_{\got g}, I(a_1)\rrbra, I(a_2)\rrbra,\ldots\rrbra, I(a_l)\rrbra
\end{equation*}
for $l\geq 1$ and $a_1,\ldots, a_l\in\mathfrak a$.

First recall that $\llbra\mu_{\got g},I(a)\rrbra=\delta_{\got{g}}^{\ad}(I(a))$ for $a\in\mathfrak a$. Hence, for $a_1,a_2,a_3,a_4\in\mathfrak a$
\begin{equation*}
	\begin{split}
		&\llbra\llbra\llbra\llbra\mu_{\got g}, I(a_1)\rrbra, I(a_2)\rrbra,I(a_3)\rrbra, I(a_4)\rrbra=\llbra\llbra\llbra\delta_{\got{g}}^{\ad}(I(a_1)), I(a_2)\rrbra,I(a_3)\rrbra, I(a_4)\rrbra\\
	\end{split}
\end{equation*}
is, up to signs, 
\begin{equation*}
	\begin{split}
		&\pm\left(\left(\delta_{\got{g}}^{\ad}(I(a_1))\circ I(a_2)\right)\circ I(a_3)\right)\circ I(a_4)
		\pm I(a_4)\circ \left(\left(\delta_{\got{g}}^{\ad}(I(a_1))\circ I(a_2)\right)\circ I(a_3)\right)\\
		&\pm \left(I(a_3)\circ\left(\delta_{\got{g}}^{\ad}(I(a_1))\circ I(a_2)\right)\right)\circ I(a_4)
		\pm\cancel{ I(a_4)\circ \left(I(a_3)\circ\left(\delta_{\got{g}}^{\ad}(I(a_1))\circ I(a_2)\right)\right)}\\
		&\pm\left(\left(I(a_2)\circ\delta_{\got{g}}^{\ad}(I(a_1))\right) \circ I(a_3)\right)\circ I(a_4)\pm \cancel{I(a_4)\circ\left(\left( I(a_2)\circ \delta_{\got{g}}^{\ad}(I(a_1))\right)\circ I(a_3)\right)}\\
		&\pm\left(\cancel{I(a_3)\circ \left(I(a_2) \circ \delta_{\got{g}}^{\ad}(I(a_1))\right)}\right)\circ I(a_4)
		\pm I(a_4)\circ\left(\cancel{I(a_3)\circ\left(I(a_2)\circ   \delta_{\got{g}}^{\ad}(I(a_1))\right)}\right).
	\end{split}
\end{equation*}
Note that the contraction operator $\circ$ used here is \emph{not} associative. However a careful study using $\pr_\sigma\circ \sigma=0$ of these eight terms show that they all vanish.
The four barred terms vanish immediately using $\pr_\sigma\circ \sigma=0$, since $I(a)=\sigma\circ(\pr_\sigma^*a)$ and the respective parentheses all start with applying $\sigma$ to a form with values in $\got g/\got h$. The first term vanishes
because all summands of $\delta^{\rm ad}_{\got g}I(a_1)$ are contracted with at least one of the forms $I(a_2), I(a_3), I(a_4)$, and again $\pr_\sigma\circ \sigma=0$.
The second term vanishes because the only term that does not vanish in the parenthesis is of the form $\sigma(a_1(\pr_\sigma[I(a_2)(\ldots),I(a_3)(\ldots)], \pr_\sigma(\cdot),\ldots,\pr_\sigma(\cdot)))$, which is then vanishing in the contraction in $I(a_4)$. Similarly, the last two terms vanish as well.
This shows that $m_k=0$ for all $k\geq 4$, see also  \cite{Zambon-Fregier-SimultaneousDefOfAlgebraAndMorphisms-15'}.

Choose  $\xi\in \wedge^r\got{h}^*\otimes {\got g/\got h}$ (i.e.~of degree $r-1$). Then 
\begin{equation*}
	\begin{split}
		m_1(\xi)&=P\llbra\mu_{\got g},I(\xi)\rrbra= P\left(\delta_{\got{g}}^{\ad}(I(\xi))\right).
	\end{split}
\end{equation*}
For $h_1,\ldots, h_{r+1}\in\got{h}$ compute
\begin{equation*}
	\begin{split}
		P\left(\delta^{\ad}_{\got g}(I(\xi))\right)(h_1,\ldots, h_{r+1})&=\pi_{\got g/\got h}\left(\delta_{\got{g}}^{\ad}(I(\xi))(h_1,\ldots, h_{r+1})\right)\\
		&=\sum_{i=1}^{r+1}(-1)^{i+1}\pi_{\got g/\got h}\left[h_i, \sigma(\xi(h_1, \ldots, \widehat{i}, \ldots, h_{r+1}))\right]\\
		&\quad \quad +\sum_{1\leq i<j\leq r}(-1)^{i+j}\xi\left(\left[h_i, h_j\right], h_1, \ldots, \widehat{i}, \ldots, \widehat{j}, \ldots, h_{r+1}\right)\\
		&=(\delta_{\got{h}}^{\text{Bott}}\xi)(h_1,\ldots, h_{r+1}).
	\end{split}
\end{equation*}
This shows that  $m_1=\delta_{\got{h}}^{\text{Bott}}$, and in particular that $m_1$ does not depend on the choice of $\sigma$.

Now, let $\xi\in C^{r}(\got{h},{\got g/\got h})=\wedge^{r}\got{h}^*\otimes {\got g/\got h}$ and $\eta\in C^{s}(\got{h},{\got g/\got h})$ (of degrees $r-1$ and $s-1$, respectively).  
Then 
\begin{equation*}
	\begin{split}
		m_2(\xi,\eta)&= P\llbra \llbra \mu_{\got g}, I(\xi)\rrbra, I(\eta)\rrbra= P\left\llbra \delta^{\rm ad}_{\got g}(I(\xi)), I(\eta)\right\rrbra \in C^{r+s}(\got{h}, {\got g/\got h})
	\end{split}
\end{equation*}
Assume that $r=s=1$ and choose $h_1, h_2\in\got{h}$. Then by \eqref{eq_bracket_explicit}
\begin{equation}\label{compute_m_2}
	\begin{split}
		m_2(\xi,\eta)(h_1, h_2)& = \pi_{\got g/\got h}\bigl(\delta^{\rm ad}_{\got g}(I(\xi))\left(\sigma(\eta(h_{1})), h_{2}\right)-\delta^{\rm ad}_{\got g}(I(\xi))\left(\sigma(\eta(h_{2})), h_{1}\right)-\sigma\left(\eta\left(p_\sigma\left(\delta^{\rm ad}_{\got g}(I(\xi))(h_{1}, h_{2})\right)\right)\right)\bigr)\\
		&=\pi_{\got g/\got h}\bigl([(\sigma\circ\eta)(h_1),(\sigma\circ\xi)(h_2)]-[h_2, \cancel{\xi((p_\sigma\circ \sigma\circ \eta)(h_1))}]-\sigma(\xi(p_\sigma[(\sigma\circ\eta)(h_1),h_2]))\\
		&\qquad -[(\sigma\circ\eta)(h_2),(\sigma\circ\xi)(h_1)]+[h_1, \cancel{\xi((p_\sigma\circ \sigma\circ \eta)(h_2))}]+\sigma(\xi(p_\sigma[(\sigma\circ\eta)(h_2),h_1]))\bigr)\\
		&\qquad -\eta\left(p_\sigma\left([h_1,(\sigma\circ\xi)(h_2)]-[h_2,(\sigma\circ\xi)(h_1)]\right)\right)+\cancel{\eta\left((p_\sigma\circ\sigma\circ\xi)[h_1,h_2]\right)}\\
		&=\pi_{\got g/\got h}[(\sigma\circ\eta)(h_1),(\sigma\circ\xi)(h_2)]-\xi(p_\sigma[(\sigma\circ\eta)(h_1),h_2])\\
		&\qquad -\pi_{\got g/\got h}[(\sigma\circ\eta)(h_2),(\sigma\circ\xi)(h_1)]+\xi(p_\sigma[(\sigma\circ\eta)(h_2),h_1])\\
		&\qquad -\eta\left(p_\sigma\left([h_1,(\sigma\circ\xi)(h_2)]-[h_2,(\sigma\circ\xi)(h_1)]\right)\right).
	\end{split}
\end{equation}
In particular 
\begin{equation}\label{compute_m_2_b}
	\begin{split}
		m_2(\xi,\xi)(h_1, h_2)=2\pi_{\got g/\got h}[(\sigma\circ\xi)(h_1),(\sigma\circ\xi)(h_2)]-2\xi\left(p_\sigma[(\sigma\circ\xi)(h_1),h_2]\right)-2\xi\left(p_\sigma[h_1,(\sigma\circ\xi)(h_2)]\right).
	\end{split}
\end{equation}

Last,  the trinary bracket is worked out. Let $\xi,\eta,\nu\in C^{\bullet}(\got{h};{\got g/\got h})[1]$.
Then 
\begin{equation*}
	\begin{split}
		m_3(\xi,\eta,\nu)=P\llbra\llbra\llbra \mu_{\got g}, I(\xi)\rrbra, I(\eta)\rrbra, I(\nu)\rrbra=P\llbra\llbra \delta^{\rm ad}_{\got g}(I(\xi)), I(\eta)\rrbra, I(\nu)\rrbra.
	\end{split}
\end{equation*}
Assume that $|\xi|=|\eta|=|\nu|=0$ and compute for $h_1,h_2\in\got{h}$
\begin{equation}\label{m_3_subalgebra}
	\begin{split}
		&m_3(\xi,\eta,\nu)(h_1,h_2)=\pi_{\got g/\got h}\left(\llbra\llbra \delta^{\rm ad}_{\got g}(I(\xi)), I(\eta)\rrbra, I(\nu)\rrbra(h_1,h_2)\right)\\
		&\,=\pi_{\got g/\got h}\left(\llbra \delta^{\rm ad}_{\got g}(I(\xi)), I(\eta)\rrbra((\sigma\circ\nu)(h_1), h_2)
		-\llbra \delta^{\rm ad}_{\got g}(I(\xi)), I(\eta)\rrbra((\sigma\circ\nu)(h_2), h_1)\right)\\
		&\qquad \qquad	-(\nu\circ p_\sigma)\left(\llbra \delta^{\rm ad}_{\got g}(I(\xi)), I(\eta)\rrbra(h_1,h_2)\right).
	\end{split}
\end{equation}
Easy computations give that this is 
\begin{equation*}
	\begin{split}
		&\pi_{\got g/\got h}\left(\llbra \delta^{\rm ad}_{\got g}(I(\xi)), I(\eta)\rrbra((\sigma\circ\nu)(h_1), h_2)
		-\llbra \delta^{\rm ad}_{\got g}(I(\xi)), I(\eta)\rrbra((\sigma\circ\nu)(h_2), h_1)\right)\\
		&\qquad \qquad \qquad  -(\nu\circ p_\sigma)\left(\llbra \delta^{\rm ad}_{\got g}(I(\xi)), I(\eta)\rrbra(h_1,h_2)\right)\\
		&\, =-\xi\left(p_\sigma[(\sigma\circ\nu)(h_1),(\sigma\circ\eta)(h_2)]\right)-\eta\left(p_\sigma[(\sigma\circ\nu)(h_1),(\sigma\circ\xi)(h_2)]\right)-\xi\left(p_\sigma[(\sigma\circ\eta)(h_1),(\sigma\circ\nu)(h_2)]\right)\\
		&\qquad-\eta\left(p_\sigma[(\sigma\circ\xi)(h_1),(\sigma\circ\nu)(h_2)]\right)
		-\nu\left(p_\sigma[(\sigma\circ\xi)(h_1), (\sigma\circ\eta)(h_2)]\right)-\nu\left(p_\sigma[(\sigma\circ\eta)(h_1),(\sigma\circ\xi)(h_2)]\right).
	\end{split}
\end{equation*}
In particular, 
\[m_3(\xi,\xi,\xi)(h_1,h_2)=-6\xi\left(p_\sigma[(\sigma\circ\xi)(h_1), (\sigma\circ\xi)(h_2)]\right)
\]
for all $h_1,h_2\in\mathfrak h$.

The above are the formulas of $m_2$ and $m_3$ that are needed for the proof of the following proposition. 

\begin{prop}[\cite{Richardson-Deformations-of-subalgebras-69', Zambon-Fregier-SimultaneousDefOfAlgebraAndMorphisms-15'}]
	Given a $k$-dimensional Lie subalgebra $\got{h}\subset\got{g}$ and a splitting $\sigma\colon\got g/\got h\to \got g$ of \eqref{ses_pi_h}, there is a bijection between
	\begin{enumerate}
		\item Maurer-Cartan elements $\phi$ (of  degree $0$) of the $L_\infty[1]$-algebra $(C^\bullet(\got{h}; {\got g/\got h})[1],m_1, m_2^\sigma,m_3^\sigma)$ and
		\item $k$-dimensional Lie subalgebras $\got{h}'\subset\got{g}$ such that $\got{g}=\got{h}'\oplus \sigma({\got g/\got h})$,
	\end{enumerate}
	given by the correspondence $\phi\mto\grap(\sigma\circ\phi)$.
\end{prop}
\begin{proof}
	Consider $\phi\in \wedge^1\got h^*\otimes {\got g/\got h}$. Then 
	$\grap(\sigma\circ\phi)=\{u+(\sigma\circ\phi)(u) \ | \ u\in\got{h}\}$ is a Lie subalgebra of $\got g$ if and only if $[\grap(\sigma\circ\phi),\grap(\sigma\circ\phi)]\subset\grap(\sigma\circ\phi)$. Equivalently, this means that for every $h_1,h_2\in\got{h}$:
	\begin{align*}
		&[h_1+(\sigma\circ\phi)(h_1),h_2+(\sigma\circ\phi)(h_2)]\\
		&=[h_1,h_2]+[h_1,(\sigma\circ\phi)(h_2)]+[(\sigma\circ\phi)(h_1),h_2]+[(\sigma\circ\phi)(h_1),(\sigma\circ\phi)(h_2)]\in\grap(\sigma\circ\phi).
	\end{align*}
	This is equivalent to
	\begin{align*}
		&\pi_{\got g/\got h}\left([h_1,(\sigma\circ\phi)(h_2)]+[(\sigma\circ\phi)(h_1),h_2]+[(\sigma\circ\phi)(h_1),(\sigma\circ\phi)(h_2)]\right)\\
		&=\phi\Bigl([h_1,h_2]+p_{\sigma}\left([h_1,(\sigma\circ\phi)(h_2)]+[(\sigma\circ\phi)(h_1),h_2]
		+[(\sigma\circ\phi)(h_1),(\sigma\circ\phi)(h_2)]\right)\Bigr).
	\end{align*}
	On the other hand consider
	\begin{align*}
		m_1(\phi)+\frac{1}{2}m_2^\sigma(\phi,\phi)+\frac{1}{6}m_3^\sigma(\phi,\phi,\phi).
	\end{align*}
	On  $h_1,h_2\in\got{h}$ this is 
	\begin{align*}
		&\pi_{\got g/\got h}[h_1,(\sigma\circ\phi)(h_2)]-\pi_{\got g/\got h}[h_2,(\sigma\circ\phi)(h_1)]-\phi[h_1,h_2]\\
		&+ \pi_{\got g/\got h}[(\sigma\circ\phi)(h_1),(\sigma\circ\phi)(h_2)]-\phi\left(p_\sigma[(\sigma\circ\phi)(h_1),h_2]\right)-\phi\left(p_\sigma[h_1,(\sigma\circ\phi)(h_2)]\right)\\
		&-\phi\left(p_\sigma[(\sigma\circ\phi)(h_1), (\sigma\circ\phi)(h_2)]\right),
	\end{align*}
	which concludes the proof.
\end{proof}
\medskip

Denote by $\mathbf{\Lambda}\subset\Hom(\wedge^2\got{g},\got{g})$ the space of Lie brackets on the vector space $\got{g}$ and by $\mathbf{S}\subset\Gr_{k}(\got{g})$ the space of $k$-dimensional Lie subalgebras of the Lie algebra $(\got{g},\mu_{\got g})$.

\begin{defin}\label{defin of rigidity of Lie subalgebras}
	A Lie subalgebra $\got{h}\subset\got{g}$ is called \textbf{a rigid subalgebra} if 
	for each open subset $V\subseteq G$ containing $e\in G$ there exists an open subset $U^V_{\got{h}}\subseteq\Gr_{k}(\got{g})$ of $\got{h}$ 		such that for every Lie subalgebra $\got{h}'\subset\got{g}$ with $\got{h}'\in U^V_{\got{h}}$, there exists $g\in V$ such that
	the equality $\got{h}'=\Ad_{g}\got{h}$ holds.
\end{defin}

\begin{rem}
	Note that if $\got h \subseteq \got g$ is rigid, then for each $g\in G$, the Lie subalgebra $\Ad_g\got h\subseteq \got g$ is rigid as well: for $V\subseteq G$ open around $e$ 
	take the neighborhood \[U^V_{\Ad_g\got h}:=\Ad_g\left(U^{g^{-1}Vg}_{\got h}\right)\subseteq \Gr_{k}(\got{g})\] of $\Ad_g\got h$. Then for all $\got h'\in U^V_{\Ad_g\got h}$ there exists $\got h''\in U^{g^{-1}Vg}_{\got h}$ such that $\got h'=\Ad_{g}\got h''$   and so there exists $g'\in g^{-1}Vg$ such that 
	\[ \got h'=\Ad_{g}\got h''=\Ad_g\Ad_{g'}\got h=\Ad_{gg'g^{-1}}\Ad_g\got h,
	\] 
	with $gg'g^{-1}\in V$.
	In particular, if a Lie subalgebra $\got{h}\subset\got{g}$ is \textbf{rigid}, then 
	its orbit $\cali{O}_{\got{h}}$ is open  in $\mathbf{S}$:
	\[ \cali{O}_{\got h}:=\{\Ad_g\got h\mid g\in G\}\subseteq \cup_{g\in G}\left(U^G_{\Ad_g\got h}\cap \mathbf S\right)\subseteq \cali{O}_{\got h}
	\]
	shows that 
	\[ \cali{O}_{\got h}=\cup_{g\in G}\left(U^G_{\Ad_g\got h}\cap \mathbf S\right),
	\]
	which is an open subset of $\mathbf S$.
	
\end{rem}

\begin{thm}[\cite{Crainic-Ivan-Schaetz}]\label{rigidity theorem for Lie subalgebras}
	Let $\got{h}\subset\got{g}$ be a Lie subalgebra. If $H^1_{\rm def}(\got{h}\subset\got{g})=0$, then $\got{h}$ is a rigid subalgebra.
\end{thm}

A Lie subalgebra $\got{h}\subset(\got{g},\mu_{\got g})$ is \emph{stable} if all nearby Lie algebra structures on $\got{g}$ have a nearby Lie subalgebra isomorphic to $\got{h}$.

\begin{defin}\label{defin of stability of Lie subalgebras}
	A Lie subalgebra $\got{h}\subset(\got{g},\mu_{\got g})$ is called \textbf{a stable subalgebra} if for every neighborhood $U_{\got{h}}\subset\Gr_{k}(\got{g})$ of $\got{h}$, there exists a neighborhood $V_{\mu_{\got g}}\subseteq\mathbf \Lambda$ of $\mu_{\got g}$ such that for every $\mu'\in V_{\mu_{\got g}}$ there exists $\got{h}'\in U_{\got{h}}$ such that $\got{h}'\subset(\got{g},\mu')$ is a Lie subalgebra.
\end{defin}

\begin{thm}[\cite{Crainic-Ivan-Schaetz}]\label{stability theorem for Lie subalgebras}
	Let $\got{h}\subset\got{g}$ be a Lie subalgebra. If $H^2_{\rm def}(\got{h}\subset\got{g})=0$, then $\got{h}$ is a stable subalgebra. In addition, the space of $k$-dimensional Lie subalgebras of $\got{g}$ is then locally a manifold, around $\got{h}$, of dimension equal to the dimension of $Z^1_{\rm def}(\got{h}\subset\got{g})$.
\end{thm}

\section{Infinitesimal deformations of an ideal in a Lie algebra}\label{def_coh_ideal1}
Consider again Definition \ref{def_def_subalgebras} and Proposition \ref{ass_cocycle_deformations_subalgebras}
in the case where $\got{h}\coloneqq\got{i}$ happens not to be just a Lie subalgebra of the Lie algebra $\got g$, but also an ideal of $\got{g}$. In this case  the Bott representation vanishes, so $H^1_{\rm def}(\got i\subset \got g)$ equals $[\got i, \got i]^\circ\otimes \got g/\got i$, and more generally, the deformation complex of the ideal as a Lie subalgebra
is then $(C^\bullet_{\rm def}(\got i\subset \got g), \delta_{\got i}^{\rm Bott})=(C^\bullet(\got{i};\got{g}/\got{i}),\delta_{\tr})$, with the trivial representation  $\tr\colon \got{i}\to\mathfrak{gl}(\got{g}/\got{i})$, i.e.~the zero linear map. This section describes the appropriate cohomology associated to deformations of an ideal \emph{as an ideal} and not as a mere Lie subalgebra.

\subsection{Smooth deformations and infinitesimal deformations of an ideal}\label{def_coh_ideal}
A Lie ideal $\got i$ of a Lie algebra $\got g$ comes with two natural representations:
\begin{enumerate}[label=(\roman*)]
	\item $\ad^\got{i}\colon \got{g}\to\got{gl}(\got{i})$, \quad $\ad^\got{i}_x(u)=[x,u]$ for all $x\in \got g$ and all $u\in \got i$, and 
	\item $\ad^{\got{g}/\got{i}}\colon \got{g}\to\got{gl}(\got{g}/\got{i}), \quad {\ad}^{\got{g}/\got{i}}_x(\overline{y})=\overline{[x,y]}=[\overline{x},\overline{y}]$ for all $x,y\in\got g$.
\end{enumerate}
Consider the associated $\Hom$-representation on $\Hom(\got{i},\got{g}/\got{i})=\got i^*\otimes \got g/\got i$:
\begin{equation*}
	\ad^{\Hom}\colon \got{g}\otimes\Hom(\got{i},\got{g}/\got{i})\to\Hom(\got{i},\got{g}/\got{i}), \quad \ad^{\Hom}_x(\phi)(u)=\ad^{\got{g}/\got{i}}_x(\phi(u))-\phi(\ad^\got{i}_x(u)),
\end{equation*}
for all $x\in \got g$ and all $u\in \got i$.
Consider the corresponding Chevalley-Eilenberg complex $C^\bullet(\got{g};\got{i}^*\otimes\got{g}/\got{i})\coloneqq\Wedge^{\bullet}\got{g}^*\otimes\got{i}^*\otimes\got{g}/\got{i}$, with $\delta_{\got{g}\rhd\got{i}}^{\Hom}:=\delta_{\got g}^{\ad^{\Hom}}$. The obtained cohomology is written $H^\bullet_{\delta_{\got{g}\rhd\got{i}}^{\Hom}}(\got{g};\got{i}^*\otimes\got{g}/\got{i})=:H^\bullet(\got i\lhd \got g)$ for simplicity. Given $f\in C^k(\got{g};\got{i}^*\otimes\got{g}/\got{i})$, the Chevalley-Eilenberg differential is given explicitly by the following formula:
\begin{equation}\label{formula of the differential on the def. complex of ideals}
	\begin{split}
		\delta_{\got{g}\rhd\got{i}}^{\Hom}(f)(x_1,\dots,x_{k+1},u)&=\sum_{i=1}^{k+1}(-1)^{i+1}\ad^{\got{g}/\got{i}}_{x_i}(f(x_1,\dots, \widehat{x_i} ,\dots,x_{k+1},u))\\
		&+\sum_{i=1}^{k+1}(-1)^if(x_1,\dots, \widehat{x_i} ,\dots,x_{k+1},[x_i,u])\\
		&+\sum_{i<j}(-1)^{i+j}f([x_i,x_j],x_1,\dots, \widehat{x_i} ,\dots, \widehat{x_j} ,\dots,x_{k+1},u)
	\end{split}
\end{equation}
for $x_1,\ldots, x_{k+1}\in\got g$ and $u\in \got i$.

\begin{defin}\label{lemma for association of a cocycle to a deformation of an ideal}
	A smooth deformation $(\got{i}_t)_{t\in I}$ of an ideal $\got{i}$ inside $\got{g}$ is a smooth curve $\tilde{\got{i}}\colon [0,1]\to\Gr_{k}(\got{g})$ such that $\got{i}_0=\got{i}$ and $\got{i}_t\coloneqq\tilde{\got{i}}(t)$ is an ideal of $\got{g}$ for all $t\in I\coloneqq [0,1]$. 
\end{defin}
Because ideals are $\Ad$-invariant, 
the equivalence relation in Definition \ref{def_def_subalgebras} would here become trivial: in the sense of this definition, two equivalent smooth deformations of an ideal as an ideal (not only as a Lie subalgebra) would be equal. Accordingly, two deformations of an ideal are \emph{equivalent} if they are equal -- there is no room here for a more permissive relation of equivalence of ideals.

\begin{prop}\label{def_cocycle_ideal}
	Let $\got g$ be a Lie algebra and let $\got i\subseteq \got g$ be an ideal. If $(\got{i}_t)_{t\in I}$ is a smooth deformation of the ideal $\got{i}$ inside $\got{g}$, then
	\begin{equation}\label{the velocity of a curve of ideals}
		\fracddtz\got{i}_t\in T_{\got{i}}\Gr_k(\got{g})
	\end{equation}
	is a $0$-cocycle in $ C^\bullet(\got{g};\got{i}^*\otimes\got{g}/\got{i})$.
\end{prop}
\begin{defin}
	Because of this result, cohomology classes in $H^0(\got g; \got i^*\otimes \got g/ \got i)$ are called \textbf{infinitesimal deformations} of $\got i\lhd\got{g}$. An infinitesimal deformation of the form $\left[\fracddtzs\got{i}_t\right]\in H^0(\got{g};\got{i}^*\otimes\got{g}/\got{i})$ with a deformation $(\got{i}_t)_{t\in I}$ of the ideal $\got{i}$ inside $\got{g}$ is called the \textbf{deformation class} associated to $(\got{i}_t)_{t\in I}$.
	
	Instead of the term \emph{deformation class}, the expression \textbf{deformation cocycle} associated to $(\got{i}_t)_{t\in I}$ is used here as well, since $H^0(\got{g};\got{i}^*\otimes\got{g}/\got{i})=Z^0(\got{g};\got{i}^*\otimes\got{g}/\got{i})$.
\end{defin}

This paper shows later on that not every infinitesimal deformation of $\got{i}\lhd\got{g}$ may be regarded as a deformation class of some $(\got{i}_t)_{t\in I}$.
\begin{proof}[Proof of Proposition \ref{def_cocycle_ideal}]
	First recall that $\fracddtzs\got{i}_t$ is a linear map from $\got{i}$ to $\got{g}/\got{i}$ using the canonical identification $T_{\got{i}}\Gr_k(\got{g})\simeq\got{i}^*\otimes\got{g}/\got{i}$ (see Appendix \ref{tangent_grk}). In addition identify $\fracddtzs{\got{i}_t}=\pi_{\got{g}/\got{i}}\circ\fracddtzs\alpha(t)\arrowvert_{\got i}$, where the curve\footnote{$\alpha$ might be defined on a small neighborhood of $0$ only.} $\alpha\colon I\ni t\mapsto \alpha(t)\in\GL(\got{g})$ is such that $\alpha(t)(\got{i})=\got{i}_t$ for each $t\in I$, and where $\pi_{\got{g}/\got{i}}\colon \got{g}\to\got{g}/\got{i}$ is the canonical projection. 
	
	For each $t\in I$ the map $\pi_{\got{g}/\got{i}_t}$ has kernel $\got i_t$ and factors hence to an isomorphism $\overline{\alpha(t)}\colon \got g/\got i\to \got g/\got i_t$ such that 
	\begin{equation}\label{commutative diagram in the proof of the cocycle in Lie subalgebras}
		\begin{tikzcd}
			\got{g}\arrow[r,"\alpha(t)"] \arrow[d,swap,"\pi_{\got{g}/\got{i}}"] &
			\got{g} \arrow[d,"\pi_{\got{g}/\got{i}_t}"]\\
			\got{g}/\got{i} \arrow[r,"\overline{\alpha(t)}"] & \got{g}/\got{i}_t
		\end{tikzcd}
	\end{equation}
	commutes.
	The fact that $\got{i}_t\subset\got{g}$ is an ideal for each $t\in I$ means that $[\got{g},\got{i}_t]\subset \got{i}_t$. Consider the smooth map $\sigma\colon I\to\Hom(\got{g}, \got{i}^*\otimes\got{g}/\got{i})$ defined, for any $t\in I$, $x\in\got{g}$ and $u\in\got{i}$, by
	\begin{equation*}
		\sigma(t)(x,u)=\overline{\alpha(t)}^{-1}\circ\pi_{\got{g}/\got{i}_t}([x,\alpha(t)(u)])\stackrel{\eqref{commutative diagram in the proof of the cocycle in Lie subalgebras}}{=}\pi_{\got{g}/{\got{i}}}\circ{\alpha(t)}^{-1}([x,\alpha(t)(u)]).
	\end{equation*}
	$\sigma$ vanishes at $t\in I$ if and only if $\got i_t$ is an ideal. 
	Differentiating with respect to $t$ the  expression $\sigma=0$ yields
	\begin{equation}\label{equation of differentiation to appear the cocycle for deform of ideals}
		0=\biggleftpar\fracddtz\sigma(t)\biggrightpar(x,u)=-\pi_{\got{g}/\got{i}}\circ\biggleftpar\fracddtz{\alpha(t)}\biggrightpar\bigleftpar[x, u]\bigrightpar+\pi_{\got{g}/\got{i}}\biggleftpar\biggleftbra x,\biggleftpar\fracddtz{\alpha(t)}\biggrightpar(u)\biggrightbra\biggrightpar,
	\end{equation}
	for all $x\in \got g$ and $u\in \got i$.
	This is the cocycle condition for the linear map $\pi_{\got{g}/\got{i}}\circ\fracddtzs{\alpha(t)}\arrowvert_{\got i}\colon \got{i}\to\got{g}/\got{i}$ in $ C^0(\got{g};\got{i}^*\otimes\got{g}/\got{i})$. 
\end{proof}

\subsection{Relation with other deformation cohomologies induced by ideals in Lie algebras}\label{rel_other_coh}
This section compares the deformation cohomology and deformation classes of ideals with several deformation cohomology and deformation classes associated to ideals in Lie algebras:
\begin{enumerate}
	\item The ideal $\got i$ of $\got g$ is a subrepresentation of the adjoint representation of $\got g$, which can be deformed as a subrepresentation, i.e.~as a morphism of Lie algebras $\got g\to \got{gl}(\got g,\got i)$, where $\got{gl}(\got{g},\got{i})$ is the Lie subalgebra of $\got{gl}(\got{g})$ consisting of all the endomorphisms of $\got{g}$ preserving the vector subspace $\got{i}\lhd\got{g}$.
	\item The ideal $\got i$ of $\got g$ defines the morphism $\pi_{\got g/\got i}\colon \got g\to \got g/\got i$ of Lie algebras, which can be deformed as such a morphism, see Remark \ref{rem_def_morphisms}.
	\item The Lie algebra $\got g/\got i$ defined by $\got i$ can be deformed as a Lie algebra.
	\item This section considers as well Nijenhuis and Richardson's deformation class \cite{Nijenhuis-Richardson-Deformations-of-Lie-algebra-structures-67'} associated to the ideal $\got i\subseteq \got g$.
	\item Finally, the deformation class of $\got i$ as a mere \emph{Lie subalgebra} of $\got g$ can be considered.
\end{enumerate}

For the convenience of the reader, the adjoint action of $\got g/\got i$ on itself is written $\overline{\ad}$ in this section. The adjoint action of $\got g$ on itself is  simply written $\ad$. It is easy to see that $\pi_{\got{g}/\got{i}}^*\overline{\ad}=\ad^{\got{g}/\got{i}}$.

\subsubsection{Relation with the deformation cohomology of the ideal $\got i$ as a subrepresentation of the adjoint of $\got g$}\label{def_subrep_adjoint_subsec}

Denote by $\got{gl}(\got{g},\got{i})$ the Lie subalgebra of $\got{gl}(\got{g})$ consisting of all the endomorphisms of $\got{g}$ preserving the vector subspace $\got{i}\lhd\got{g}$. Since the adjoint representation $\ad^{\got g}\colon \got g\to \got{gl}(\got g)$ of $\got g$ preserves $\got i$, it has image in $\got{gl}(\got g, \got i)$.
The natural inclusion $i_{\got{gl}(\got{g})}\colon \got{gl}(\got{g},\got{i})\hookrightarrow\got{gl}(\got{g})$ induces  the following linear map for all $k\geq 0$:
\begin{equation*}
	(i_{\got{gl}(\got{g})})_*\colon  C^k(\got{g}; \got{gl}(\got g, \got{i}))\hookrightarrow\underbrace{ C^{k}(\got{g}; \got{gl}(\got{g}))}_{C_{\defor}^k(\ad^{\got{g}})}, \quad (i_{\got{gl}(\got{g})})_*(\phi)(x_1,\dots,x_k)=i_{\got{gl}(\got{g})}(\phi(x_1,\dots,x_k)).
\end{equation*}
The representation in $ C^{\bullet}(\got{g}; \got{gl}(\got{g}))$, the deformation complex of the Lie algebra morphism $\ad^{\got{g}}\colon \got{g}\to\got{gl}(\got{g})$, is $r\eqdef(\ad^{\got{g}})^*(\ad^{\got{gl}(\got{g})})$ and so its differential $\delta_{\ad^{\got{g}}}$ is given by 
\begin{align*}
	\delta_{\ad^{\got{g}}}(\omega)(x_1,\dots,x_{k+1})(y)&=\sum_{i=1}^{k+1}(-1)^{i+1}\left[x_i,\omega(x_1,\dots, \widehat{x_i} ,\dots,x_{k+1})(y)\right]\\
	&+\sum_{i=1}^{k+1}(-1)^i\omega\left(x_1,\dots, \widehat{x_i} ,\dots,x_{k+1}\right)([x_i,y])\\
	&+\sum_{i<j}(-1)^{i+j}\omega\left([x_i,x_j],x_1,\dots, \widehat{x_i} ,\dots, \widehat{x_j} ,\dots,x_{k+1}\right)(y)
\end{align*}
for $\omega\in C^k(\got{g}; \got{gl}(\got{g}))$ and $x_1,\dots,x_{k+1},y\in\got{g}$.
This differential restricts to the subspace $ C^\bullet(\got{g}; \got{gl}(\got g, \got{i}))$ of $ C^{\bullet}(\got{g}; \got{gl}(\got{g}))$ and $(i_{\got{gl}(\got{g})})_*\colon \bigleftpar C^{\bullet}(\got{g};\got{gl}(\got{g},\got{i})), \delta_{\ad^{\got{g}}}\bigrightpar\hookrightarrow\bigleftpar C^{\bullet}(\got{g};\got{gl}(\got{g})), \delta_{\ad^{\got{g}}}\bigrightpar$ is a cochain map.

\bigskip

The category of cochain complexes is an abelian category and so kernels and cokernels exist in it. If $\Phi\colon (B^{\bullet}, d_B)\to (A^{\bullet}, d_A)$ is a cochain map, its kernel $\Ker^{\bullet}(\Phi)$, its image $\img^{\bullet}(\Phi)$ and its cokernel $\Coker^{\bullet}(\Phi)\eqdef B^{\bullet}/\img^{\bullet}(\Phi)$ complexes are defined degree-wise. Denote by $p_{\Coker}\colon (A^{\bullet}, d_A)\twoheadrightarrow(\Coker^{\bullet}(\Phi), \overline{d_A})$, the canonical projection, where $\overline{d_A}$ is exactly defined such that $p_{\Coker}$ is a cochain map: $\overline{d_A}\circ p_{\Coker}=p_{\Coker}\circ d_A$.

\begin{prop}
	The deformation complex $ C^{\bullet}(\got{g}; \got{i}^*\otimes\got{g}/\got{i})$ of an ideal $\got{i}\lhd\got{g}$ fits into the following short exact sequence of cochain complexes
	\begin{equation*}
		\begin{tikzcd}
			(C^{\bullet}(\got{g};\got{gl}(\got{g},\got{i})), \delta_{\ad^{\got{g}}}) \ar[rr, hookrightarrow, "(i_{\got{gl}({\got{g}})})_*"] && \bigleftpar C^{\bullet}(\got{g}; \got{gl}(\got{g})), \delta_{\ad^{\got{g}}} \bigrightpar \ar[rr, twoheadrightarrow] && \bigleftpar C^{\bullet}(\got{g}; \got{i}^*\otimes\got{g}/\got{i}), \delta_{\got{g}\rhd\got{i}}^{\Hom}\bigrightpar.
		\end{tikzcd}
	\end{equation*}
\end{prop}
\begin{proof}
	It is sufficient to check that the cokernel complex of the inclusion $(i_{\got{gl}({\got{g}})})_*$ is isomorphic to the deformation complex $ C^{\bullet}(\got{g};\got{i}^*\otimes\got{g}/\got{i})$. Consider, for any $k\geq 0,$ the surjective linear map
	\begin{equation*}
		\pi_*\colon  C^k(\got{g};\got{gl}(\got{g}))\to C^k(\got{g}; \got{i}^*\otimes\got{g}/\got{i}), \quad \pi_*(\omega)(x_1,\dots,x_k)=\pi_{\got{g}/\got{i}}\circ\omega(x_1,\dots,x_k)|_{\got{i}}.
	\end{equation*}
	The following direct computation shows that $\pi_*$ is a cochain map:
	\begin{align*}
		\pi_*\bigleftpar\delta_{\ad^{\got{g}}}(\omega)\bigrightpar(x_1,\dots,x_{k+1})(u)&=\pi_{\got{g}/\got{i}}	\bigleftpar\delta_{\ad^{\got{g}}}(\omega)(x_1,\dots,x_{k+1})(u)\bigrightpar\\
		&=\sum_{i=1}^{k+1}(-1)^{i+1}\pi_{\got{g}/\got{i}}\bigleftpar[x_i,\omega(x_1,\dots, \widehat{x_i} ,\dots,x_{k+1})(u)]\bigrightpar\\
		&+\sum_{i=1}^{k+1}(-1)^i\pi_{\got{g}/\got{i}}\bigleftpar\omega(x_1,\dots, \widehat{x_i} ,\dots,x_{k+1})([x_i,u])\bigrightpar\\
		&+\sum_{i<j}(-1)^{i+j}\pi_{\got{g}/\got{i}}\bigleftpar\omega([x_i,x_j],x_1,\dots, \widehat{x_i} ,\dots, \widehat{x_j} ,\dots,x_{k+1})(u)\bigrightpar\\
		&=\sum_{i=1}^{k+1}(-1)^{i+1}\ad^{\got{g}/\got{i}}_{x_i}\bigleftpar\pi_*(\omega)(x_1,\dots, \widehat{x_i} ,\dots,x_{k+1})(u)\bigrightpar\\
		&+\sum_{i<j}(-1)^{i+j}\pi_*(\omega)\bigleftpar[x_i,x_j],x_1,\dots, \widehat{x_i} ,\dots, \widehat{x_j} ,\dots,x_{k+1}\bigrightpar(u)\\
		&+\sum_{i=1}^{k+1}(-1)^i\pi_*(\omega)\bigleftpar x_1,\dots, \widehat{x_i} ,\dots,x_{k+1}\bigrightpar([x_i,u])\\
		&=\delta_{\got{g}\rhd\got{i}}^{\Hom}\bigleftpar\pi_{*}(\omega)\bigrightpar(x_1,\dots,x_{k+1},u)
	\end{align*}
	for $x_1,\ldots, x_{k+1}\in\got g$ and $u\in\got i$.
	The kernel complex of $\pi_*$ is exactly the complex $ C^{\bullet}(\got{g};\got{gl}(\got{g},\got{i}))$. The following commutative diagram is then automatically a commutative diagram of cochain complexes, that yields the desired isomorphism.
	\begin{equation*}
		\begin{tikzcd}
			\bigleftpar C^\bullet(\got{g};\got{gl}(\got{g})), \delta_{\ad^{\got{g}}}\bigrightpar \arrow[r, "\pi_*"] \ar[d, swap, "p_{\Coker}"] & \bigleftpar C^\bullet(\got{g}; \got{i}^*\otimes\got{g}/\got{i}), \delta_{\got{g}\rhd\got{i}}^{\Hom}\bigrightpar\\
			\biggleftpar\frac{ C^{\bullet}(\got{g};\got{gl}(\got{g}))}{ C^{\bullet}(\got{g};\got{gl}(\got{g},\got{i}))}, \overline{\delta_{\ad^{\got{g}}}}\biggrightpar \ar[ur, dashed , swap, "\simeq"] 
\end{tikzcd} \end{equation*}\end{proof}

\subsubsection{Relation with the deformation cohomology of the canonical projection $\got g\to \got g/\got i$}
The ideal $\got i$ of $\got g$ defines the quotient Lie algebra $\got g/\got i$ and the canonical projection $\pi_{\got g/\got i}\colon \got g\to \got g/\got i$, which is an epimorphism of Lie algebras.
\medskip

Recall here that the shift by $k$ of a cochain complex $(A^\bullet, d_A)$ is $(A^\bullet[k], (-)^kd_A)$, by convention. In this section, several cochain complexes are shifted by $1$, giving a minus sign to the `shifted' differentials. \textbf{Note that whenever a cochain complex $(A^\bullet=\oplus_{k\geq 0}A^k[-k], d_A)$  is shifted by 1 in this section, the shift by 1 is understood to be truncated to non-negative degrees; i.e.~$A^\bullet[1]$ is understood to be $\oplus_{k\geq 1}A^k[-k+1]$.}

\begin{prop}\label{the lemma where we defined the cochain map from def. complex of can projection to the def. complex of ideals}
	There is a  natural cochain map from the deformation complex of the canonical projection $\pi_{\got{g}/\got{i}}\colon \got{g}\to\got{g}/\got{i}$, as a Lie algebra morphism, to the deformation complex of the ideal $\got{i}\lhd\got{g}$, defined by:
	\begin{equation*}
		\Pi\colon (C^{\bullet}(\got{g};\got{g}/\got{i})[1], -\delta_{\pi_{\got{g}/\got{i}}})\to (C^{\bullet}(\got{g};\got{i}^*\otimes\got{g}/\got{i}), \delta_{\got{g}\rhd\got{i}}^{\Hom})\end{equation*}
	\begin{equation*}
		C^{k+1}(\got{g};\got{g}/\got{i})\ni  \phi\mto (-1)^{k+1}\phi\arrowvert_{(\wedge^{k}\got{g})\wedge\got{i}}
	\end{equation*}
	for all $k\geq 0$,
	which in addition sends deformation cocycles to deformation cocycles.
\end{prop}

Denote the image complex of $\Pi$ by $ C_{\wedge}^{\bullet}(\got{g};\got{i}^*\otimes\got{g}/\got{i})\eqdef\{\phi|_{\wedge^{\bullet}\got{g}\wedge\got{i}} \ | \ \phi\in\wedge^{\bullet+1}\got{g}^*\otimes\got{g}/\got{i}\}$ and the restriction of $\Pi$ to its image by $\Pi_{\wedge}$. Note that there is no better description of this complex -- it needs to be defined as the image of the map $\Pi$.

\begin{proof}
	Choose $\phi\in\wedge^{k}\got{g}^*\otimes \got{g}/\got{i}$. Then the following computation shows that $\Pi\circ(-\delta_{\pi_{\got{g}/\got{i}}})(\phi)=\delta_{\got{g}\rhd\got{i}}^{\Hom}\circ\Pi(\phi)$ and so $\Pi$ is a cochain map.
	Compute for $x_1, \ldots, x_{k}\in\got g$ and $x_{k+1}\in \got i$
	\begin{align*} 
		&(-1)^{k+1}(-\delta_{\pi_{\got{g}/{\got{i}}}}(\phi))|_{\wedge^{k}\got{g}\wedge\got{i}}(x_1,\dots,x_{k}, \underbrace{x_{k+1}}_{\in\got{i}})\\
		&=(-1)^{k}\sum_{i=1}^{k+1}(-1)^{i+1}(\pi_{\got{g}/\got{i}}^*\overline{\ad})_{x_i}(\phi(x_1,\dots, \widehat{x_i} ,\dots,x_{k+1}))\\
		&\quad +(-1)^{k}\sum_{1\leq i<j\leq k+1}(-1)^{i+j}\phi([x_i,x_j],x_1,\dots, \widehat{x_i} ,\dots, \widehat{x_j} ,\dots,x_k, ,x_{k+1})\\
		&=(-1)^{k}\sum_{i=1}^{k}(-1)^{i+1}\ad^{\got{g}/\got{i}}_{x_i}(\phi(x_1,\dots, \widehat{x_i} ,\dots,x_k, x_{k+1}))\\
		&\quad +(-1)^{k}\sum_{1\leq i<j\leq k}(-1)^{i+j}\phi([x_i,x_j],x_1,\dots, \widehat{x_i} ,\dots, \widehat{x_j} ,\dots,x_k, x_{k+1})\\
		&\quad +(-1)^{k}\sum_{i=1}^{k}(-1)^{i}\phi(x_1,\dots, \widehat{x_i} ,\dots,x_k,[x_i,x_{k+1}])\\
		&=(-1)^{k}\delta_{\got{g}\rhd\got{i}}^{\Hom}(\phi|_{\wedge^{k-1}\got{g}\wedge\got{i}})(x_1,\dots,x_{k}, x_{k+1})\\
		&=\delta_{\got{g}\rhd\got{i}}^{\Hom}((-1)^{k}\phi|_{\wedge^{k-1}\got{g}\wedge\got{i}})(x_1,\dots,x_{k}, x_{k+1}),
	\end{align*}
	where the second equality uses that $\pi_{\got{g}/\got{i}}^*\overline{\ad}=\ad^{\got{g}/\got{i}}$ and $\ad^{\got{g}/\got{i}}_{x_{k+1}}=0$. 
	For the second claim, let $\widetilde{\pi_{\got{g}/\got{i}}}\colon I\times\got{g}\to\got{g}/\got{i}$ with $\widetilde{\pi_{\got{g}/\got{i}}}(t,\cdot)\eqqcolon\pi^t_{\got{g}/\got{i}}$ be a smooth deformation of the canonical projection $\pi^0_{\got{g}/\got{i}}\eqdef\pi_{\got{g}/\got{i}}$. 
	Note that the dimension of $\Ker(\pi_t)$ can possibly vary but since the surjectivity of $\pi_{\got g/\got i}$ is an open condition, at least for $t$ in a sufficiently small neighborhood $J\subset I$ of $0$, the rank of $\pi_{\got g/\got i}^t$ remains maximal. Assume hence without loss of generality that the deformation of $\pi_{\got g/\got i}$ is deformed by surjective Lie algebra morphisms.
	The induced deformation of $\got{i}\eqdef\Ker(\got{\pi_{\got{g}/\got{i}}})$ as an ideal is given by $(\got i_t:=\Ker(\pi^t_{\got{g}/\got{i}}))_{t\in I}$ and the smoothness of this curve in the Grassmannian is guaranteed by the smoothness of $\widetilde{\pi_{\got{g}/\got{i}}}$. 
	
	Pick a smooth curve of linear vector space isomorphisms $(\alpha_t)_{t\in J\subset I}$ in $\GL(\got{g})$ with $\alpha_0=\Id$  and $\alpha_t(\got{i})=\got{i}_t$ for all $t\in J$ and consider the linear map $\pi_{\got{g}/\got{i}}^t\circ\alpha_t\colon \got{g}\to\got{g}/\got i$. Differentiating the equation $\pi_{\got{g}/\got{i}}^t\circ\alpha_t|_{\got{i}}=0$ yields  $\fracddtzs{\pi^t_{\got{g}/\got{i}}}(u)=\pi_{\got{g}/\got{i}}\circ\fracddtzs{(-\alpha_t)}(u)$ for all $u\in\got{i}$.
	The left-hand side is the deformation cocycle associated to the deformation of $\pi_{\got g/\got i}$, see Remark \ref{rem_def_morphisms}, and the right-hand side is, up to the minus sign, the deformation cocycle associated to the deformation of $\got i$, see Proposition \ref{def_cocycle_ideal}.
	Since $\Pi$ sends $\phi\in \wedge^1\got g^*\otimes\got g/\got i$ to $-\phi\arrowvert_{\got i}\in \wedge^0\got g^*\otimes \got i^*\otimes \got g/\got i$, this completes the proof of the second statement.
\end{proof}

\begin{rem}
	One could hope that the natural inclusion \[i\colon C_{\wedge}^{\bullet}(\got{g};\got{i}^*\otimes\got{g}/\got{i})\hookrightarrow C^{\bullet}(\got{g};\got{i}^*\otimes\got{g}/\got{i})\] is a quasi-isomorphism. However this is not the case in general: Let $(\got{g},\mu_{\got{g}}=0)$ be an abelian Lie algebra. Then any vector subspace of $\got{g}$ is an ideal in $\got g$. Here $\delta_{\got{g}\rhd\got{i}}^{\Hom}=0$, so the cohomologies are identical to the complexes themselves and therefore cannot be isomorphic. 
	
	However,  in degree $0$, the vector spaces $C^0_\wedge(\got g, \got i^*\otimes \got g/\got i)$ and $C^0(\got g, \got i^*\otimes \got g/\got i)$ are equal, so also the $0$-cohomologies of the two complexes are equal.
\end{rem}

\subsubsection{Relation with the deformation cohomology of $\got g/\got i$}
Now the deformation complex of the ideal $\got i$ of $\got g$ is set in relation with the deformation complex $\left( C^{\bullet}(\got{g}/\got{i};\got{g}/\got{i}), \delta_{\overline{\ad}}\right)$ of the quotient Lie algebra $\got g/\got i$.
\begin{prop}\label{lemma proving the s.e.s. relating deform. of ideals, canonical projection, and quotient Lie algebra}
	The cochain complex $\left(C^{\bullet}_{\wedge}(\got{g};\got{i}^*\otimes\got{g}/\got{i}), \delta_{\got{g}\rhd\got{i}}^{\Hom}\right)$ fits into the following short exact sequence of cochain complexes:
	\begin{equation}\label{the s.e.s. of the def. complex of ideals with C(g,g/i)}
		\begin{tikzcd}
			\left( C^{\bullet}(\got{g}/\got{i};\got{g}/\got{i})[1], -\delta_ {\got{g}/\got{i}}^{\overline{\ad}}\right) \ar[r, hookrightarrow] & \left(C^{\bullet}(\got{g};\got{g}/\got{i})[1], -\delta_{\pi_{\got{g}/\got{i}}}\right)\ar[r, "\Pi_{\wedge}", twoheadrightarrow] & \left( C_{\wedge}^{\bullet}(\got{g};\got{i}^*\otimes\got{g}/\got{i}), \delta_{\got{g}\rhd\got{i}}^{\Hom}\right).
		\end{tikzcd}
	\end{equation}
	The first map is simply the inclusion $\wedge^\bullet \got i^\circ\otimes \got g/\got i\hookrightarrow \wedge^\bullet \got g^*\otimes \got g/\got i$, via the canonical isomorphism $\pi_{\got g/\got i}^*\colon (\got g/\got i)^*\longrightarrow\got i^\circ\subseteq \got g^*$.
\end{prop}
\begin{proof}
	The kernel complex of $\Pi_{\wedge}$ is given at degree $k$ by:
	\begin{align}\label{the kernel complex in the content of deform. of ideals}
		\Ker^{k+1}(\Pi_{\wedge})&=\left\{\phi\colon\wedge^{k+1}\got{g}\to\got{g}/\got{i} \ | \ \phi(x_1,\dots,x_i,\dots,x_{k+1})=0 \ \text{if} \ x_i\in\got{i} \ \text{for some} \ i\in\{1,\dots,k+1\}\right\}\nonumber\\
		&=\wedge^{k+1}\got{i}^\circ\otimes\got{g}/\got{i}\simeq C^{k+1}(\got{g}/\got{i};\got{g}/\got{i})
	\end{align}
	via the natural isomorphism
	\begin{equation}\label{the projection from the kernel complex to the def. complex of the quotient Lie algebra}
		\wedge^{k+1}\got i^\circ\otimes \got g/\got i\to C^{k+1}(\got{g}/\got{i};\got{g}/\got{i}), \quad \phi\mto \left(\overline{\phi}\colon (\overline{x_1},\ldots, \overline{x_{k+1}})\mapsto\phi(x_1,\ldots, x_{k+1})\right).
	\end{equation}
	Via this isomorphism, the inclusion $\left( C^{\bullet}(\got{g}/\got{i};\got{g}/\got{i}), \delta_{\got{g}/\got{i}}^{\overline{\ad}}\right)\hookrightarrow  \left( C^{\bullet}(\got{g};\got{g}/\got{i}), \delta_{\pi_{\got{g}/\got{i}}}\right) $ is a cochain map, as shown in the following computation.
	\begin{align*}
		\delta_{\got{g}/\got{i}}^{\overline{\ad}}(\overline{\phi})(\overline{x_1},\dots,\overline{x_{k+1}})&= \sum_{i=1}^{k+1}(-1)^{i+1}\overline{\ad}_{\overline{x_i}}(\overline{\phi}(\overline{x_1},\dots,  \widehat{\overline{x_i}} ,\dots,\overline{x_{k+1}}))\\
		&+\sum_{i<j}(-1)^{i+j}\overline{\phi}([\overline{x_i},\overline{x_j}],\overline{x_1},\dots,  \widehat{\overline{x_i}} ,\dots,  \widehat{\overline{x_j}} ,\dots,\overline{x_{k+1}})\\
		&=\sum_{i=1}^{k+1}(-1)^{i+1}(\pi_{\got{g}/\got{i}}^*\overline{\ad})_{x_i}(\phi(x_1,\dots,\widehat{x_i},\dots,x_{k+1}))\\
		&+\sum_{i<j}(-1)^{i+j}\phi([x_i,x_j],x_1,\dots,  \widehat{x_i} ,\dots,  \widehat{x_j} ,\dots,x_{k+1})\\
		&=\delta_{\pi_{{\got{g}/\got{i}}}}(\phi)(x_1,\dots,x_{k+1})
	\end{align*}
	for $\phi\in \wedge^{k}\got{i}^\circ\otimes\got{g}/\got{i}$ and $x_1,\ldots, x_{k+1}\in \got g$.
\end{proof}

\subsubsection{Relation with Nijenhuis and Richardson's deformation cohomology of an ideal}\label{rel_nij_ric}
Let here $ C_{\got{i}}^\bullet(\got{g};\got{g})$ be the graded vector space defined at degree $k$ by
\begin{equation*}
	C_{\got{i}}^k(\got{g};\got{g})\eqdef\left\{\phi\colon \wedge^k\got{g}\to\got{g}\ | \ \phi(x_1,\dots,x_i,\dots,x_k)\in\got{i}\  \text{if} \ x_i\in\got{i} \ \text{for some}\  i\in\{1,\dots,k\}\right\},
\end{equation*}
and equip it with the usual differential $\delta_{\got g}^{\rm ad}$, which preserves $ C_{\got{i}}^\bullet(\got{g};\got{g})$.

The obtained complex $ C_{\got{i}}^{\bullet}(\got{g};\got{g})$ appeared in \cite{Nijenhuis-Richardson-Deformations-of-Lie-algebra-structures-67'} as the complex controlling deformations of the Lie algebra $(\got{g}, \mu_{\got{g}})$, such that $\got{i}$ remains an ideal in it. 

Recall the differential graded Lie algebra structure on $ C^\bullet(\got{g};\got{g})[1]$ defined in Section \ref{sec_LA_MC} by the Lie algebra $\got g$.

\begin{prop}\label{proposition/definition for/of the complex related to deform of ideals in Nij-Rich}
	Let $\got g$ be a Lie algebra and $\got i$ be an ideal in $\got g$. The graded vector space $ C^{\bullet}_{\got{i}}(\got{g};\got{g})[1]$ is a differential graded Lie subalgebra of $ C^\bullet(\got{g};\got{g})[1]$ and the linear map
	\begin{equation}\label{def_pi_star}
		\overline{\pi_*}\colon C_{\got{i}}^\bullet(\got{g};\got{g})[1]\to C^\bullet(\got{g}/\got{i};\got{g}/\got{i})[1], \qquad \phi\mto\overline{\pi_{\got{g}/\got{i}}\circ\phi} \end{equation}
	with $\overline{\pi_{\got{g}/\got{i}}\circ\phi}\in (C^\bullet(\got{g}/\got{i};\got{g}/\got{i})[1])_k
	$ defined as in \eqref{the projection from the kernel complex to the def. complex of the quotient Lie algebra} by 
	\[ \overline{\pi_{\got{g}/\got{i}}\circ\phi} (\overline{x_1}, \ldots, \overline{x_{k+1}})\eqdef(\pi_{\got{g}/\got{i}}\circ\phi)(x_1, \ldots, x_{k+1}),
	\]
	for all $\phi\in C^{k+1}_{\got i}(\got g, \got g)$, $x_1, \ldots, x_{k+1}\in\got{g}$,
	is a dgLa morphism\footnote{This map can be extended to the whole complex $C^\bullet_{\got i}(\got g; \got g)[1]$, i.e.~is not only defined on non-negative degrees. The space $(C^\bullet_{\got i}(\got g; \got g)[1])_{-1}$ is $C^0_{\got i}(\got g;\got g)=\got i$. The map $\overline{\pi_*}$ is simply zero on this space.}. Furthermore, the cokernel complex $\frac{ C^{\bullet}(\got{g};\got{g})[1]}{ C^{\bullet}_{\got{i}}(\got{g};\got{g})[1]}$ of the inclusion $ C^{\bullet}_{\got{i}}(\got{g};\got{g})[1]\hookrightarrow C^{\bullet}(\got{g};\got{g})[1]$ is isomorphic to $ C_{\wedge}^{\bullet}(\got{g};\got{i}^*\otimes\got{g}/\got{i})$ via the linear map
	\begin{equation}\label{iso_coker_pi_complex}
		\overline{\Pi}\colon  \overline{\phi}\mto (-1)^k(\pi_{\got{g}/\got{i}}\circ\phi)|_{\wedge^{k-1}\got{g}\wedge\got{i}}.
	\end{equation}
\end{prop}
\begin{proof}
	The fact that the differential $\delta_{\got{g}}^{\ad}$ and the Gerstenhaber bracket $\llbra\cdot,\cdot\rrbra$ both restrict to $ C^{\bullet}_{\got{i}}(\got{g};\got{g})[1]$ is a direct observation of their formulas. Recalling the complex $\Ker^{\bullet}(\Pi_{\wedge})$ from \eqref{the kernel complex in the content of deform. of ideals}, for $k\geq 1$ the map $\overline{\pi_*}\colon C_{\got{i}}^k(\got{g};\got{g})\to C^k(\got{g}/\got{i};\got{g}/\got{i})$ factors as
	\begin{equation*}
		C_{\got{i}}^k(\got{g};\got{g})\to\Ker^k(\Pi_{\wedge})\overset{\sim}{\longrightarrow} C^k(\got{g}/\got{i};\got{g}/\got{i}), \quad \phi\mto\pi_{\got{g}/\got{i}}\circ\phi\mto\overline{\pi_{\got{g}/\got{i}}\circ\phi}.
	\end{equation*}
	Since the canonical identification $\Ker^k(\Pi_{\wedge})=\wedge^{k}\got i^\circ\otimes \got g/\got i\overset{\sim}{\longrightarrow}  C^k(\got{g}/\got{i};\got{g}/\got{i})$ is a cochain map, it is enough to check that also the map $ C_{\got{i}}^k(\got{g};\got{g})\to\Ker^k(\Pi_{\wedge})$ is a cochain map. Compute
	\begin{equation*}
		\begin{split}
			\delta_{\pi_{\got{g}/\got{i}}}(\pi_{\got{g}/\got{i}}\circ\phi)(x_1,\dots,x_{k+1})&=\sum_{i=1}^{k+1}(-1)^{i+1}(\pi_{\got{g}/\got{i}}^*\overline{\ad})_{x_i}(\pi_{\got{g}/\got{i}}\circ\phi(x_1,\dots, \widehat{x_i} ,\dots,x_{k+1}))\\
			&+\sum_{i<j}(-1)^{i+j}\pi_{\got{g}/\got{i}}\circ\phi([x_i,x_j],x_1,\dots, \widehat{x_i} ,\dots, \widehat{x_j} ,\dots,x_{k+1})\\
			&=\sum_{i=1}^{k+1}(-1)^{i+1}\pi_{\got{g}/\got{i}}([x_i,\phi(x_1,\dots, \widehat{x_i} ,\dots,x_{k+1})])\\
			&+\sum_{i<j}(-1)^{i+j}\pi_{\got{g}/\got{i}}\circ\phi([x_i,x_j],x_1,\dots, \widehat{x_i} ,\dots, \widehat{x_j} ,\dots,x_{k+1})\\
			&=(\pi_{\got{g}/\got{i}}\circ\delta^{\ad}_{\got{g}}(\phi))(x_1,\dots,x_{k+1})
		\end{split}
	\end{equation*}
	for $k\in\mathbb N$, $\phi\in C^k_{\got i}(\got g; \got g)$ and  $x_1,\ldots, x_{k+1}\in \got g$.
	
	Next check that  $\overline{\pi_*}\colon C_{\got{i}}^\bullet(\got{g};\got{g})[1]\to  C^\bullet(\got{g}/\got{i};\got{g}/\got{i})[1]$ respects as well the Lie brackets. Take  $\phi\in (C_{\got{i}}^\bullet(\got{g};\got{g})[1])_k$ and $\psi\in (C_{\got{i}}^\bullet(\got{g};\got{g})[1])_l$ and compute
	\begin{align*}
		&\left\llbra\overline{\pi_{\got{g}/\got{i}}\circ\phi},\overline{\pi_{\got{g}/\got{i}}\circ\psi}\right\rrbra(\overline{x_1},\ldots, \overline{x_{k+l+1}})\\
		=\,&(-1)^{kl}\!\!\!\!\sum_{\tau\in S_{(l+1,k)}}\overline{\pi_{\got{g}/\got{i}}\circ\phi}\left(\overline{\pi_{\got{g}/\got{i}}\circ\psi}\left(\overline{x_{\tau(1)}},\dots,\overline{x_{\tau(l+1)}}\right),\overline{x_{\tau(l+2)}},\dots\overline{x_{\tau(k+l+1)}}\right)\\
		&-\sum_{\tau\in S_{(k+1,l)}}\overline{\pi_{\got{g}/\got{i}}\circ\psi}\left(\overline{\pi_{\got{g}/\got{i}}\circ\phi}\left(\overline{x_{\tau(1)}},\dots,\overline{x_{\tau(k+1)}}\right),\overline{x_{\tau(k+2)}},\dots\overline{x_{\tau(k+l+1)}}\right)
	\end{align*}
	for $x_1,\ldots, x_{k+l+1}\in\got g$.
	By definition, this is 
	\begin{align*}
		&(-1)^{kl}\!\!\!\!\sum_{\tau\in S_{(l+1,k)}}\pi_{\got{g}/\got{i}}\circ\phi\left(\psi\left(x_{\tau(1)},\dots,x_{\tau(l+1)}\right),x_{\tau(l+2)},\dots x_{\tau(k+l+1)}\right)\\
		&-\sum_{\tau\in S_{(k+1,l)}}\pi_{\got{g}/\got{i}}\circ\psi\left(\phi\left(x_{\tau(1)},\dots,x_{\tau(k+1)}\right),x_{\tau(k+2)},\dots, x_{\tau(k+l+1)}\right)=\pi_{\got{g}/\got{i}}\circ\llbra\phi,\psi\rrbra\left(x_1,\ldots, x_{k+l+1}\right).
	\end{align*}
	
	Finally check that the  linear map
	\begin{equation*}
		\overline{\Pi}\colon 	\biggleftpar\frac{ C^\bullet(\got{g};\got{g})}{ C_{\got{i}}^\bullet(\got{g};\got{g})}[1], -\overline{\delta_{\got{g}}^{\ad}}\biggrightpar \to \left(C_{\wedge}^{\bullet}(\got{g};\got{i}^*\otimes\got{g}/\got{i}), \delta_{\got{g}\rhd\got{i}}^{\Hom}\right), \qquad \overline{\phi}\mto (-1)^k(\pi_{\got{g}/\got{i}}\circ\phi)|_{\wedge^{k-1}\got{g}\wedge\got{i}}
	\end{equation*}
	for $k\geq 1$ and $\phi\in C^k(\got g;\got g)$, is an isomorphism of cochain complexes. Consider the composition of the map
	\[(\pi_{\got{g}/\got{i}})_*\colon \bigleftpar C^{\bullet}(\got{g};\got{g})[1], -\delta_{\got{g}}^{\ad}\bigrightpar\to\bigleftpar C^{\bullet}(\got{g};\got{g}/\got{i})[1], -\delta_{\pi_{\got{g}/\got{i}}}\bigrightpar, \quad \phi\mto \pi_{\got{g}/\got{i}}\circ\phi
	\]
	with 
	\[ 		\Pi_\wedge\colon (C^{\bullet}(\got{g};\got{g}/\got{i})[1], -\delta_{\pi_{\got{g}/\got{i}}})\to 
	\left(C_{\wedge}^{\bullet}(\got{g};\got{i}^*\otimes\got{g}/\got{i}), \delta_{\got{g}\rhd\got{i}}^{\Hom}\right)\]
	which was defined by 
	\[	C^{k+1}(\got{g};\got{g}/\got{i})\ni  \phi\mto (-1)^{k+1}\phi\arrowvert_{(\wedge^{k}\got{g})\wedge\got{i}}\]
	for all $k\geq 0$. Both maps are epimorphisms. The first one by Remark \ref{remark_functor_g-modules_cochain_complexes}, and the second one by definition of its codomain, see Proposition \ref{lemma proving the s.e.s. relating deform. of ideals, canonical projection, and quotient Lie algebra}.
Since the kernel of this composition is $C^\bullet_{\got{i}}(\got{g};\got{g})[1]\hookrightarrow C^\bullet(\got{g};\got{g})[1]$, the conclusion follows since each epimorphism is the cokernel of its kernel, see Remark \ref{remark_functor_g-modules_cochain_complexes}.
\end{proof}

\begin{rem}
The complex $ C_{\got{i}}^{\bullet}(\got{g};\got{g})$ appeared in \cite{Nijenhuis-Richardson-Deformations-of-Lie-algebra-structures-67'} as the complex which controls deformations of the Lie algebra $(\got{g}, \mu_{\got{g}})$ such that $\got{i}$ remains an ideal. This remark explains how this deformation problem is related to the one studied in this paper, at the infinitesimal level. A smooth deformation $(\got{i}_t)_{t\in I}$ of an ideal $\got{i}\subset\got{g}$ induces as follows a smooth deformation $(\mu_{\got{g}}^t)_{t\in I}$ of the Lie algebra $(\got{g},\mu_{\got{g}})$ such that $\mu_{\got{g}}^t(\got{g},\got{i})\subset\got{i}$ for all $t\in I$ -- i.e.~the Lie bracket on $\got g$ is deformed such that $\got i$ remains an ideal in the deformed Lie algebras. This smooth deformation is defined, after choosing a smooth family of linear isomorphisms $(\alpha_t)_{t\in I}\in\GL(\got{g})$ such that $\alpha_0=\id$ and $\alpha_t(\got{i})=\got{i}_t$, at each time $t$, by conjugation:
\begin{equation*}
	\mu_{\got{g}}^t(x,y)\eqdef\alpha_t^{-1}\circ\mu_{\got{g}}\left(\alpha_t(x),\alpha_t(y)\right), \qquad  \forall \ x,y\in\got{g}.
\end{equation*}
Differentiating the above equation at $t=0$ yields
\begin{equation*}
	\bigg{(}\fracddtz\mu_{\got{g}}^t\bigg{)}(x,y)=-\biggleftpar\fracddtz{\alpha_t}\biggrightpar\circ\mu_{\got{g}}(x,y)+\mu_{\got{g}}\biggleftpar\fracddtz\alpha_t(x),y\biggrightpar+\mu_{\got{g}}\biggleftpar x,\fracddtz{\alpha_t}(y)\biggrightpar
\end{equation*}
for all $x, y\in \got g$.
Assume that $y\eqdef u\in\got{i}$ and apply the canonical projection $\pi_{\got g/\got i}$ to the above equation:
\begin{equation*}
	\pi_{\got{g}/\got{i}}\circ\bigg{(}\fracddtz\mu_{\got{g}}^t\bigg{)}(x,u)=-\pi_{\got{g}/\got{i}}\circ\biggleftpar\fracddtz\alpha_t\biggrightpar([x,u])+\cancel{\pi_{\got{g}/\got{i}}\biggleftbra\fracddtz\alpha_t(x),u\biggrightbra}+\pi_{\got{g}/\got{i}}\biggleftbra x,\fracddtz\alpha_t(u)\biggrightbra.
\end{equation*}
The right-hand side of this equation vanishes since it is  the cocycle condition for $\pi_{\got{g}/\got{i}}\circ\fracddtzs\alpha_t\in C^0(\got{g};\got{i}^*\otimes\got{g}/\got{i})$, see \eqref{equation of differentiation to appear the cocycle for deform of ideals} in the proof of Proposition \ref{def_cocycle_ideal}. Therefore $\pi_{\got{g}/\got{i}}\circ\bigleftpar\fracddtzs\mu^t_{\got{g}}\bigrightpar\big{|}_{\got{g}\wedge\got{i}}=0$, which shows that $\fracddtzs\mu_{\got{g}}^t$ indeed is an element of $ C_{\got{i}}^2(\got{g};\got{g})$.
\end{rem}

\subsubsection{Relation with the deformation cohomology of the ideal as a Lie subalgebra}
Let $\got i$ be an ideal in a Lie algebra $\got g$. Recall that the deformation cochain complex of $\got i$ as a Lie subalgebra of $\got g$ is 
$(C^\bullet_{\rm def}(\got i\subset \got g), \delta_{\got i}^{\rm Bott})=(C^\bullet(\got{i};\got{g}/\got{i}),\delta_{\tr})$. Let $\iota_{\got i}\colon \got i\to \got g$ be the inclusion. 

\begin{prop}\label{prop_ideal_def_subalgebra_def}
The natural restriction
\begin{equation*}
	\res_{\wedge\got{i}}\colon 	 C^\bullet_{\wedge}(\got{g};\got{i}^*\otimes\got{g}/\got{i})\longtwoheadrightarrow C^{\bullet}(\got{i};\got{g}/\got{i})[1], 
\end{equation*}
\begin{equation}\label{dgla morphism from ideals to subalgebras}
	C^k_{\wedge}(\got{g};\got{i}^*\otimes\got{g}/\got{i}) \ni \phi\mto (-1)^{k+1}\phi|_{\wedge^{k+1}\got{i}},
\end{equation}
for $k\geq 0$ 
is a cochain map and $$H^0(\res_{\wedge\got{i}})\colon H_{\wedge}^0(\got{g};\got{i}^*\otimes\got{g}/\got{i})\rightarrow H_{\defor}^1(\got{i}\subset\got{g})$$ is injective. Furthermore, there is a commutative diagram of short exact sequences:
\begin{equation}\label{com_diag_ideal_subalgebra}
	\begin{tikzcd}
		C^{\bullet}(\got{g}/\got{i};\got{g}/\got{i})[1] \ar[r, hookrightarrow] \ar[d, hookrightarrow]&  C^{\bullet}(\got{g};\got{g}/\got{i})[1]  \ar[d, equal] \ar[r, twoheadrightarrow, "\Pi_{\wedge}"] &  C_{\wedge}^{\bullet}(\got{g};\got{i}^*\otimes\got{g}/\got{i}) \ar[d, twoheadrightarrow, "\res_{\wedge\got{i}}"]\\
		\Ker^{\bullet}(\iota_{\got i}^*)[1] \ar[r, hookrightarrow] &  C^{\bullet}(\got{g};\got{g}/\got{i})[1] \ar[r, twoheadrightarrow, "\iota_{\got{i}}^*"] &  C^{\bullet}(\got{i};\got{g}/\got{i})[1]
	\end{tikzcd}
\end{equation}
where $\Ker(\iota_{\got i}^*)=\{\phi\colon \wedge^{\bullet}\got{g}\to\got{g}/\got{i} \ | \ \phi|_{\wedge^{\bullet}\got{i}}=0\}$, and the top left and left  inclusions are just given by the canonical inclusion $\wedge^\bullet (\got g/\got i)^*\simeq \wedge^\bullet \got i^\circ \hookrightarrow \wedge^\bullet \got g^*$.
\end{prop}

Note that $C^0_{\wedge}(\got g; \got i^*\otimes \got g/\got i)=\got i^*\otimes \got g/\got i=C^1(\got i, \got g/\got i)$, so $\res_{\wedge\got i}$ is minus the identity in degree $0$. However, because the cocycle conditions on $C^0_{\wedge}(\got g; \got i^*\otimes \got g/\got i)$ and $C^1(\got i, \got g/\got i)$ are very different, $H^0(\res_{\wedge\got i})$ is in general not surjective and so not an isomorphism in cohomology.

Recall as well that in \eqref{com_diag_ideal_subalgebra}, not the full complexes $C^\bullet(\got g, \got g/\got i)[1]$, etc, are considered, but only their truncations to non-negative degrees, as the top-right term does not have negative degrees. The map $\Pi_\wedge$ is \emph{not} defined on $(C^\bullet(\got g; \got g/\got i)[1])_{-1}=C^0(\got g; \got g/\got i)=\got g/\got i$,
and also e.g.~$(C^\bullet(\got g/\got i; \got g/\got i)[1])_{-1}=C^0(\got g/\got i; \got g/\got i)=\got g/\got i$ is \emph{not} a vector subspace of $\Ker^0(\iota_{\got i}^*)$ which is $\{0\}$ by definition. 

\begin{proof}[Proof of Proposition \ref{prop_ideal_def_subalgebra_def}]
Compute using Proposition \ref{the lemma where we defined the cochain map from def. complex of can projection to the def. complex of ideals}
and $\res_{\wedge\got{i}}\circ\Pi=\iota_{\got i}^*$
\[
\res_{\wedge\got{i}}\circ\delta_{\got{g}\rhd\got{i}}^{\Hom}\circ \Pi=\res_{\wedge\got{i}}\circ\Pi\circ (-\delta_{\pi_{\got g/\got i}})
=\iota_{\got i}^*\circ (-\delta_{\pi_{\got g/\got i}}).
\]
On $\phi\in C^k(\got g, \got g/\got i)$ and $u_1,\ldots, u_{k+1}\in \got i$, this is 
\[-\sum_{i<j}\phi\left([u_i,u_j], u_1, \ldots, \widehat{u_i},\ldots, \widehat{u_j}, \ldots, u_{k+1}\right),
\]
which equals 
\[ (-\delta_{\got i}^{\rm Bott}(\iota_{\got i}^*\phi))(u_1,\ldots, u_{k+1}).
\]
Hence 
\[\res_{\wedge\got{i}}\circ\delta_{\got{g}\rhd\got{i}}^{\Hom}\circ \Pi
=\iota_{\got i}^*\circ (-\delta_{\pi_{\got g/\got i}})=-\delta_{\got i}^{\rm Bott}\circ \iota_{\got i}^*=-\delta_{\got i}^{\rm Bott}\circ \res_{\wedge\got{i}}\circ\Pi
\]
Since $\Pi$ is surjective on its image $C^\bullet_{\wedge}(\got{g};\got{i}^*\otimes\got{g}/\got{i})$, this shows that $\res_{\wedge\got{i}}\circ\delta_{\got{g}\rhd\got{i}}^{\Hom}=-\delta_{\got{i}}^{\Bott}\circ\res_{\wedge\got{i}}$ on $C^\bullet_{\wedge}(\got{g};\got{i}^*\otimes\got{g}/\got{i})$, i.e.~that $\res_{\wedge\got{i}}$ in \eqref{dgla morphism from ideals to subalgebras} is a cochain map.

If $\phi\colon \got{i}\to\got{g}/\got{i}$ is an infinitesimal deformation of $\got{i}\lhd\got{g}$, then $H^0(\res_{\wedge\got{i}})[\phi]=[-\phi]\in H^1_{\defor}(\got{i}\subset\got{g})$, 
where $[\phi]=\phi$ on the left-hand side is the cohomology class of $\phi$ as a closed element of $C^0_\wedge(\got g, \got i^*\otimes\got g/\got i)=\wedge^0\got g^*\otimes \got i^*\otimes \got g/\got i$, and $[-\phi]$ on the right-hand side is the cohomology class of $-\phi$ as an element of $C^1(\got i, \got g/\got i)$. Assume that the image $H^1(\got i, \got g/\got i)\ni [-\phi]=0$, then there exists $\overline{x}\in\got{g}/\got{i}$, i.e.~the class in $\got g/\got i$ of some $x\in \got g$, such that $-\phi=\pi_{\got{g}/\got{i}}\circ[\cdot, x]\colon \got{i}\to\got{g}/\got{i}$. But this vanishes because $\got i$ is an ideal in $\got g$. As a consequence, $\phi=0$ and so $H^0(\res_{\wedge\got{i}})$ is indeed injective.

\medskip

Next note that \[\Ker^\bullet(\iota_{\got i}^*)[1]=\{\phi\colon \wedge^{k}\got{g}\to\got{g}/\got{i} \ | \ k\geq 1 \text{ and } \phi|_{\wedge^{k}\got{i}}=0\}=\bigoplus_{p\geq 0}\left(\got i^\circ\wedge \bigwedge^p\got g^*\right)\otimes \got g/\got i\] naturally contains $ C^{\bullet\geq 1}(\got{g}/\got{i};\got{g}/\got{i})\simeq \wedge^{\bullet\geq 1} \got i^\circ\otimes \got g/\got i$.
The commutativity of  \eqref{com_diag_ideal_subalgebra} is easy to check: the right-hand square commutes since $\res_{\wedge\got{i}}\circ\Pi=\iota_{\got i}^*$ and the left-hand square is just the fact that via the identification $\wedge^{\bullet\geq 1}(\got g/\got i)^*\otimes \got g/\got i\simeq \wedge^{\bullet\geq 1}\got i^\circ\otimes \got g/\got i$, the inclusion of $\wedge^{\bullet\geq 1}(\got g/\got i)^*\otimes \got g/\got i$ has image in $\Ker^\bullet(\iota_{\got i}^*)$.
\end{proof}

\subsubsection{Bringing everything together}
Consider the following diagram.
\begin{equation}\label{commutative diagram where the def. complex of an ideal lives}
\begin{tikzcd}[row sep=small, column sep=large]
	C_{\got{i}}^{\bullet}(\got{g};\got{g})[1] \ar[r, hookrightarrow] \ar[dd, swap, "\overline{\pi_*}"] &  C^\bullet(\got{g};\got{g})[1] \ar[r, twoheadrightarrow, "p_{\Coker}"] \ar[dd, swap, "(\pi_{\got{g}/\got{i}})_*"] & \frac{ C^\bullet(\got{g};\got{g})[1]}{ C_{\got{i}}^{\bullet}(\got{g};\got{g})[1]} \ar[dd, "\simeqd"]\\ \\
	C^\bullet(\got{g}/\got{i};\got{g}/\got{i})[1] \ar[r, hookrightarrow] &  C^{\bullet}(\got{g};\got{g}/\got{i})[1] \ar[r, "\Pi_{\wedge}", twoheadrightarrow] &  C_{\wedge}^{\bullet}(\got{g};\got{i}^*\otimes\got{g}/\got{i})
\end{tikzcd}
\end{equation}
The short exact sequence at the bottom of the diagram was already found in \eqref{com_diag_ideal_subalgebra} and the map $\overline{\pi_*}$ was defined in \eqref{def_pi_star}. The isomorphism on the right-hand side is given by the map defined in \eqref{iso_coker_pi_complex}.
The map $(\pi_{\got g/\got i})_*\colon C^\bullet (\got g, \got g)\to C^\bullet(\got g, \got g/\got i)$ sends $\omega\in \wedge^\bullet\got g^*\otimes \got g$ to 
$\pi_{\got g/\got i}\circ \omega\in \wedge^\bullet\got g^*\otimes \got g/\got i$. It is easy to check that this is a cochain map
\[ \left(C^\bullet (\got g, \got g), \delta^{\ad}_{\got g}
\right) \to \left(C^\bullet(\got g, \got g/\got i),\delta^{\ad^{\got{g}/\got{i}}}_{\got g}
\right).
\]
The square on the left-hand side of \eqref{commutative diagram where the def. complex of an ideal lives} commutes by definition of $\overline{\pi_*}$ and $(\pi_{\got g/\got i})_*$, and the square on the right-hand side commutes as well by definition of the involved linear maps.

\begin{prop}\label{prop. that the def. cohom. of an ideal fits in a commutative diagram of l.e.s.}
The cohomology $H_{\wedge}^k(\got{g};\got{i}^*\otimes\got{g}/\got{i})$ fits into two long exact sequences, which are related by the following commutative diagram (starting at $k=1$).
\begin{equation}\label{commutative diagram where the def. cohomology of an ideal lives}
	\begin{tikzcd}[row sep=small, column sep=small]
		\cdots \ar[r] & H_{\got{i}}^{k}(\got{g};\got{g}) \ar[r] \ar[dd, "H^{k}(\overline{\pi_*})"] & H^k(\got{g};\got{g})\ar[r] \ar[dd, "H^{k}((\pi_{\got{g}/\got{i}})_*)"]& H^k_{\Coker}(\got{g};\got{g}) \ar[r, "{\partial_{\got g}^{k-1}}"]\ar[dd, "\simeqd"] & H^{k+1}_{\got{i}}(\got{g};\got{g}) \ar[r]  \ar[dd, "H^{k+1}(\overline{\pi_*})"]& \cdots\\
		\\
		\cdots \ar[r] & H^k(\got{g}/\got{i};\got{g}/\got{i}) \ar[r] & H^k(\got{g};\got{g}/\got{i}) \ar[r] & H_{\wedge}^{k-1}(\got{g};\got{i}^*\otimes\got{g}/\got{i}) \ar[r,  "{\partial^{k-1}_{\got{g}/\got{i}}}" ] & H^{k+1}(\got{g}/\got{i};\got{g}/\got{i}) \ar[r] & \cdots 
	\end{tikzcd}
\end{equation}
with the connecting homomorphisms $\partial^\bullet_{\got g}$ and $\partial^\bullet_{\got g/\got i}$  defined as usual (see e.g.~\cite[I.\S 2]{BoTu82}).
\end{prop}

The proof of Proposition \ref{prop. that the def. cohom. of an ideal fits in a commutative diagram of l.e.s.} is just a straightforward application of the following lemma to the diagram in \eqref{commutative diagram where the def. complex of an ideal lives}. The proof of the following lemma can be found, among other sources, in \cite[Proposition 1.3.4]{Weibel-book-Homological-Algebra}, 
\begin{lem}
Let 
\begin{equation*}
	\begin{tikzcd}[row sep=small, column sep=large]
		(A_1^\bullet, d_{A,1}) \ar[r, hookrightarrow, "\phi_1"] \ar[dd, swap, "F_A"] & (B_1^\bullet, d_{B,1})  \ar[r, twoheadrightarrow, "\psi_1"] \ar[dd, swap, "F_B"] &(C_1^\bullet, d_{C,1}) \ar[dd, "F_C"]\\ \\
		(A_2^\bullet, d_{A,2}) \ar[r, hookrightarrow, "\phi_2"] &  (B_2^\bullet, d_{B,2})  \ar[r, "\psi_2", twoheadrightarrow] &  	(C_2^\bullet, d_{C,2})	\end{tikzcd}
\end{equation*}
be a commutative diagram of cochain complexes, with the two horizontal sequences being exact.

Then the maps $[F_A]$, $[F_B]$, $[F_C]$ defined in cohomology by $F_A$, $F_B$ and $F_C$ intertwine the long exact sequences in cohomology induced by the short exact sequences $(A_i, d_{A,i})\hookrightarrow (B_i, d_{B,i}) \twoheadrightarrow (C_i, d_{C,i})$ for $i=1,2$. That is, the following diagram of long exact sequences in cohomology commutes, with $k\geq 0$.
\begin{equation*}
	\begin{tikzcd}[row sep=small, column sep=small]
		\cdots \ar[r,"{\partial_{1}^{k-1}}"] & H^k(A_1) \ar[r,"{[\phi_1]}"] \ar[dd, "{[F_A]}"] & H^k(B_1)\ar[r,"{[\psi_1]}"] \ar[dd,"{[F_B]}"]& H^k(C_1) \ar[r,"{\partial_1^{k}}"]\ar[dd,"{[F_C]}"] & H^{k+1}(A_1) \ar[r]  \ar[dd, "{[F_A]}"]& \cdots\\
		\\
		\cdots \ar[r,"{\partial^{k-1}_{2}}"] & H^k(A_2) \ar[r,"{[\phi_2]}"] & H^k(B_2) \ar[r,"{[\psi_2]}"] & H^k(C_2) \ar[r,"{\partial^{k}_{2}}"] & H^{k+1}(A_1) \ar[r] & \cdots 
	\end{tikzcd}
\end{equation*}
\end{lem}
%
%

\section{Deformations of a Lie ideal as Maurer-Cartan elements}\label{voronov_ideals}

While for investigating the infinitesimal deformation theory of ideals, no choice of complement is needed, for finding the Lie structure on the deformation complex $ C^{\bullet}(\got{g};\got{i}^*\otimes\got{g}/\got{i})$ it is crucial to fix a complement $\got i^c$ of the ideal $\got{i}\lhd\got{g}$ in $\got g$, or equivalently a splitting $\sigma\colon\got g/\got i\to \got g$ of the short exact sequence 
\begin{equation}\label{ses_pi}
0 \rightarrow \got i \rightarrow \got g\rightarrow \got g/\got i\rightarrow 0
\end{equation}
of vector spaces defined by the epimorphism $\pi_{\got g/\got i}\colon \got g\to \got g/\got i$ of Lie algebras.
That is, $\got i^c$, respectively $\sigma(\got g/\got i)\subseteq \got g$ is canonically isomorphic to $\got{g}/\got{i}$ as a vector space. 
The projection $\got g\to \got i$, $x\mapsto x-\sigma(\pi_{\got g/\got i}(x))$ associated to the choice of splitting $\sigma$ is denoted by $\pr_{\sigma}\colon \got{g}\twoheadrightarrow\got{i}$. The inclusion $\got i\hookrightarrow \got g$ is denoted by $\iota_{\got i}$.

Consider small deformations of $\got{i}\lhd\got{g}$ in the sense of the following definition.

\begin{defin}
Let $\got{i}\lhd{\got{g}}$ be a $k$-dimensional ideal  in a Lie algebra $(\got g, [\cdot\,,\cdot])$ and let $\got i^c\subset\got{g}$ be a vector space complement of $\got{i}$. A $k$-dimensional subspace $\got{i}'\subset\got{g}$ is called a \textbf{small deformation of $\got{i}\lhd\got{g}$} if $\got{g}=\got{i}'\oplus \got i^c$ and $\got{i}'$ is an ideal in $\got{g}$.
\end{defin}

Recall that the choice of ${\got i^c}$ defines a smooth chart of $\Gr_k(\got g)$ around $\got i$. Its elements are the $k$-dimensional subspaces of $\got g$ that are complementary to ${\got i^c}$.
Therefore, small deformations of $\got{i}\lhd{\got{g}}$ are the subspaces of the vector space $\got{g}$ which lie in the chart domain $U_{\got i^c}\simeq\Hom(\got{i},\got i^c)$ around $\got{i}\in\Gr_{k}(\got{g})$ and are, in addition, ideals of the Lie algebra $(\got{g},[\cdot,\cdot])$.

\subsection{The Voronov dataset associated to an ideal in a Lie algebra}
This section begins with the construction of  a dgL$[1]$a structure on $C^\bullet(\got g; \got i^*\otimes \got g/\got i)$ associated to an ideal $\got i$ in a Lie algebra $\got g$ and a choice of splitting of \eqref{ses_pi}, that will then be the dgL$[1]$a controlling the deformations of $\got i$ in $\got g$.
\begin{prop}\label{lemma for V-data of an ideal} Let $\got g$ be a Lie algebra and let $\got i$ be an ideal in it. Choose a splitting $\sigma\colon \got g/\got i\to \got g$ of the short exact sequence \eqref{ses_pi} of vector spaces. The following quadruple is then a Voronov-dataset:
\begin{enumerate}
	\item The graded Lie algebra $\mbf{g}\eqdef\bigleftpar C^{\bullet}(\got{g};\got{g})[1]\oplus C^{\bullet}(\got{g};\got{gl}(\got{g})), \llbra\cdot,\cdot\rrbra\bigrightpar$ defined in Section \ref{graded_LA_LA+rep}.
	\item $\got{a}\eqdef C^{\bullet}(\got{g};\got{i}^*\otimes \got g/\got i)$ with the inclusion $\mbf{I}\colon \got{a}\hookrightarrow\mbf{g}$ defined by
	$ \mbf{I}(\psi)\in C^k(\got g;\got{gl}(\got{g}))$ for $ \psi\in C^k(\got{g};\got{i}^*\otimes \got g/\got i)$,
	\[\mbf{I}(\psi)(x_1,\dots,x_{k},y)\eqdef \sigma\bigleftpar\psi(x_1,\dots,x_{k},\pr_{\sigma}(y))\bigrightpar\]
	for $x_1,\ldots, x_k\in\got g$ and $y\in \got g$.
	\item $\mbf{P}\colon \mbf{g}\twoheadrightarrow\got{a}$ defined by \[\mbf{P}(\phi,\psi)(x_1,\dots,x_k,u)\eqdef\pi_{\got g/\got i}\bigleftpar\psi\bigleftpar x_1,\dots,x_k,\iota_{\got{i}}(u)\bigrightpar\bigrightpar\]
	for $(\phi,\psi)\in C^{k+1}(\got{g};\got{g})\oplus C^k(\got{g};\got{gl}(\got{g}))$, $x_1,\ldots, x_k\in\got g$ and $u\in \got i$.
	\item $\Theta\eqdef\mu_{\got{g}}+\ad^{\got{g}}$.
\end{enumerate}
\end{prop}
Note that in the proposition above, the map $\mbf I$ depends on the choice of $\sigma$, while the map $\mbf P$ does not.
\begin{proof}
By Section \ref{graded_LA_LA+rep}, $(\mbf{g},\llbra\cdot,\cdot\rrbra)$ is a graded Lie algebra and $\mu_{\got{g}}+\ad^{\got{g}}$ is a Maurer-Cartan element in it. 

First check that $\got{a}\subset\mbf{g}$ is an abelian subalgebra via the inclusion $\mbf I$. If $\psi_1\in C^{k_1}(\got{g};\got{i}^*\otimes \got g/\got i)$ and $\psi_2\in C^{k_2}(\got{g};\got{i}^*\otimes \got g/\got i)$, then for $x_1,\ldots, x_{k_1+k_2}, y\in \got g$
\begin{align*}
	&\llbra\mbf{I}(\psi_1),\mbf{I}(\psi_2)\rrbra(x_1,\dots,x_{k_1+k_2},y)=\bigleftpar(-1)^{k_1k_2}\hspace{+0.1cm}\mbf{I}(\psi_1)\circ\mbf{I}(\psi_2)-\mbf{I}(\psi_2)\circ\mbf{I}(\psi_1)\bigrightpar(x_1,\dots,x_{k_1+k_2},y).
\end{align*}
It is enough to check that $(\mbf{I}(\psi_1)\circ\mbf{I}(\psi_2))(x_1,\dots,x_{k_1+k_2},y)$ vanishes.
By definition, this is 
\begin{equation}\label{comp_I}\begin{split}
		&\sum_{\tau\in S_{(k_2,k_1)}}\hspace{-0.2cm}(-1)^{\tau}\hspace{+0.1cm}\mbf{I}(\psi_{1})\bigleftpar x_{\tau(k_2+1)},\dots,x_{\tau(k_2+k_1)},\mbf{I}(\psi_2)(x_{\tau(1)},\dots,x_{\tau(k_2)},y)\bigrightpar\\
		&=\sum_{\tau\in S_{(k_2,k_1)}}\hspace{-0.2cm}(-1)^{\tau}\hspace{+0.1cm}\sigma\left(\psi_1\biggleftpar x_{\tau(k_2+1)},\dots,x_{\tau(k_1+k_2)},(\pr_{\sigma}\circ \sigma)\bigleftpar\psi_2\bigleftpar x_{\tau(1)},\dots,x_{\tau(k_2)},\pr_{\sigma}(y)\bigrightpar\bigrightpar\biggrightpar\right),
	\end{split}
\end{equation}
which vanishes because $\pr_{\sigma}\circ \sigma=0$.

\medskip

Next check that $\Ker(\mbf{P})$ is a graded Lie subalgebra of $(\mbf{g},\llbra\cdot,\cdot\rrbra)$. Observe that the kernel of the projection $\mbf{P}:\mbf{g}\twoheadrightarrow\got{a}$ consists of the following graded vector spaces:
\begin{equation*}
	\Ker(\mbf{P})=\underbrace{ C^{\bullet}(\got{g};\got{g})[1]}_{=:A}\oplus\left(\underbrace{ C^{\bullet}(\got{g};\got g^*\otimes\got{i})}_{=:B}+\underbrace{ C^{\bullet}(\got{g};\got{i}^\circ\otimes\got{g})}_{=:C}\right).
\end{equation*}
By \eqref{eq1_LALArep}, the first component $A$ is a graded Lie (sub)algebra of $\mbf g$, while \eqref{eq2_LALArep} implies immediately that $\llbra A,B\rrbra\subset B$ and $\llbra A,C\rrbra\subset C$. 
Last, \eqref{eq3_LALArep} implies that $\llbra B,B\rrbra\subset B$, $\llbra C,C\rrbra\subset C$ and $\llbra B,C\rrbra\subset B$. The element $\mu_{\got{g}}$ is clearly in $A$ 
and $\ad^{\got{g}}\in B\subseteq \Ker(\mbf{P})$ since $\got{i}\subset\got{g}$ is an ideal. 
Hence $\Theta=\mu_{\got{g}}+\ad^{\got{g}}\in\Ker(\mbf{P})\cap\MC(\mbf{g})$.
\end{proof}

The following shows that the dgL$[1]$a structure on $ C^{\bullet}(\got{g};\got{i}^*\otimes\ \got{g}/\got{i})$ induced by the above Voronov dataset is the one controlling the deformation problem of the ideal $\got{i}\lhd\got{g}$. That is, the unary bracket 
\[ m_1\colon C^{\bullet}(\got{g};\got{i}^*\otimes\ \got g/\got i)\to C^{\bullet+1}(\got{g};\got{i}^*\otimes\ \got g/\got i)
\]
defined by 
$m_1(\psi)=\mbf{P}\llbra\mu_{\got{g}}+\ad^{\got{g}},\mbf{I}(\psi)\rrbra$ equals the differential $\delta_{\got{g}\rhd\got{i}}^{\Hom}$. The following computation confirms this. For $\psi\in \wedge^k\got g^*\otimes(\got i^*\otimes {\got g/\got i})$, $x_1, \ldots, x_{k+1}\in\got g$ and $u\in \got i$ 
\begin{align*}
&m_1(\psi)(x_1,\dots,x_{k+1},u)=\mbf{P}\llbra\mu_{\got{g}}+\ad^{\got{g}},\mbf{I}(\psi)\rrbra(x_1,\dots,x_{k+1},u)\\
&=\mbf{P}\llbra\mu_{\got{g}},\mbf{I}(\psi)\rrbra(x_1,\dots,x_{k+1},u)+\mbf{P}\llbra\ad^{\got{g}},\mbf{I}(\psi)\rrbra(x_1,\dots,x_{k+1},u).
\end{align*}
The first term is 
\begin{align*}
&\pi_{\got g/\got i}\bigleftpar\llbra\mu_{\got{g}},\mbf{I}(\psi)\rrbra(x_1,\dots,x_{k+1},\iota_{\got{i}}(u))\bigrightpar\\
&=-\pi_{\got g/\got i}\left(\sum_{\tau\in S_{(2,k-1)}}\hspace{-0.2cm}(-1)^{\tau}\hspace{+0.1cm}\mbf{I}(\psi)\bigleftpar\mu_{\got{g}}(x_{\tau(1)},x_{\tau(2)}),x_{\tau(3)},\dots,x_{\tau(k+1)},\iota_{\got{i}}(u)\bigrightpar\right)\\
&=-\hspace{-0.3cm}\sum_{\tau\in S_{(2,k-1)}}\hspace{-0.2cm}(-1)^{\tau}\psi([x_{\tau(1)},x_{\tau(2)}],x_{\tau(3)},\dots,x_{\tau(k+1)},u)\\
&=\sum_{i<j}^{k+1}(-1)^{i+j}\hspace{+0.1cm}\psi([x_{i},x_{j}],x_1,\dots, \widehat{x_i} ,\dots, \widehat{x_j} ,\dots,x_{k+1},u)
\end{align*}
and the second term is 
\begin{align*}
&\pi_{\got g/\got i}\bigleftpar\llbra\ad^{\got{g}},\mbf{I}(\psi)\rrbra(x_1,\dots,x_{k+1},\iota_{\got{i}}(u))\bigrightpar\\
=\,&\pi_{\got g/\got i}\left((-1)^{k}\hspace{-0.2cm}\sum_{\tau\in S_{(k,1)}}\hspace{-0.2cm}(-1)^{\tau}\ad^{\got{g}}_{x_{\tau(k+1)}}\left(\mbf{I}(\psi)(x_{\tau(1)},\dots,x_{\tau(k)},\iota_{\got{i}}(u))\right)\right)\\
&-\pi_{\got g/\got i}\left(\sum_{\tau\in S_{(1,k)}}(-1)^{\tau}\hspace{+0.1cm}\mbf{I}(\psi)\left(x_{\tau(2)},\dots,x_{\tau(k+1)},\ad^{\got{g}}_{x_{\tau(1)}}(\iota_{\got{i}}(u))\right)\right)\\
=\,&\sum_{\tau\in S_{(1,k)}}(-1)^{\tau}\pi_{\got g/\got i}\biggleftpar\ad^{\got{g}}_{x_{\tau(1)}}\bigleftpar \sigma(\psi(x_{\tau(2)},\dots,x_{\tau(k+1)},u))\bigrightpar\biggrightpar\\
&-\sum_{\tau\in S_{(1,k)}}(-1)^{\tau}\hspace{+0.1cm}\psi\bigleftpar x_{\tau(2)},\dots,x_{\tau(k+1)},[x_{\tau(1)},u]\bigrightpar\\
=\,&\sum_{i=1}^{k+1}(-1)^{i+1}\ad^{\got{g}/\got i}_{x_{i}}(\psi(x_{1},\dots, \widehat{x_i} ,\dots,x_{k+1},u))+\sum_{i=1}^{k+1}(-1)^i\hspace{+0.1cm}\psi(x_{1},\dots, \widehat{x_i} , \dots,x_{k+1},[x_i, u]).
\end{align*}
Together, the two terms add therefore up to $\delta_{\got{g}\rhd\got{i}}^{\Hom}(\psi)(x_1,\ldots, x_{k+1}, u)$. Note along the way that the equality $m_1=\delta_{\got{g}\rhd\got{i}}^{\Hom}$ shows that $m_1$ does not depend on the choice of the splitting $\sigma$.

\bigskip

Next consider the higher brackets. 
Choose $\psi\in 
C^k(\got{g};\got{i}^*\otimes \got g/\got i)$ and $\phi\in 
C^l(\got{g};\got{i}^*\otimes \got g/\got i)$. 
Compute for $k\geq 1$ 
\begin{equation*}\begin{split}
	\llbra\llbra\mu_{\got{g}}+\ad^{\got{g}},\mbf{I}(\psi)\rrbra,\mbf{I}(\phi)\rrbra
	&=\llbra\llbra\mu_{\got{g}},\mbf{I}(\psi)\rrbra,\mbf{I}(\phi)\rrbra+\llbra\llbra\ad^{\got{g}},\mbf{I}(\psi)\rrbra,\mbf{I}(\phi)\rrbra\\
	&=\llbra -\mbf I(\psi)\circ \mu_{\got g}, \mbf{I}(\phi)\rrbra+\llbra(-1)^k\ad^{\got{g}}\circ \mbf{I}(\psi)-\mbf{I}(\psi)\circ \ad^{\got{g}},\mbf{I}(\phi)\rrbra
	\end{split}\end{equation*}
	and for $k=0$
	\begin{equation*}\begin{split}
	\llbra\llbra\mu_{\got{g}}+\ad^{\got{g}},\mbf{I}(\psi)\rrbra,\mbf{I}(\phi)\rrbra
	&=\llbra\llbra\mu_{\got{g}},\mbf{I}(\psi)\rrbra,\mbf{I}(\phi)\rrbra+\llbra\llbra\ad^{\got{g}},\mbf{I}(\psi)\rrbra,\mbf{I}(\phi)\rrbra\\
	&=\llbra 0, \mbf{I}(\phi)\rrbra+\llbra \ad^{\got{g}}\circ \mbf{I}(\psi)-\mbf{I}(\psi)\circ \ad^{\got{g}},\mbf{I}(\phi)\rrbra.
	\end{split}\end{equation*}
	
	In the case $k\geq 1$, the term $-\llbra \mbf I(\psi)\circ \mu_{\got g}, \mbf{I}(\phi)\rrbra$ equals 
	\begin{equation}\label{interm_1}
(-1)^{(k+1)l}(\mbf I(\psi)\circ \mu_{\got g})\circ \mbf{I}(\phi)-\mbf{I}(\phi)\circ (\mbf I(\psi)\circ \mu_{\got g}).
\end{equation}
Since 
\begin{equation*}
\begin{split}
	\left(\mbf I(\psi)\circ \mu_{\got g}\right)(x_1,\ldots, x_{k+1})&=\sum_{\tau\in S_{(2,k-1)}}(-1)^\tau\, \mbf I(\psi)(\mu_{\got g}(x_{\tau(1)}, x_{\tau(2)}), x_{\tau(3)}, \ldots, x_{\tau(k+1)})\\
	&=\sigma\circ \left(\sum_{\tau\in S_{(2,k-1)}}(-1)^\tau\, \psi(\mu_{\got g}(x_{\tau(1)}, x_{\tau(2)}), x_{\tau(3)}, \ldots, x_{\tau(k+1)})\right)\circ\pr_\sigma
\end{split}
\end{equation*}
for all $x_1,\ldots, x_{k+1}$ and since $\pr_\sigma\circ\sigma=0$, both terms in \eqref{interm_1} vanish.

Next,
\begin{equation}\label{explicit_m2}
\begin{split}
	& \llbra(-1)^k\ad^{\got{g}}\circ \mbf{I}(\psi)-\mbf{I}(\psi)\circ \ad^{\got{g}},\mbf{I}(\phi)\rrbra\\
	&=(-1)^{k+(k+1)l}\ad^{\got{g}}\circ \cancel{\mbf{I}(\psi)\circ \mbf{I}(\phi)}
	-(-1)^k\mbf I(\phi)\circ\ad^{\got{g}}\circ \mbf{I}(\psi)\\
	&\quad -(-1)^{(k+1)l}\mbf{I}(\psi)\circ \ad^{\got{g}}\circ \mbf{I}(\phi)
	+\cancel{ \mbf{I}(\phi)\circ \mbf{I}(\psi)}\circ \ad^{\got{g}}.
	\end{split}\end{equation}
	
	At this point it is already easy to see that $\pr_\sigma\circ \sigma=0$ and so $\mbf I(\psi)\circ \mbf I(\phi)=0$ for all $\psi,\phi\in C^\bullet(\got g; \got i^*\otimes \got g/\got i)$, see also \eqref{comp_I}, imply that $m_r= 0$ for all $r\geq 3$: the iterated brackets with $\mu_{\got g}$ all vanish, and the iterated brackets with $\ad^{\got g}$ consist of sums of terms, each with at least one composition of the type $\mbf I(\psi)\circ \mbf I(\phi)$.

	\medskip

	The case $r=2$ remains to be computed explicitly. Choose again $\psi\in 
	C^k(\got{g};\got{i}^*\otimes \got g/\got i)$ and $\phi\in 
	C^l(\got{g};\got{i}^*\otimes \got g/\got i)$ and compute
	\begin{equation*}
\begin{split}
	&\left(\mbf I(\phi)\circ \ad^{\got{g}}\circ \mbf{I}(\psi)\right)(x_{1},\ldots, x_{k+l+1},y)\\
	&=\sum_{\tau\in S_{(k,1,l)}}(-1)^{\tau}\left(\mbf I(\phi)(x_{\tau(k+2)}, \ldots, x_{\tau(k+l+1)})\circ \ad^{\got g}_{x_{\tau(k+1)}}\circ \mbf I(\psi)(x_{\tau(1)},\dots,x_{\tau(k)})\right)(y)\\
	&=\sum_{\tau\in S_{(k,1,l)}}(-1)^{\tau}\sigma\left(\phi\left(x_{\tau(k+2)}, \ldots, x_{\tau(k+l+1)},\pr_\sigma\left[x_{\tau(k+1)}, \sigma(\psi(x_{\tau(1)}, \ldots,x_{\tau(k)},\pr_\sigma(y)))\right]\right)\right)\\
\end{split}
\end{equation*}
for all $x_1,\ldots, x_{k+1}, y\in\got g$. Then for all $x_1,\ldots, x_{k+l+1}\in\got g$ and $u\in\got i$,
\begin{equation*}
\begin{split}
	&\left(\mbf P\left(\mbf I(\phi)\circ \ad^{\got{g}}\circ \mbf{I}(\psi)\right)\right)(x_{1},\ldots, x_{k+l+1},u)\\
	&=\pi_{\got g/\got i}\left(\sum_{\tau\in S_{(k,1,l)}}(-1)^{\tau}
	\sigma\left(\phi\left(x_{\tau(k+2)}, \ldots, x_{\tau(k+l+1)},\pr_\sigma\left[x_{\tau(k+1)}, \sigma(\psi(x_{\tau(1)}, \ldots,x_{\tau(k)},u))\right]\right)\right)\right)\\
	&=\sum_{\tau\in S_{(k,1,l)}}(-1)^{\tau}\phi\left(x_{\tau(k+2)}, \ldots, x_{\tau(k+l+1)},\pr_\sigma\left[x_{\tau(k+1)}, \sigma(\psi(x_{\tau(1)}, \ldots,x_{\tau(k)},u))\right]\right).
\end{split}
\end{equation*}
This shows that $m^{\sigma}_2(\psi, \phi)(x_1,\ldots, x_{k+l+1},u)$ equals 
\begin{equation*}
\begin{split}
	& -(-1)^k\sum_{\tau\in S_{(k,1,l)}}(-1)^{\tau}\phi\left(x_{\tau(k+2)}, \ldots, x_{\tau(k+l+1)},\pr_\sigma\left[x_{\tau(k+1)}, \sigma(\psi(x_{\tau(1)}, \ldots,x_{\tau(k)},u))\right]\right)\\
	&-(-1)^{(k+1)l}\sum_{\tau\in S_{(l,1,k)}}(-1)^{\tau}\psi\left(x_{\tau(l+2)}, \ldots, x_{\tau(k+l+1)},\pr_\sigma\left[x_{\tau(l+1)}, \sigma(\phi(x_{\tau(1)}, \ldots,x_{\tau(l)},u))\right]\right)\\\end{split}
	\end{equation*}
	for all $x_1,\ldots, x_{k+l+1}\in\got g$ and $u\in\got i$
	
	In particular if $k=l=0$, then for all $x\in \got g$ and $u\in \got i$
	\begin{equation*}\begin{split}
	m_2^{\sigma}(\psi,\phi)(x,u)
	&=-\phi\left(\pr_{\sigma}[x,\sigma(\psi(u))]\right)-\psi\left(\pr_{\sigma}[x,\sigma(\phi(u))]\right).
	\end{split}\end{equation*}
	This shows that for $\phi\in C^0(\got g, \got i^*\otimes \got g/\got i)$, the bracket 
	$m^{\sigma}_2(\phi,\phi)$ is defined by 
	\begin{equation}\label{m_2_on_same_deg0}
m^{\sigma}_2(\phi,\phi)(x,u)=-2\phi(\pr_{\sigma}[x,\sigma(\phi(u))])
\end{equation}
for all $x\in \got g$ and all $u\in \got i$.

	%
	
	\begin{thm}\label{dgl1a_ideal}
Given a Lie ideal $\got{i}\lhd(\got{g},[\cdot,\cdot])$ together with a choice of splitting $\sigma\colon \got g/\got i\to \got g$ of \eqref{ses_pi}, there is a bijection between
\begin{enumerate}
	\item MC-elements of the $L_\infty[1]$-algebra $( C^{\bullet}(\got{g};\got{i}^*\otimes \got g/\got i),\delta_{\got{g}\rhd\got{i}}^{\Hom}=m_1,m_2^\sigma)$, and
	\item small deformations of the ideal $\got{i}\lhd\got{g}$,
\end{enumerate}
given by the correspondence: \[C^0(\got g; \got i^*\otimes {\got g/\got i})\ni \phi\mto\Graph(\sigma\circ \phi)\subseteq \got g.\]
\end{thm}
\begin{proof}
The subspace $\Graph(\sigma \circ\phi)=\{u+\sigma(\phi(u)) | u\in\got{i}\}$ is an ideal in $\got g$ if and only if $[\got{g},\Graph(\sigma\circ\phi)]\subset\Graph(\sigma\circ\phi)$. Equivalently, this means that for every $x\in\got{g}$ and $u\in\got{i}$:
\begin{align*}
	[x,u]+[x,\sigma(\phi(u))]=[x,u+\sigma(\phi(u))]\in\Graph(\sigma\circ \phi).	\end{align*}
This is equivalent to 
\[(\sigma\circ\phi)\left([x,u]+\pr_{\sigma}[x,\sigma(\phi(u))]\right)=\sigma\left(\pi_{\got g/\got i}[x,\sigma(\phi(u))]\right)
\]
for all $x\in \got g$ and all $u\in \got i$.
Applying $\pi_{\got g/\got i}$ to both sides of this equation yields the equivalent equation
\begin{equation*}
	\phi([x,u])+\phi(\pr_{\sigma}[x,\sigma(\phi(u))])=\pi_{\got g/\got i}([x,\sigma(\phi(u))])=\left[\pi_{\got g/\got i}(x),\phi(u)\right]
\end{equation*}
for all $x\in \got g$ and all $u\in \got i$.

On the other hand, compute the Maurer-Cartan equation of $(C^\bullet(\got{g};\got{i}^*\otimes \got g/\got i),\delta_{\got{g}\rhd\got{i}}^{\Hom}=m_1,m_2^\sigma)$ for $\phi\in C^0(\got g; \got i^*\otimes {\got g/\got i})$:
\begin{equation*}
	\delta_{\got{g}\rhd\got{i}}^{\Hom}(\phi)+\frac{1}{2}m^{\sigma}_2(\phi,\phi)=0.
\end{equation*}
On $x\in\got{g}$ and $u\in\got{i}$, this reads \begin{equation*}
	\left[\pi_{\got g/\got i}(x),\phi(u)\right]-\phi([x,u])-\phi(\pr_{\sigma}[x,\sigma(\phi(u))])=0
\end{equation*}
by the considerations above.
This concludes the proof.
\end{proof}

Observe that in the case of a Lie subalgebra $\got{h}\subset\got{g}$, the controlling $L_{\infty}[1]$-algebra becomes a dgL$[1]$a 
if the chosen complement $\got{h}^c$ is a Lie subalgebra itself (such a complement does not always exist), since then a straightforward computation shows that the trinary bracket $m_3$ defined in \eqref{m_3_subalgebra} vanishes. 

However, in the case of a Lie ideal $\got{i}\lhd\got{g}$, the choice of a splitting $\sigma$ which is a morphism of Lie algebras does not simplify the problem. However, picking a complement which is itself a Lie ideal turns the deformation problem into a linear one since then the binary bracket $m_2$ vanishes. Once again, a choice like this,\ is not always possible -- except for example when $\got{g}$ is semisimple.

\begin{rem}
Given a dgL$[1]$a $(\mbf{V},[\cdot,\cdot],\mbf{d})$, two MC-elements $x_0, x_1\in V_1[-1]$ are called \emph{gauge-equivalent} if there exists a smooth path $(x_t)_{t\in [0,1]}$ of MC-elements interpolating between $x_0$ and $x_1$, together with a smooth path $(y_t)_{t\in [0,1]}$ of degree $-1$ elements of $\mbf{V}$, such that the following equation holds:
\begin{equation*}
	\frac{d}{dt}x_t=\mbf{d}(y_t)+\frac{1}{2}[y_t,x_t].
\end{equation*}
It is worth mentioning that here there are no degree $-1$ elements in the dgL$[1]$a $C^\bullet(\got{g};\got{i}^*\otimes\got{i}^c)$. This means that the gauge equivalence of its Maurer--Cartan elements is trivial. Therefore, the bijection of Theorem \ref{dgl1a_ideal} intertwines the trivial gauge equivalence of the MC-elements with the trivial equivalence relation on the small deformations of the Lie ideal $\got{i}\lhd\got{g}$.
\end{rem}

\bigskip

Recall from the considerations above Theorem \ref{dgl1a_ideal} that after a choice of splitting $\sigma\colon \got g/\got i\to \got g$ of \eqref{ses_pi},  the unary bracket $m_1=\delta_{\got{g}\rhd\got{i}}^{\Hom}$ of the $L_\infty[1]$-algebra $( C^{\bullet}(\got{g};\got{i}^*\otimes \got g/\got i),\delta_{\got{g}\rhd\got{i}}^{\Hom}=m_1,m_2^\sigma)$ defined by $\sigma$ does not depend of $\sigma$. The following theorem compares the binary brackets given by two choices of splittings of \eqref{ses_pi}.
Consider two choices of splittings $\sigma, \nu\colon \got g/\got i\to \got g$ of \eqref{ses_pi}. Then the difference $\sigma-\nu$ is an element of $\operatorname{Hom}(\got g/\got i, \got i)$. This space is a representation of $\got g$ via 
\[ \ad^{\rm Hom}\colon \got g\times \operatorname{Hom}(\got g/\got i, \got i)\to \operatorname{Hom}(\got g/\got i, \got i),
\]
\[\ad^{\rm Hom}_x(\phi)=\ad^{\got g}_x\circ \phi-\phi\circ\ad^{\got g/\got i}_x,
\]
i.e.
\[\ad^{\rm Hom}_x(\phi)(\bar y)=[x, \phi(\bar y)]-\phi[\bar x,\bar y]
\]
for all $x,y\in\got g$.

An easy computation shows that the difference $\pr_\sigma-\pr_\nu$ equals $(\nu-\sigma)\circ\pi_{\got g/\got i}\in\operatorname{Hom}(\got g, \got i)$.
\begin{lem}
Given a  Lie ideal $\got{i}\lhd(\got{g},[\cdot,\cdot])$ together with two choices of splittings $\sigma, \nu\colon \got g/\got i\to \got g$ of \eqref{ses_pi}, 
the $L_\infty[1]$-algebras $( C^{\bullet}(\got{g};\got{i}^*\otimes \got g/\got i),\delta_{\got{g}\rhd\got{i}}^{\Hom}=m_1,m_2^\sigma)$ and $( C^{\bullet}(\got{g};\got{i}^*\otimes \got g/\got i),\delta_{\got{g}\rhd\got{i}}^{\Hom}=m_1,m_2^\nu)$
satisfy
\begin{equation}\label{diff_m2s}
	(m_2^\sigma-m_2^\nu)(\psi, \phi)=-(-1)^k\phi\circ \delta_{\got g}^{\ad^{\rm Hom}}(\sigma-\nu)\circ \psi-(-1)^{(k+1)l}\psi\circ \delta_{\got g}^{\ad^{\rm Hom}}(\sigma-\nu)\circ \phi
\end{equation}
for all $\psi\in\wedge^k\got g^*\otimes \got i^*\otimes \got g/\got i$, $\phi\in\wedge^l\got g^*\otimes \got i^*\otimes \got g/\got i$
-- with the compositions defined as in the second part of Remark \ref{remark_on_comp_forms}.
\end{lem}

\begin{proof}
%
%
%
%
By \eqref{explicit_m2}  
\begin{equation}\label{diff_m2}
\begin{split}
	\left(m_2^\sigma-m_2^\nu\right)(\psi,\phi)=&(-1)^k\mbf P\left(\mbf I_\nu(\phi)\circ\ad^{\got{g}}\circ \mbf{I}_\nu(\psi)-\mbf I_\sigma(\phi)\circ\ad^{\got{g}}\circ \mbf{I}_\sigma(\psi)\right)\\
	&\quad +(-1)^{(k+1)l}\mbf P\left(\mbf{I}_\nu(\psi)\circ \ad^{\got{g}}\circ \mbf{I}_\nu(\phi)-\mbf{I}_\sigma(\psi)\circ \ad^{\got{g}}\circ \mbf{I}_\sigma(\phi)\right)
\end{split}
\end{equation}
The first term of this sum equals
\begin{equation*}
\begin{split}
	&(-1)^k\mbf P\left(\mbf I_\nu(\phi)\circ\ad^{\got{g}}\circ \mbf{I}_\nu(\psi)-\mbf I_\sigma(\phi)\circ\ad^{\got{g}}\circ \mbf{I}_\sigma(\psi)\right)\\
	&=(-1)^k\mbf P\left((\mbf I_\nu(\phi)-\mbf I_\sigma(\phi))\circ\ad^{\got{g}}\circ \mbf{I}_\nu(\psi)-\mbf I_\sigma(\phi)\circ\ad^{\got{g}}\circ (\mbf{I}_\sigma(\psi)-\mbf I_\nu(\psi))\right)\\
	&=(-1)^k\left(\phi\circ(\pr_\nu-\pr_\sigma)\circ\ad^{\got{g}}\circ \nu\circ\psi-\phi\circ\pr_\sigma\circ\ad^{\got{g}}\circ (\sigma-\nu)\circ \psi\right)\\
	&=(-1)^k\left(\phi\circ(\sigma-\nu)\circ\pi_{\got g/\got i}\circ\ad^{\got{g}}\circ \nu\circ\psi-\phi\circ\pr_\sigma\circ\ad^{\got{g}}\circ (\sigma-\nu)\circ \psi\right).
\end{split}
\end{equation*}
The map $\sigma-\nu$ has image in $\got i$ and $\ad^{\got g}$ preserves $\got i$, so 
$\pr_\sigma\circ\ad^{\got{g}}\circ (\sigma-\nu)=\ad^{\got g}\circ (\sigma-\nu)$. 
Furthermore, $\pi_{\got g/\got i}\circ \ad^{\got g}\circ \nu\circ \psi$ equals $\ad^{\got g/\got i}\circ \psi$ by the definition of $\ad^{\got{g}/\got{i}}\colon \got{g}\to\got{gl}(\got{g}/\got{i})$.
This leads to 
\begin{equation*}
\begin{split}
	&(-1)^k\mbf P\left(\mbf I_\nu(\phi)\circ\ad^{\got{g}}\circ \mbf{I}_\nu(\psi)-\mbf I_\sigma(\phi)\circ\ad^{\got{g}}\circ \mbf{I}_\sigma(\psi)\right)\\
	&=(-1)^k\phi\circ\left((\sigma-\nu)\circ\ad^{\got{g}/\got i}-\ad^{\got{g}}\circ (\sigma-\nu)\right)\circ \psi\\
	&=-(-1)^k\phi\circ\delta^{\ad^{\rm Hom}}_{\got g}(\sigma-\nu)\circ \psi.
\end{split}
\end{equation*}
Similarly, the term of the second sum equals
\begin{equation*}
\begin{split}
&-(-1)^{(k+1)l}\psi\circ\delta^{\ad^{\rm Hom}}_{\got g}(\sigma-\nu)\circ \phi.
\end{split}
\end{equation*}
This shows that \eqref{diff_m2} equals 
\[m_2^\sigma(\psi,\phi)-m_2^\nu(\psi,\phi)=-(-1)^k\phi\circ\delta_{\got g}^{\ad^{\Hom}}(\sigma-\nu)\circ \psi-(-1)^{(k+1)l}\psi\circ\delta_{\got g}^{\ad^{\Hom}}(\sigma-\nu)\circ \phi. \qedhere
\]
\end{proof}

\begin{thm}
Given a  Lie ideal $\got{i}\lhd(\got{g},[\cdot,\cdot])$ together with two choices of splittings $\sigma, \nu\colon \got g/\got i\to \got g$ of \eqref{ses_pi}, 
the $L_\infty[1]$-algebras $(C^{\bullet}(\got{g};\got{i}^*\otimes \got g/\got i),\delta_{\got{g}\rhd\got{i}}^{\Hom}=m_1,m_2^\sigma)$ and $(C^{\bullet}(\got{g};\got{i}^*\otimes \got g/\got i),\delta_{\got{g}\rhd\got{i}}^{\Hom}=m_1,m_2^\nu)$ are isomorphic via the maps
\[ f_1=\id_{C^{\bullet}(\got{g};\got{i}^*\otimes \got g/\got i)}\colon S^1\left(C^{\bullet}(\got{g};\got{i}^*\otimes \got g/\got i)\right) \to C^{\bullet}(\got{g};\got{i}^*\otimes \got g/\got i),
\]
\begin{equation*}
\begin{split}
f_2\colon S^2\left(C^{\bullet}(\got{g};\got{i}^*\otimes \got g/\got i)\right)&\to C^{\bullet}(\got{g};\got{i}^*\otimes \got g/\got i),\\
(\psi, \phi)&\mapsto -\psi\circ (\sigma-\nu)\circ \phi-(-1)^{|\psi||\phi|}\phi\circ (\sigma-\nu)\circ \psi
\end{split}
\end{equation*}
and 
\begin{equation*}
\begin{split}
f_r=0\colon S^r\left(C^{\bullet}(\got{g};\got{i}^*\otimes \got g/\got i)\right)&\to C^{\bullet}(\got{g};\got{i}^*\otimes \got g/\got i),
\end{split}
\end{equation*}
for all $r\geq 3$.
\end{thm}

\begin{proof}
Using the last part of Remark \ref{remark_on_comp_forms}, compute for all $\psi\in\wedge^k\got g^*\otimes \got i^*\otimes \got g/\got i$, $\phi\in\wedge^l\got g^*\otimes \got i^*\otimes \got g/\got i$
\begin{equation*}
\begin{split}
m_1\left(f_2(\psi, \phi)\right)=\,&\,\delta_{\got{g}\rhd\got{i}}^{\Hom}\left(-\psi\circ (\sigma-\nu)\circ \phi-(-1)^{kl}\phi\circ (\sigma-\nu)\circ \psi\right)\\
=\,&-\delta_{\got{g}\rhd\got{i}}^{\Hom}(\psi)\circ(\sigma-\nu)\circ \phi-(-1)^k\psi\circ\delta_{\got g}^{\ad^{\Hom}}(\sigma-\nu)\circ\phi-(-1)^k\psi\circ(\sigma-\nu)\circ\delta_{\got{g}\rhd\got{i}}^{\Hom}(\phi)\\
&\,-(-1)^{kl}\delta_{\got{g}\rhd\got{i}}^{\Hom}(\phi)\circ(\sigma-\nu)\circ\psi
-(-1)^{kl}(-1)^l\phi\circ\delta_{\got g}^{\ad^{\Hom}}(\sigma-\nu)\circ\psi\\
&\,-(-1)^{kl}(-1)^l\phi\circ(\sigma-\nu)\circ\delta_{\got{g}\rhd\got{i}}^{\Hom}(\psi)\\
=\,&-(-1)^k\psi\circ\delta_{\got g}^{\ad^{\Hom}}(\sigma-\nu)\circ\phi-(-1)^{(k+1)l}\phi\circ\delta_{\got g}^{\ad^{\Hom}}(\sigma-\nu)\circ\psi\\
&\,-\delta_{\got{g}\rhd\got{i}}^{\Hom}(\psi)\circ(\sigma-\nu)\circ \phi-(-1)^{(k+1)l}\phi\circ(\sigma-\nu)\circ\delta_{\got{g}\rhd\got{i}}^{\Hom}(\psi)\\
&\,+(-1)^k\left(-\psi\circ(\sigma-\nu)\circ\delta_{\got{g}\rhd\got{i}}^{\Hom}(\phi)
-(-1)^{k(l+1)}\delta_{\got{g}\rhd\got{i}}^{\Hom}(\phi)\circ(\sigma-\nu)\circ\psi
\right)\\
\overset{\eqref{diff_m2s}}{=}&\,m_2^\sigma(\psi,\phi)-m_2^\nu(\psi,\phi)+f_2\left(\delta_{\got{g}\rhd\got{i}}^{\Hom}(\psi), \phi\right)+(-1)^kf_2\left(\psi, \delta_{\got{g}\rhd\got{i}}^{\Hom}(\phi)\right)\\
=\,&\,m_2^\sigma(\psi,\phi)-m_2^\nu(\psi,\phi)+f_2\left(m_1(\psi), \phi\right)+(-1)^k(-1)^{k(l+1)}f_2\left(m_1(\phi),\psi\right).
\end{split}
\end{equation*}
In other words, since $m_1=m_1^\nu=m_1^\sigma$ and $f_1$ is the identity
\begin{equation*}
\begin{split}
m_2^\nu(f_1(\psi),f_1(\phi))+m_1^\nu\left(f_2(\psi, \phi)\right)=f_1(m_2^\sigma(\psi,\phi))+f_2\left(m_1^\sigma(\psi), \phi\right)+(-1)^{kl}f_2\left(m_1^\sigma(\phi),\psi\right).
\end{split}
\end{equation*}
Since $m_1=m_1^\nu=m_1^\sigma$ and $f_1$ is the identity, there is nothing further to check and $(f_1,f_2, f_r=0 \mid r\geq 3)$ is an isomorphism of $L_\infty[1]$-algebras, see Definition \ref{iso_L_infty}.
\end{proof}

%
%
%
%

\subsection{The controlling $L_{\infty}[1]$-algebra of simultaneous small deformations}
In the situation of the previous section, using Theorem \ref{theorem Voronov for L-infty algebra on the mapping cone}, 
there is a cubic $L_{\infty}[1]$-algebra structure $\{\widetilde{m_i}\}_{i=1}^3$ on the direct sum
\begin{equation}\label{L_infty_simul1}
\mbf{g}[1]\oplus\got{a}\eqdef C^{\bullet}(\got{g};\got{g})[2]\oplus C^{\bullet}(\got{g};\got{gl}(\got{g}))[1]\oplus C^{\bullet}(\got{g};\got{i}^*\otimes \got g/\got i),
\end{equation}
where, for any\footnote{Here $|x|$ stands for the degree of $x$ as an element of $\mbf g$.} $x,y\in\mbf{g}$ and $a,a_1,a_2\in\got{a}$, the brackets are defined as follows:
\begin{equation}\label{L_infty_simul2}
\begin{split}
\widetilde{m_1}(x+a)&\eqdef-\left\llbra \mu_{\got{g}}+\ad^{\got{g}},x\right\rrbra+ \mbf{P}(x)+ m_1(a)\\
\widetilde{m_2}(x,y)&\eqdef(-1)^{|x|}\llbra x,y\rrbra\\
\widetilde{m_2}(x,a)&\eqdef\mbf{P}\llbra x,\mbf{I}(a)\rrbra\\
\widetilde{m_2}(a_1,a_2)&\eqdef m_2(a_1,a_2)\\
\widetilde{m_3}(x,a_1,a_2)&\eqdef\mbf{P}\llbra\llbra x,\mbf{I}(a_1)\rrbra,\mbf{I}(a_2)\rrbra.
\end{split}
\end{equation}

\begin{defin}
Let $(\got g, \mu_{\got g})$ be a Lie algebra and $\got i$ an ideal in this Lie algebra. Choose a splitting $\sigma\colon \got g/\got i\to \got g$ of \eqref{ses_pi}. A \textbf{simultaneous small deformation} of $(\mu_{\got g}, \got i)$ is a pair $(\mu'\in\wedge^2\got g^*\otimes \got g, \got i'\subseteq \got g)$ such that $\mu_{\got g}+\mu'$ is a Lie bracket on $\got g$ and  $\got i'$ a vector subspace of $\got g$ such that $\got g=\got i'\oplus \sigma(\got g/\got i)$ and $\got i'$ is an ideal in  $(\got g, \mu_{\got g}+\mu')$.
\end{defin}

\begin{thm}\label{thm-for-simultaneous-deform}
Let $\got{i}$ be an ideal in a Lie algebra $(\got{g},\mu_{\got{g}})$ and choose  splitting $\sigma\colon \got g/\got i\to \got g$ of \eqref{ses_pi}. Consider $\phi\in \got i^*\otimes \got g/\got i$ and $\mu'\in \wedge^2\got g^*\otimes \got g$. Define $\ad'\in\wedge^1\got{g}^*\otimes\got{gl}(\got{g})$ by $\ad'_x(y)\eqdef\mu'(x,y)$ for all $x,y\in \got g$, and set $\Graph(\sigma\circ \phi)=:\got{i}'$.

Then the pair $(\mu',\got{i}')$ is a simultaneous small deformation of $(\mu_{\got{g}},\got{i})$ if and only if $(\mu',\ad',\phi)$  is a Maurer-Cartan element of the $L_\infty[1]$-algebra $\bigleftpar\mbf{g}[1]\oplus\got{a},\widetilde{m_1},\widetilde{m_2},\widetilde{m_3}\bigrightpar$ defined in \eqref{L_infty_simul1} and \eqref{L_infty_simul2}.
\end{thm}
\begin{proof}
The sum $\mu_{\got{g}}+\mu'$ is a Lie bracket on $\got g$ and $\got{i}'\lhd(\got{g}, \mu_{\got{g}}+\mu')$ is an ideal in the obtained Lie algebra if and only if 
\begin{align*}
\llbra\mu_{\got{g}}+\mu',\mu_{\got{g}}+\mu'\rrbra&=0,  \quad \text{ see \eqref{Lie_Brackets_mu}},\\
(\mu_{\got{g}}+\mu')(x,u+(\sigma\circ\phi)(u))&\in \Graph(\sigma\circ \phi) \quad \text{ for all } x\in \got g, u\in \got i.
\end{align*}
The first condition is equivalent to 
\begin{equation}\label{eq1111}
\llbra\mu_{\got{g}}+\ad^{\got g}+\mu'+\ad',\mu_{\got{g}}+\ad^{\got g}+\mu'+\ad'\rrbra=0
\end{equation}
by Section \ref{graded_LA_LA+rep}, since $\ad^{\got g}+\ad'$ is then the adjoint representation of $\got g$ associated to the Lie bracket $\mu_{\got g}+\mu'$.
The second condition is equivalent to the following equality for all $x\in\got{g}$ and $u\in\got{i}$:
\begin{equation}\label{eq1112}
\begin{split}
&\pi_{\got g/\got i}(\mu_{\got{g}}(x,(\sigma\circ\phi)(u)))+\pi_{\got g/\got i}(\mu'(x,u))+\pi_{\got g/\got i}(\mu'(x,(\sigma\circ\phi)(u)))\\
&=\phi(\mu_{\got{g}}(x,u))+\phi(\pr_{\sigma}(\mu'(x,u)))+\phi(\pr_{\sigma}(\mu_{\got{g}}(x,(\sigma\circ\phi)(u))))+\phi(\pr_{\sigma}(\mu'(x,(\sigma\circ\phi)(u)))).
\end{split}
\end{equation}

\bigskip

The Maurer-Cartan equation for \[\mu'+\ad'+\phi\in C^2(\got{g};\got{g})[2]\oplus C^1(\got{g};\got{gl}(\got{g}))[1]\oplus C^0(\got{g};\got{i}^*\otimes \got g/\got i)\]
is 
\begin{equation}\label{MC_sim}
\widetilde{m_1}\left(\mu'+\ad'+\phi\right)+\frac{1}{2}\widetilde{m_2}\left(\mu'+\ad'+\phi,\mu'+\ad'+\phi\right)+\frac{1}{6}\widetilde{m_3}\left(\mu'+\ad'+\phi,\mu'+\ad'+\phi,\mu'+\ad'+\phi\right)=0.
\end{equation}
Compute 
\begin{equation*}
\begin{split}
\widetilde{m_1}\left(\mu'+\ad'+\phi\right)&\,=-\left\llbra \mu_{\got g}+\ad^{\got g}, \mu'+\ad'\right\rrbra+\mbf P(\mu'+\ad')+m_1(\phi)\\
&\,=-\left\llbra \mu_{\got g}+\ad^{\got g}, \mu'+\ad'\right\rrbra+\mbf P(\ad')+\mbf P\left\llbra\mu_{\got g}+\ad^{\got g}, \mbf I(\phi)\right\rrbra\\
&\overset{\eqref{eq2_LALArep}}{=}-\left\llbra \mu_{\got g}+\ad^{\got g}, \mu'+\ad'\right\rrbra+\mbf P(\ad')+\mbf P\left\llbra\ad^{\got g}, \mbf I(\phi)\right\rrbra,
\end{split}
\end{equation*}
\begin{equation*}
\begin{split}
&\frac{1}{2}\widetilde{m_2}\left(\mu'+\ad'+\phi,\mu'+\ad'+\phi\right)=-\frac{1}{2}\left\llbra \mu'+\ad', \mu'+\ad'\right\rrbra+\tilde m_2\left(\mu'+\ad', \mbf I(\phi)\right)+\frac{1}{2}m_2(\phi,\phi)\\
&\,=-\frac{1}{2}\left\llbra \mu'+\ad', \mu'+\ad'\right\rrbra+\mbf P\left\llbra \mu'+\ad', \mbf I(\phi)\right\rrbra+\frac{1}{2}\mbf P\left\llbra \left\llbra \mu_{\got g}+\ad^{\got g}, \mbf I(\phi)\right\rrbra, \mbf I(\phi)\right\rrbra\\
&\overset{\eqref{eq2_LALArep}}{=}-\frac{1}{2}\left\llbra \mu'+\ad', \mu'+\ad'\right\rrbra+\mbf P\left\llbra \ad', \mbf I(\phi)\right\rrbra+\frac{1}{2}\mbf P\left\llbra \left\llbra \ad^{\got g}, \mbf I(\phi)\right\rrbra, \mbf I(\phi)\right\rrbra
\end{split}
\end{equation*}
and 
\begin{equation*}
\begin{split}
\frac{1}{6}\widetilde{m_3}\left(\mu'+\ad'+\phi,\mu'+\ad'+\phi,\mu'+\ad'+\phi\right)&\,=\frac{1}{2}\tilde m_3(\mu'+\ad', \phi, \phi)=\frac{1}{2}\mbf P\left\llbra \left\llbra \mu'+\ad', \mbf I(\phi)\right\rrbra, \mbf I(\phi)\right\rrbra\\
&\overset{\eqref{eq2_LALArep}}{=}\frac{1}{2}\mbf P\left\llbra \left\llbra \ad', \mbf I(\phi)\right\rrbra, \mbf I(\phi)\right\rrbra.
\end{split}
\end{equation*}
Using $\left\llbra \mu_{\got g}+\ad_{\got g}, \mu_{\got g}+\ad_{\got g}\right\rrbra=0$ (see Section \ref{graded_LA_LA+rep}) and $\left\llbra \mu_{\got g}+\ad^{\got g}, \mu'+\ad'\right\rrbra=\left\llbra \mu'+\ad',  \mu_{\got g}+\ad^{\got g}\right\rrbra$, the left-hand side of \eqref{MC_sim} now reads
\begin{equation*}
\begin{split}
&-\frac{1}{2}\left\llbra \mu_{\got g}+\ad^{\got g}+ \mu'+\ad', \mu_{\got g}+\ad^{\got g}+ \mu'+\ad'\right\rrbra\\
&+ \mbf P(\ad')+\mbf P\left\llbra \ad^{\got g}+\ad', \mbf I(\phi)\right\rrbra+\frac{1}{2}\mbf P\left\llbra \left\llbra \ad^{\got g}+ \ad', \mbf I(\phi)\right\rrbra, \mbf I(\phi)\right\rrbra.
\end{split}
\end{equation*}
The first term has degree $2$ in $\mbf{g}\eqdef C^{\bullet}(\got{g};\got{g})[1]\oplus C^{\bullet}(\got{g};\got{gl}(\got{g}))$, and the remaining terms have degree $1$.
Hence \eqref{MC_sim} holds if and only if \eqref{eq1111} holds true
and 
\[\mbf P(\ad')+\mbf P\left\llbra\ad^{\got g}+\ad', \mbf I(\phi)\right\rrbra+\frac{1}{2}\mbf P\left\llbra \left\llbra \ad^{\got g}+ \ad', \mbf I(\phi)\right\rrbra, \mbf I(\phi)\right\rrbra=0.
\]
A straightforward computation using \eqref{eq3_LALArep} shows that the latter equation reads
\begin{align*}
&\pi_{\got g/\got i}\left(\ad'_{x}(u)\right)+\pi_{\got g/\got i}\left(\ad^{\got g}_x((\sigma\circ\phi)(u))+\ad'_x((\sigma\circ\phi)(u))\right)
-(\phi\circ \pr_{\sigma})\left(\ad^{\got g}_x(u)+\ad'_x(u)\right)\\
&-\phi\left(\pr_{\sigma}\left(\ad^{\sigma}_x((\sigma\circ\phi)(u))\right)\right)-\phi\left(\pr_{\sigma}\left(\ad'_x((\sigma\circ\phi)(u))\right)\right)=0
\end{align*}
on $x\in \got g$ and $u\in \got i$. Since $\mu_{\got g}(x,u)$ lies in $\got i$ and by definition of $\ad'$, this is \eqref{eq1112}.
\end{proof}

\begin{rem}
As pointed out by the referee, it is worth noting that if one drops the assumption in Proposition \ref{lemma for V-data of an ideal} that $\got{i}$ be a Lie ideal, and assumes only that it is a vector subspace of the Lie algebra $\got{g}$, then the quadruple $(\mbf{g},\got{a},P,\Theta)$ is in fact a curved Voronov dataset. In this more general setting, it still holds that $\phi\colon \got{i}\to\got g/\got{i}$ is a Maurer--Cartan element if and only if $\Graph(\sigma\circ \phi)$ is a Lie ideal, thereby generalizing Theorem \ref{dgl1a_ideal}. In light of this observation, Theorem \ref{thm-for-simultaneous-deform} may be viewed as a straightforward application of \cite[Theorem 3]{Zambon-Fregier-SimultaneousDefOfAlgebraAndMorphisms-15'}.
\end{rem}

\section{Geometric applications: obstructions, rigidity and stability}\label{geometric_results_ideals}

This section focusses on applying the machinery built above to a study of the local geometry of the (moduli) space of Lie ideals in a given Lie algebra.

Let $\got g$ be a Lie algebra and choose $k\in\{0,\ldots, \dim\got g\}$.
Denote by $\bm{I}_k(\got{g})$ the subset of the Grassmannian $\Gr_k(\got{g})$ consisting of the space of $k$-dimensional Lie ideals inside $(\got{g},\mu_{\got{g}})$. 

\subsection{The Kuranishi map and obstructions to deformations of ideals}
Recall that the deformation cohomology \[H^\bullet_{\delta_{\got{g}\rhd\got{i}}^{\Hom}}(\got{g};\got{i}^*\otimes\got{g}/\got{i})=:H^\bullet(\got i\lhd \got g)\] of an ideal $\got i$ in a Lie algebra $\got g$ was defined in Section \ref{def_coh_ideal}, and that the underlying complex $(C^\bullet(\got g; \got i^*\otimes \got g/\got i, \delta_{\got{g}\rhd\got{i}}^{\Hom})$ fits in the dgL$[1]$a $(C^\bullet(\got g; \got i^*\otimes \got g/\got i), \delta_{\got{g}\rhd\got{i}}^{\Hom}=m_1, m_2)$ of Theorem \ref{dgl1a_ideal} -- writing now simply $m_2$ for the binary bracket $m_2^\sigma$ defined by a splitting $\sigma\colon \got g/\got i\to \got g$ of \eqref{ses_pi}. Note that in this section the choice $\sigma$ is sometimes replaced by the (equivalent) choice of a complement $\mathfrak i^c$ of $\got i$ in $\got g$.
\begin{defin}
Let $\got g$ be a Lie algebra and let $\got i\lhd{\got{g}}$ be an ideal in $\got g$.
The \textbf{Kuranishi map 
\[\Kur_{\got{i}\lhd\got{g}}\colon H^0(\got{i}\lhd\got{g})\to H^1(\got{i}\lhd\got{g})
\]
associated to the Lie ideal $\got{i}\lhd{\got{g}}$} is defined as follows:
\begin{align*}
[\eta]&\mto\frac{1}{2}[m_2(\eta,\eta)].
\end{align*}
\end{defin}
This is well-defined since $m_1=\delta^{\rm Hom}_{\got i\lhd \got g}\colon C^\bullet(\got g; \got i^*\otimes \got g/\got i)\to C^{\bullet+1}(\got g; \got i^*\otimes \got g/\got i)$ and 
\[ m_1(m_2(\eta,\eta))=2m_2(m_1(\eta), \eta)
\]
for $\eta\in C^0(\got g; \got i^*\otimes \got g/\got i)$.

The following proposition is standard in deformation theory, but its proof is repeated here for the convenience of the reader.
\begin{prop}
If $\Kur_{\got{i}\lhd\got{g}}\neq 0$, then the deformation problem of $\got{i}\lhd\got{g}$ is \emph{obstructed}: there exists a cohomology class $[\eta]\in H^0(\got{i}\lhd\got{g})$ that is not a deformation class -- in other words, there exists no smooth deformation $(\got{i}_t)_{t\in I}$ of the ideal $\got{i}$ inside $\got{g}$ such that $\left.\frac{d}{dt}\right\arrowvert_{t=0}\got i_t=\eta$.
\end{prop}
\begin{proof}
Let $\got{i}^c$ a complement of $\got{i}\lhd\got{g}$ and let $(\phi_t)_{t\in I}\in\got{i}^*\otimes \got{i}^c$ be a smooth family of linear maps starting at the zero map. Then $\got{i}_t\eqdef\Graph(\phi_t)$ is a smooth deformation of the vector space $\got{i}\eqdef\got{i}_0$.	Consider the Taylor expansion of $\phi_t$ around $t=0$:
\begin{equation*}
\phi_t\simeq t\eta+t^2\omega+\mbf{O}(t^3)
\end{equation*}
with $\eta, \omega\in \got i^*\otimes {\got i^c}$.
The following shows that the condition of $\Graph(\phi_t)$ being an ideal in $\got{g}$ for all $t$ forces the Kuranishi map to vanish on $\eta=\left.\frac{d}{dt}\right\arrowvert_{t=0}\phi_t$. Let $x\in\got{g}$ and $u\in\got{i}$, then 
\begin{align*}
[x,u+\phi_t(u)]&=[x,u]+t[x,\eta(u)]+t^2[x,\omega(u)]+\mbf{O}(t^3)\in \Graph(\phi_t)
\end{align*}
for all $t$
induces
\begin{equation*}	t\pr_{\got i^c}([x,\eta(u)])+t^2(\pr_{\got i^c}[x,\omega(u)])=t\eta([x,u])+t^2\omega([x,u])+t^2\eta(\pr_{\got{i}}[x,\eta(u)])+\mbf{O}(t^3),
\end{equation*}
for all $t$,
which is equivalent to the following two equations:
\begin{align*}
\pr_{\got i^c}([x,\eta(u)])-\eta([x,u])&=0\\
\pr_{\got i^c}([x,\omega(u)])-\eta(\pr_{\got{i}}[x,\eta(u)])-\omega([x,u])&=0.
\end{align*}
Since $x\in \got g$ and $u\in \got i$ where arbitrary, the first equation
recovers the fact that $\delta_{\got i\lhd \got g}^{\rm Hom}(\eta)=0$, while the second equation and \eqref{m_2_on_same_deg0} yield $m_2(\eta,\eta)=-2\cdot \delta_{\got i\lhd \got g}^{\rm Hom}(\omega)$.
Hence \[\Kur_{\got{i}\lhd{\got{g}}}([\eta])=\frac{1}{2}[m_2(\eta,\eta)]=\left[\delta_{\got{g}\rhd{\got{i}}}^{\Hom}(-\omega)\right]=0.\] 
This shows that the Kuranishi map vanishes on deformation classes.
\end{proof}

\begin{ex}[The case $H^2(\got{g}/\got{i};\got{g}/\got{i})=0$]
Recall that Proposition \ref{the lemma where we defined the cochain map from def. complex of can projection to the def. complex of ideals} shows that  the cochain map $$\Pi_{\wedge}\colon C^{\bullet+1}(\got{g};\got{g}/\got{i})\twoheadrightarrow  C_{\wedge}^{\bullet}(\got{g};\got{i}^*\otimes\got{g}/\got{i}), \quad \phi\mto\phi|_{\wedge^k\got{g}\wedge\got{i}}$$ sends the deformation cocycle associated to a deformation $(\pi^t_{\got{g}/\got{i}})_{t\in I}$ of the projection $\pi_{\got g/\got i}\colon \got g\to \got g/\got i$ to the deformation cocycle defined by the deformation of the ideal $\got i$ defined (locally) by the kernels $\left(\Ker(\pi^t_{\got{g}/\got{i}})\right)_{t\in I}$.

Combining the assumption $H^2(\got{g}/\got{i};\got{g}/\got{i})=0$ (e.g.~when $\got{i}=\text{Rad}(\got{g})$ or $\got{g}/\got{i}$ is semi-simple) and \eqref{commutative diagram where the def. cohomology of an ideal lives}
shows that \[H^0(\Pi_{\wedge})\colon H^1(\got{g};\got{g}/\got{i})\to H^0_\wedge(\got g; \got i^*\otimes \got g/\got i)=H^0(\got{i}\lhd\got{g})\] is surjective. This implies that  $\got{i}\lhd\got{g}$ is (topologically) rigid, see Definition \ref{def_top_rig} and Theorem \ref{theorem for Aut(g)-rigidity of ideals}
below. 

In particular every deformation class of the ideal $\got{i}\lhd\got{g}$ is the image under $H^0(\Pi_{\wedge})$ of a cohomology class in the deformation cohomology of the canonical projection $\pi_{\got{g}/\got{i}}\colon \got{g}\to\got{g}/\got{i}$.

When $H^2(\got{g}/\got{i};\got{g}/\got{i})=0$ in fact every \emph{smooth} deformation of the ideal $\got{i}\lhd{\got{g}}$ arises as the kernel of a smooth deformation of the canonical projection. To see this, use the fact that $(\got{g}/\got{i},\mu_{\got{g}/\got{i}})$ is (geometrically) rigid and so every smooth deformation of the Lie bracket on $\got{g}/\got{i}$ is (geometrically) equivalent to the trivial one (see \cite{Nijenhuis-Richardson-Deformations-of-Lie-algebra-structures-67'}, or \cite[Section 4.1]{Crainic-preprint-on-the-perturbation-lemma-and-deformations-2004}). Let $(\got{i}_t)_{t\in I}$ be a smooth deformation of the ideal $\got{i}\lhd\got{g}$. At each time $t\in I$, the ideal $\got i_t$ defines the canonical projections $\pi_{\got{g}/\got{i}_t}\colon \got{g}\to\got{g}/{\got{i}_t}$. The Lie algebra structures on $\got{g}/{\got{i}_t}$, $t\in I$, are then transferred to Lie algebras $(\got{g}/\got{i},\mu_{\got{g}/\got{i}}^t)$ via a smooth family of linear isomorphisms denoted by $\alpha_t\colon \got g/\got i_t \to \got g/\got i$ and satisfying $\alpha_0=\id_{\got g/\got i}$. Then  $(\mu_{\got{g}/\got{i}}^t)_{t\in I}$ is a smooth family of deformations of the Lie bracket $\mu_{\got{g}/\got{i}}$ induced by the transfer. Using the fact that $\got{g}/\got{i}$ is rigid, there exists a smooth family of Lie algebra isomorphisms $\phi_t\colon (\got{g}/{\got{i}},\mu_{\got{g}/\got{i}}^t)\to(\got{g}/\got{i},\mu_{\got{g}/\got{i}})$, with $\phi_0=\id_{\got g/\got i}$. Then the maps \[\pi^t_{\got{g}/\got{i}}\eqdef\phi_t\circ\alpha_t\circ\pi_{\got{g}/\got{i}_t}\colon \got{g}\to\got{g}/\got{i},\] for all $t\in I$, define a smooth deformation of the canonical projection $\pi_{\got{g}/\got{i}}$, such that $\Ker(\pi^t_{\got{g}/\got{i}})=\got{i}_t$ for all $t\in I$.
\end{ex}

The following example illustrates once more how the deformation theory of an ideal is \emph{not} its deformation theory as a Lie subalgebra.
\begin{ex}[Obstructed as an ideal $\centernot{\Longrightarrow}$ obstructed as Lie subalgebra]
Let $\got{h}_3(\mbb{R})\eqqcolon\got{g}$ be the 3-dimensional Heisenberg algebra, in other words the Lie algebra of $(3\times 3)$-strictly upper triangular matrices. The center $\Cent(\got{h}_3(\mbb{R}))\eqqcolon\got{i}$ of $\got{h}_3(\mbb{R})$ is the $1$-dimensional ideal generated by the basis vector:
\begin{equation*}
\begin{pmatrix}
0 & 0 & 1\\
0 & 0 & 0\\
0 & 0 & 0
\end{pmatrix}
\end{equation*}
Furthermore, consider the complement $\got i^c\subset\got{h}_3(\mbb{R})$, which is generated by the other two  canonical basis vectors:
\begin{equation*}
\begin{pmatrix}
0 & 1 & 0\\
0 & 0 & 0\\
0 & 0 & 0
\end{pmatrix} \qquad \text{and}  \qquad 	\begin{pmatrix}
0 & 0 & 0\\
0 & 0 & 1\\
0 & 0 & 0
\end{pmatrix}.
\end{equation*}
The inclusions $[\got{h}_3(\mbb{R}),\Cent(\got{h}_3(\mbb{R}))]=0$ and $[\got{h}_3(\mbb{R}), \got i^c]\subset\Cent(\got{h}_3(\mbb{R}))$ imply that any $0$-cochain is a $0$-cocycle: for any $\eta\in \got i^*\otimes\got g/\got i$ and any $x\in\got{g}=\got{h}_3(\mbb{R})$ and $u\in\got{i}=\Cent(\got{h}_3(\mbb{R}))$: $$\delta_{\got{i}\rhd\got{g}}^{\Hom}(\eta)(x,u)=\pr_{\got i^c}(\underbrace{[x,\eta(u)]}_{\in\got{i}})-\eta(\underbrace{[x,u]}_{=0})=0.$$	
The Kuranishi map  is hence given here by
\begin{equation*}
\Kur_{\got{i}\lhd{\got{g}}}\colon C^0(\got{i}\lhd\got{g})\to Z^1(\got{i}\lhd\got{g}), \quad \Kur_{\got{i}\lhd\got{g}}(\eta)=\frac{1}{2}m_2(\eta,\eta)\end{equation*}
with 
\[ m_2(\eta, \eta)(x,u)=-2\eta\left(\pr_{\got i}[x, \eta(u)]\right)=-2\eta\left([x, \eta(u)]\right)
\]
for all $x\in\got{h}_3(\mbb{R})$ and all $u\in \got i$.

A linear map $\phi\colon \got{i}\to \got i^c$ is uniquely defined by
\begin{equation*}
\begin{pmatrix}
0 & 0 &1\\
0 & 0 & 0\\
0 & 0 & 0
\end{pmatrix} \stackrel{\phi}{\mapsto}	\begin{pmatrix}
0 & \alpha & 0\\
0 & 0 & \beta\\
0 & 0 & 0
\end{pmatrix}
\end{equation*}
with $\alpha,\beta\in\mathbb R$.
It is a $0$-cocycle, but the $1$-cocycle $\Kur_{\got{i}\lhd{\got{g}}}(\phi)=\frac{1}{2}m_2(\phi,\phi)$ does not vanish (unless $\phi=0$) since it does 
\[
\left(\begin{pmatrix} 0&x& 0\\ 0& 0& y\\ 0&0&0\end{pmatrix}, \begin{pmatrix} 0&0& 1\\ 0& 0& 0\\ 0&0&0\end{pmatrix}\right)\mapsto -\begin{pmatrix} 0&\alpha^2y-\alpha\beta x& 0\\ 0& 0& \alpha\beta y-\beta^2x\\ 0&0&0\end{pmatrix}
\]
for all $x,y\in\mathbb R$.
Therefore, the deformation problem of the ideal $\Cent(\got{h}_3(\mbb{R}))\lhd\got{h}_3(\mbb{R})$ is obstructed: the ideal $\got i=\Cent(\got{h}_3(\mbb{R}))$
of $\got h_3(\mathbb R)$ admits no smooth deformation. On the other hand, since $\got i=\Cent(\got{h}_3(\mbb{R}))$ has dimension $1$, any deformation of the vector subspace $\got i$ is a deformation of $\got i$ as a Lie subalgebra of $\got{h}_3(\mbb{R})$. 
Hence, the deformation problem of the Lie subalgebra $\Cent(\got{h}_3(\mbb{R}))\subseteq \got{h}_3(\mbb{R})$ is unobstructed.
\end{ex}

Denote by $\mbf{S}_k(\got{g})$ the subset of $\Gr_k(\got{g})$ consisting of the space of $k$-dimensional Lie subalgebras of $\got{g}$. A $k$-dimensional Lie ideal $\got{i}\lhd\got{g}$ is a point in $\mbf{I}_k(\got{g})\subset\mbf{S}_k(\got{g})$. The local geometry of $\got{i}\lhd\got{g}$ as a point in the bigger space $\mbf{S}_k(\got{g})$ and as a point in the smaller $\mbf{I}_k(\got{g})$ is hence different in general.

\subsection{Rigidity of ideals in Lie algebras}
This section is inspired from the study in \cite{Crainic-Ivan-Schaetz} of the rigidity of Lie algebras, Lie algebra morphisms and Lie subalgebras. Similar techniques are developed to obtain a rigidity result for Lie ideals.

Let $(E\to M, G)$ be a \textbf{$G$-vector bundle}, i.e.~a vector bundle $E\to M$, together with a Lie group $G$ acting on it by vector bundle automorphisms. That is, the smooth Lie group action $$\alpha\colon G\times E\to E, \quad \alpha(g,e)\eqqcolon ge\eqqcolon \alpha_g(e)$$ restricts to a Lie group action on the zero section $M$, denoted by $\alpha^{\res}\colon G\times M\to M$ such that, for all $g\in G$, the map $\alpha_g\colon E\to E$ is a vector bundle automorphism over $\alpha^{\res}_g\colon M\to M$. 
\begin{rem}
The restriction $T_{0^E}E$ of the tangent space of $E\to M$ along its zero section is canonically isomorphic as a vector bundle over $M$ to $TM\oplus E$, via the isomorphism
\[TM\oplus E\to T_{0^E}E, \qquad (v_m,e_m)\mapsto T_m0^E(v_m)+\left.\frac{d}{dt}\right\arrowvert_{t=0}te_m\] In other words, 
the short exact sequence 
\begin{equation*}
\begin{tikzcd}
E\ar[r, hookrightarrow] & T_{0^E}E \ar[r, twoheadrightarrow, "Tq"] & TM,
\end{tikzcd}
\end{equation*}
with the inclusion $E\hookrightarrow T_{0^E}E$, $e\to \left.\frac{d}{dt}\right\arrowvert_{t=0}te$, 
is canonically split by the tangent of the zero section $T0^E\colon TM\to T_{0^E}E$.
\end{rem}
\begin{defin}
Let $(E\to M, G)$ be a $G$-vector bundle. A section $\sigma\colon M\to E$ is called \textbf{G-equivariant} if $\sigma(gx)=g\sigma(x)$ for any $g\in G$ and $x\in M$.
The \textbf{vertical tangent map} \[T_x^{\ver}\sigma\colon T_xM\to E_x\] 
\textbf{of $\sigma\colon M\to E$ at a zero $x\in\sigma^{-1}(0)\eqdef\{x\in M\ | \ \sigma(x)=0^E_x\}$} is the map $T_x^{\ver}\sigma\colon T_xM\to E_x$
defined as the composition of $T_x\sigma\colon T_xM\to T_{0^E_x}E\simeq T_xM\oplus E_x$, followed by the projection onto $E_x$.
A zero $x\in M$ of $\sigma$ is called \textbf{non-degenerate} if the sequence 
\begin{equation}\label{the sequence used for defining non-degenerace of zeros of a section}
\begin{tikzcd}
\got{g} \ar[r, "T\alpha^x"] & T_xM \ar[r, "T_x^{\ver}\sigma"] & E_x
\end{tikzcd}
\end{equation}
is exact, 
where $\alpha^x \colon G\to M$ is the orbit map of $\alpha^{\res}$ at $x\in M$.
\end{defin}
The following proposition on the openness of orbits is proved in \cite{Crainic-Ivan-Schaetz}.
\begin{prop}\label{proposition of openess of orbits}
Let $(E\to M, G)$ be a $G$-vector bundle, let $\sigma\colon M\to E$ be a $G$-equivariant section and assume that $x\in\sigma^{-1}(0)$ is nondegenerate. Then there is an open neighborhood $U$ of $x$ and a smooth map $h\colon U\to G$ such that for all $y\in U$ with $\sigma(y)=0$ the equality $h(y)x=y$ holds. In particular, the orbit of $x$ under the $G$-action and the zero set of $\sigma$ coincide in an open neighborhood of $x$.
\end{prop}

\bigskip

Let $\got g$ be a Lie algebra and let $\got i$ be an ideal of $\got g$ of dimension $k$.
Consider the smooth vector bundle $E\to M$, with base manifold $M\eqdef\Gr_k(\got{g})$ and fiber over $W\in\Gr_{k}(\got{g})$ given by $E_W\eqdef \{W\}\times\Hom(\got{g}, W^*\otimes\got{g}/W)$.  Let $\operatorname{Taut}_ k(\got g)\to \Gr_k(\got g)$ be the tautological vector bundle, i.e.~$\operatorname{Taut}_ k$ is the rank $k$ vector subbundle of $\Gr_k(\got g)\times \got g$ given by $\operatorname{Taut}_k(\got g)\arrowvert_{\{W\}}=\{W\}\times W$. Then $E\to \Gr_k(\got g)$ is the smooth vector bundle
\begin{equation}\label{def_E}
E=(\Gr_k(\got g)\times\got g^*)\otimes (\operatorname{Taut}_k(\got g))^*\otimes \frac{\Gr_k(\got g)\times\got g}{\operatorname{Taut}_k(\got g)} \longrightarrow \Gr_k(\got g).
\end{equation}
The dual vector bundle $\Taut_k^*(\got g)$ is the quotient of $\Gr_k(\got g)\times \got g^*$ by its smooth vector subbundle 
\[(\Taut_k(\got g))^\circ=\cup_{W\in\Gr_k(\got g)}\{W\}\times \{l\in \got g^*\mid l(w)=0 \text{ for all } w\in W\}.
\]
The vector bundle $E\to \Gr_k(\got g)$ comes as a consequence with the smooth vector bundle projection 
\[\begin{tikzcd}
{\Gr_k(\got g)\times (\got g^*\otimes\got g^*\otimes \got g)} &&& E \\
&& {\Gr_k(\got g)}
\arrow["P", from=1-1, to=1-4]
\arrow["{\pr_1}"', from=1-1, to=2-3]
\arrow["\pi_E", from=1-4, to=2-3]
\end{tikzcd}\]
sending $(W,\phi)\in \Gr_k(\got g)\times (\got g^*\otimes\got g^*\otimes \got g)$ to 
\[ (W, \pi_{\got g/W}\circ \phi\arrowvert_{\got g\otimes W})\in E\arrowvert_W,
\]
see Appendix \ref{triv_E}, where local trivialisations of $E$ are given.

The section
\begin{equation}\label{def_sigma_section}
\sigma\colon M\to E, \quad W\mto\left(W, \pi_{\got g/W}\circ \mu_{\got g}\arrowvert_{\got g\otimes W}\right)
\end{equation}
is smooth since it can be written as $P\circ \tilde \sigma\colon \Gr_k(\got g)\to E$, with the constant section 
\[ \tilde \sigma\colon \Gr_k(\got g)\to \Gr_k(\got g)\times (\got g^*\otimes \got g^*\otimes \got g), \qquad W\mapsto (W, \mu_{\got g}).
\]
The zero set $\sigma^{-1}(0)$ of $\sigma$ is exactly the space $\bm{I}_k(\got{g})\subset\Gr_k(\got{g})$.

\begin{rem}\label{sigma_to_sec_general}
Note that this construction is independent of $\mu_{\got g}$ satisfying the Jacobi identity (or being skew-symmetric). In fact, for each $\nu\in \wedge^2\got g^*\otimes \got g\simeq \Hom(\wedge^2\got g, \got g)$ the constant section 
\[ \tilde \nu\colon  \Gr_k(\got g)\to \Gr_k(\got g)\times (\got g^*\otimes \got g^*\otimes \got g), \qquad W\mapsto (W, \nu)
\]
projects as above to a smooth section \[\Sigma(\nu)\colon \Gr_k(\got g) \to E, \qquad  W\mto\left(W, \pi_{\got g/W}\circ \nu\arrowvert_{\got g\otimes W}\right)\] such that the diagram 
\[\begin{tikzcd}
{\Gr_k(\got g)\times (\got g^*\otimes \got g^*\otimes \got g)} && E \\
\\
& {\Gr_k(\got g)}
\arrow["P", from=1-1, to=1-3]
\arrow["{\tilde \nu}"{description}, from=3-2, to=1-1]
\arrow["{\Sigma(\nu)}"{description}, from=3-2, to=1-3]
\end{tikzcd}\]
commutes.
The map 
\[ \Sigma\colon \wedge^2\got g^*\otimes \got g \to \Gamma(E)
\]
is then an $\mathbb R$-linear map.
\end{rem}

\bigskip

The following computes explicitly the vertical differential $T_{\got{i}}^{\ver}\sigma\colon T_{\got{i}}\Gr_k(\got{g})\to E_{\got{i}}$. 
The following lemma is useful for this.
\begin{lem}
Let $\pi\colon E\to M$ be a smooth vector bundle with a vector bundle isomorphism $\phi\colon E\to M\times E\arrowvert_{p}$
for some $p\in M$, such that $\phi\arrowvert_{p}\colon E\arrowvert_{p}\to \{p\}\times E\arrowvert_{p}$ is the identity. If $p$ is a zero of a smooth section $\sigma\colon M\to E$, then 
\[ T_p^{\ver}\sigma=T_p^{\ver}(\phi\circ \sigma)\colon T_pM\to E\arrowvert_{p}, 
\]
and so 
\[ T_p(\phi\circ\sigma)(v_p)=\left(v_p, 0^E_p, T_p^{\ver}\sigma(v_p)\right)\in T_pM\times \left\{0^E_p\right\}\times E\arrowvert_{p}=T_pM\times T_{0^E_p}(E\arrowvert_{p})
\]
for all $v_p\in T_pM$.
\end{lem}
\begin{proof}
By definition, $T_p^{\ver}\sigma\colon T_pM\to E\arrowvert_{p}$ sends $v_p\in T_p M$ to $T_p\sigma v_p-T_p0^Ev_p\in E\arrowvert_{p}=T^\pi_{0^E_p}E$. As a consequence, 
\begin{equation*}
\begin{split}
T^{\ver}_p(\phi\circ\sigma)(v_p)&=T_p(\phi\circ\sigma)(v_p)-T_p0^{M\times E\arrowvert_{p}}(v_p)
=T_{0^E_p}\phi(T_p\sigma v_p)-T_{0^E_p}\phi(T_p0^E v_p)\\
&=T_{0^E_p}\phi(T_p\sigma v_p -T_p0^Ev_p)=T_{0^E_p}\phi\left(T^{\ver}_p\sigma v_p
\right)
\end{split}
\end{equation*}
for all $v_p\in T_pM$. Choose $e_p\in E\arrowvert_{p}$ and consider the corresponding 
vertical vector 
\[ \left.\frac{d}{dt}\right\arrowvert_{t=0}te_p\in T_{0^E_p}E.
\]
Then 
\[T_{0^E_p}\phi\left(\left.\frac{d}{dt}\right\arrowvert_{t=0}te_p
\right)=\left.\frac{d}{dt}\right\arrowvert_{t=0}t\phi(e_p)=\left.\frac{d}{dt}\right\arrowvert_{t=0}te_p
\]
since $\phi\arrowvert_{p}\colon E\arrowvert_{p}\to \{p\}\times E\arrowvert_{p}$ is the identity map. This shows that 
\[T^{\ver}_p\sigma v_p=T_{0^E_p}\phi\left(T^{\ver}_p\sigma v_p
\right)=T^{\ver}_p(\phi\circ\sigma)(v_p)
\]
for all $v_p\in T_pM$.

The second statement then follows immediately.
\end{proof}

The last lemma shows that the map $T_{\got{i}}^{\ver}\sigma\colon T_{\got{i}}\Gr_k(\got{g})\to E_{\got{i}}$ can be computed in a chart $U:=U^{\got i, \got i^c}\simeq \got i^*\otimes \got i^c$ given by a choice of complement $\got i^c$ for $\got i$ in $\got g$, since 
it induces a trivialisation 
\[ E\arrowvert_{U^{\got i, \got i^c}}\to U^{\got i, \got i^c}\times(\got g^*\otimes \got i^*\otimes \got g/\got i)
\] 
that is the identity on $E\arrowvert_{\got i}$. By \eqref{triv_E_2},
the image under the isomorphism $\Phi^{-1}\colon E\arrowvert_{U}\to U\times (\got g^*\otimes \got i^*\otimes \got g/\got i) $
of $\sigma\arrowvert_{U}\colon U\to E\arrowvert_{U}$
is given by 
\[ \tilde\sigma:=\Phi^{-1}\circ \sigma\colon U\to U\times(\got g^*\otimes \got i^*\otimes \got g/\got i), \qquad \tilde\sigma(\grap(\phi))=\left(\phi, \mu_{\phi}\right)\] for all $\phi\in \got i^*\otimes \got i^c$,
with $\mu_{\phi}\in \got g^*\otimes \got i^*\otimes \got g/\got i$ given by 
\[  \mu_{\phi}(x,u)=\pr_{\got g/\got i}[x, u+\phi(u)]-(\pr_{\got g/\got i}\circ\phi\circ \pr_{\got i})[x, u+\phi(u)]\]
for all $x\in \got g$ and all $u\in \got i$.

Take a smooth curve $\phi\colon I\to U$ with $\phi(0)=0$ i.e.~$\grap(\phi(0))=\got i$. Then as discussed in Appendix \ref{tangent_grk}, \[
\dot\phi(0)\in T_{\got i}\Gr_k(\got g)=\got i^*\otimes \got g/\got i
\] and by the considerations above
\begin{equation*}
\begin{split}
T^{\ver}_{\got i}\sigma\left(\left.\frac{d}{dt}\right\arrowvert_{t=0}\phi(t)\right)=\left.\frac{d}{dt}\right\arrowvert_{t=0}\mu_{\phi(t)}.
\end{split}
\end{equation*}
On $x\in \got g$ and $u\in \got i$, 
\begin{equation*}
\begin{split}
&\left(\left.\frac{d}{dt}\right\arrowvert_{t=0}\mu_{\phi(t)}\right)(x,u)=\left.\frac{d}{dt}\right\arrowvert_{t=0}\left(\pr_{\got g/\got i}[x, u+\phi(t)(u)]-(\pr_{\got g/\got i}\circ\phi(t)\circ \pr_{\got i})[x, u+\phi(t)(u)]\right)\\
&=\pr_{\got g/\got i}\left[x, \dot\phi(0)(u)\right]-\left(\pr_{\got g/\got i}\circ\dot \phi(0)\circ \pr_{\got i}\right)\left[x, u+\phi(0)(u)\right]-(\pr_{\got g/\got i}\circ\phi(0)\circ \pr_{\got i})\left[x, u+\dot\phi(0)(u)\right]\\
&=\pr_{\got g/\got i}\left[x, \dot\phi(0)(u)\right]-\left(\pr_{\got g/\got i}\circ\dot \phi(0)\circ \pr_{\got i}\right)\underset{\in \got i}{\underbrace{\left[x, u\right]}}=\pr_{\got g/\got i}\left[x, \dot\phi(0)(u)\right]-(\pr_{\got g/\got i}\circ\dot \phi(0))\left[x, u\right]
\end{split}
\end{equation*}
since\footnote{In this computation $\dot \phi(0)$ is considered an element of $\got i^*\otimes \got i^c$.} $\phi(0)=0$. This shows that 
\[T^{\ver}_{\got i}\sigma\left(\left.\frac{d}{dt}\right\arrowvert_{t=0}\phi(t)\right)=\delta_{\got{g}\rhd{\got{i}}}^{\Hom}\left(\left.\frac{d}{dt}\right\arrowvert_{t=0}\phi(t)\right).
\]
Since the curve $\phi\colon I \to U$ was arbitrary, this shows that 
\[T^{\ver}_{\got i}\sigma=\delta_{\got{g}\rhd{\got{i}}}^{\Hom}\colon \got i^*\otimes \got g/\got i\to \got g^*\otimes \got i^*\otimes \got g/\got i.
\]

\bigskip

Let $\Aut(\got{g})$ be the Lie group of Lie algebra automorphisms of $(\got{g},\mu_{\got{g}})$ and consider the (left) linear Lie group action 
\begin{equation*}
\tilde \alpha\colon\Aut(\got{g})\times (\Gr_k(\got g)\times (\got g^*\otimes \got g^*\otimes \got g))\to (\Gr_k(\got g)\times (\got g^*\otimes \got g^*\otimes \got g))
\end{equation*}
defined by
\begin{equation*}
(\Psi, (W, \omega)\mapsto (\Psi(W), \Psi\circ\omega\circ (\Psi^{-1}, \Psi^{-1})),
\end{equation*}
where $\omega\in \got g^*\otimes \got g^*\otimes \got g$ is seen as an element of $\Hom(\got g\otimes \got g, \got g)$. The action $\tilde \alpha$ is linear over the canonical Lie group action $\Aut(\got g)\times \Gr_k(\got g)\to \Gr_k(\got g)$, $(\Psi, W)\mapsto \Psi(W)$.
By the following proposition, there is a unique Lie group action 
\begin{equation*}
\alpha\colon\Aut(\got{g})\times E\to E
\end{equation*}
such that 
\[\begin{tikzcd}
{\Aut(\got g)\times (\Gr_k(\got g)\times (\got g^*\otimes \got g^*\otimes \got g))} && {\Gr_k(\got g)\times (\got g^*\otimes \got g^*\otimes \got g)} \\
\\
\\
{\Aut(\got g)\times E} && E
\arrow["{\tilde\alpha}", from=1-1, to=1-3]
\arrow["{\id_{\Aut(\got g)}\times P}"{description}, shift right, from=1-1, to=4-1]
\arrow["P", from=1-3, to=4-3]
\arrow["\alpha"', from=4-1, to=4-3]
\end{tikzcd}\]
commutes. The Lie group action $\alpha$ is explicitly
defined by
\begin{equation*}
\bigleftpar\Psi, (W,\phi_{W})\bigrightpar\mto\bigleftpar\Psi(W),\Psi\phi_W\bigrightpar\in\{\Psi(W)\}\times\Hom\bigleftpar\got{g},(\Psi(W))^*\otimes\got{g}/{\Psi(W)}\bigrightpar,
\end{equation*}
where the second component is given on $x\in\got{g}$ and $w\in W$ by:
\begin{equation}\label{the action of Aut(g) on ideals}
\Psi\phi_W(x,\Psi(w))\eqdef\pi_{\got{g}/\Psi(W)}\circ\Psi\bigleftpar(\underbrace{\phi_W(\Psi^{-1}(x),w)}_{\in\got{g}/W})^\text{lift}\bigrightpar,
\end{equation}
where for any $y\in \got g/W$, the element $y^\text{lift}$ of $\got g$ is an arbitrary choice of element of $\got g$ such that $\pi_{\got g/W}\left(y^\text{lift}\right)=y$.

\begin{prop}\label{aut_action}
In the situation above, the pair $(E\to M, \Aut(\got{g}))$ is an $\Aut(\got{g})$-vector bundle and its section $\sigma$ defined in \eqref{def_sigma_section} is $\Aut(\got{g})$-equivariant.
\end{prop}

This follows from the following lemma, the proof of which is straightforward and left to the reader.
\begin{lem}
Let $F\to M$ be a smooth vector bundle over a smooth manifold $M$, let $F_0\subset F$ be a vector subbundle over $M$, and let $G$ be a Lie group. 
Let $P\colon F\to F/F_0$ be the canonical epimorphism of vector bundles over the identity on $M$.
Assume that 
\[\tilde \alpha\colon G\times F\to F, \qquad \alpha_0\colon G\times M\to M
\]
are smooth (left) Lie group actions, such that 
\[\begin{tikzcd}
G\times F & F \\
G\times M & M
\arrow["{\tilde\alpha}", from=1-1, to=1-2]
\arrow[from=1-1, to=2-1]
\arrow[from=1-2, to=2-2]
\arrow["{\alpha_0}"', from=2-1, to=2-2]
\end{tikzcd}\]
is a vector bundle homomorphism.
\begin{enumerate}
\item If $\tilde\alpha_g(F_0)=F_0$ for all $g\in G$, then $\tilde \alpha$ quotients to a smooth Lie group action 
\[ \alpha\colon G\times F/F_0\to F/F_0, \qquad \alpha(g,f_p+F_0(p))\mapsto \tilde\alpha(g, f_p)+F_0(gp)
\]
for all $g\in G$, $p\in M$ and $f_p\in F(p)$, i.e. such that 
\[\begin{tikzcd}
&& {G\times F/F_0} && {F/F_0} \\
{G\times F} && F \\
\\
{G\times M} && M
\arrow["\alpha"{description}, from=1-3, to=1-5]
\arrow["{\id_G\times P}"{description}, from=2-1, to=1-3]
\arrow["{\tilde\alpha}"{description, pos=0.4}, from=2-1, to=2-3]
\arrow["P"{description}, from=2-3, to=1-5]
\arrow[dashed, from=1-3, to=4-1]
\arrow[from=2-1, to=4-1]
\arrow["{\alpha_0}"{description}, from=4-1, to=4-3]
\arrow[from=1-5, to=4-3]
\arrow[from=2-3, to=4-3]
\end{tikzcd}\]
commutes.
\item Consider a smooth $G$-equivariant section $\tilde \sigma\colon M\to F$. Then $\tilde \sigma$ quotients 
to a smooth $G$-equivariant section $\sigma:=P\circ \tilde\sigma\colon M\to F/F_0$, i.e.~such that 
\[\begin{tikzcd}
&& {G\times F/F_0} && {F/F_0} \\
{G\times F} && F \\
\\
{G\times M} && M
\arrow["\alpha"{description}, from=1-3, to=1-5]
\arrow["{\id_G\times P}"{description}, from=2-1, to=1-3]
\arrow["{\tilde\alpha}"{description, pos=0.4}, from=2-1, to=2-3]
\arrow["P"{description}, from=2-3, to=1-5]
\arrow["{\id_G\times \sigma}"{description, pos=0.4}, dashed, from=4-1, to=1-3]
\arrow["{\id_G\times\tilde\sigma}"{description}, from=4-1, to=2-1]
\arrow["{\alpha_0}"{description}, from=4-1, to=4-3]
\arrow["\sigma"{description}, from=4-3, to=1-5]
\arrow["{\tilde \sigma}"{description}, from=4-3, to=2-3]
\end{tikzcd}\]
commutes.
\end{enumerate}
\end{lem}

\begin{proof}[Proof of Proposition \ref{aut_action}]
The vector bundle $E\to \Gr_k(\got g)$ is the quotient of $\Gr_k(\got g)\times (\got g^*\otimes\got g^*\otimes \got g)$ by the kernel 
\[ (\Gr_k(\got g)\times \got g^*)\otimes\Taut_k(\got g)^\circ\otimes(\Gr_k(\got g)\times \got g)\, + \, (\Gr_k(\got g)\times \got g^*)\otimes(\Gr_k(\got g)\times \got g^*)\otimes\Taut_k(\got g).
\]
For $\alpha$ to be defined, it suffices hence to show that this vector subbundle of $\Gr_k(\got g)\times (\got g^*\otimes\got g^*\otimes \got g)$ is invariant under the action $\tilde \alpha$.
Take $(W,\omega)$ in the first summand, i.e.~$W\in\Gr_k(\got g)$ and $\omega\in\got g^*\otimes \got g^*\otimes \got g$ does $\omega(x, w)=0$ for all $x\in \got g$ and all $w\in W$.
Then for all $\Psi\in \Aut(\got g)$, for all $x\in \got g$ and all $w\in W$
\[(\Psi\circ \omega)\left(\Psi^{-1}(x), \Psi^{-1}(\Psi(w))\right)=\Psi\left(\omega\left(\Psi^{-1}(x), w\right)\right)=0,
\]
which shows that 
\[ \Psi\cdot(W,\omega)=\left(\Psi(W), \Psi\circ\omega\circ (\Psi^{-1}, \Psi^{-1})\right)
\]
lies again in $(\Gr_k(\got g)\times \got g^*)\otimes\Taut_k(\got g)^\circ\otimes(\Gr_k(\got g)\times \got g)$.
Similarly, the subspace $(\Gr_k(\got g)\times \got g^*)\otimes(\Gr_k(\got g)\times \got g^*)\otimes\Taut_k(\got g)$ is clearly preserved by the $\Aut(\got g)$-action on $\Gr_k(\got g)\times (\got g^*\otimes\got g^*\otimes \got g)$.

It remains therefore to check that $\tilde\sigma\colon \Gr_k(\got g)\to\Gr_k(\got g)\times (\got g^*\otimes \got g^*\otimes \got g)$ is $\Aut(\got g)$-equivariant. But this is immediate since $\tilde \sigma$ is the constant section 
\[ \tilde \sigma(W)=(W, \mu_{\got g})
\]
for all $W\in \Gr_k(W)$, and 
$\Psi\in \Aut(\got g)$ is by definition an isomorphism of $\got g$ preserving $\mu_{\got g}$:
\[\Psi\circ \mu_{\got g}\circ \left(\Psi^{-1}, \Psi^{-1}\right)=\mu_{\got g}.\]
\end{proof}

The differential at $\Id_{\got{g}}\in\Aut(\got{g})$  of the orbit map $\alpha^{\got{i}}\colon \Aut(\got{g})\to\Gr_{k}(\got{g})$ of the action $\alpha^{\res}$ is computed as follows.
First recall that$$T_{\Id_{\got{g}}}(\Aut({\got{g}}))=\got{aut}(\got{g})\eqdef \{\phi\in\got{gl}(\got{g}) \ | \ \phi([x,y])=[\phi(x),y]+[x,\phi(y)]\}=Z^1(\got{g};\got{g}).$$Therefore, $T_{\Id_{\got{g}}}\alpha^{\got{i}}\colon \got{aut}(\got{g})\to T_{\got{i}}\Gr_{k}(\got{g})$. Let $(\phi_t)_{t\in I}\in\Aut(\got{g})$ a smooth curve of automorphisms such that $\phi_0=\Id_{\got{g}}$. Then as an element of $T_{\got i}\Gr_k(\got g)=\got i^*\otimes \got g/\got i$,
\begin{align*}
T_{\Id_{\got{g}}}\alpha^{\got{i}}\biggleftpar\fracddtz\phi_t\biggrightpar=\fracddtz\phi_t(\got{i})
=\pi_{\got g/\got i}\circ \dot\phi(0)\arrowvert_{\got i},
\end{align*}
see the computation at the end of Appendix \ref{tangent_grk}.
This shows that for all $\psi\in\got{aut}(\got{g})$, $T_{\Id_{\got{g}}}\alpha^{\got{i}}(\psi)=\pi_{\got{g}/\got{i}}\circ\psi|_{\got{i}}\in\Hom(\got{i},\got{g}/\got{i})\simeq T_{\got{i}}\Gr_{k}(\got{g})$.

\medskip
As a summary, the sequence \eqref{the sequence used for defining non-degenerace of zeros of a section} reads here 
\begin{eqnarray*}
Z^1(\got{g};\got{g})=\got{aut}(\got g) \longrightarrow &T_{\got i}\Gr_k(\got g)=\got i^*\otimes \got g/\got i&\longrightarrow E_{\got i}=\got g^*\otimes \got i^*\otimes \got g/\got i
\end{eqnarray*}
with the first map 
\[\psi \mapsto \pi_{\got g/\got i}\circ \psi\arrowvert_{\got i}
\] and the second map
\[\psi  \mapsto \delta^{\Hom}_{\got i\lhd \got g}\psi.\]

\bigskip

\begin{defin}\label{def_top_rig}
An ideal $\got{i}\lhd(\got{g},\mu_{\got{g}})$ is called a \textbf{(topologically) rigid ideal} if the space of $k$-dimensional Lie ideals $\mbf{I}_k(\got{g})$ coincides locally, in some open neighborhood of $\got{i}\in\mbf{I}_k(\got{g})\subset\Gr_k(\got{g})$, with the $\Aut(\got{g})$-orbit of $\got{i}$. Namely, $\got{i}\lhd\got{g}$ is (topologically) rigid if there exists a neighborhood $U_{\got{i}^{\perp}}$ of $\got{i}\in\Gr_{k}(\got{g})$ such that every $\got{i}'\in U_{\got{i}^{\perp}}\cap\mbf{I}_k(\got{g})$ belongs in the orbit of $\got{i}$: $\got{i}'=\Psi(\got{i})$, for some $\Psi\in\Aut(\got{g})$.
\end{defin}
\begin{rem}
Note that in the case of Lie subalgebras, two Lie subalgebras are considered equivalent
if they are related by an \emph{inner automorphism} of the ambient Lie algebra. As already mentioned, in the case of an ideal, this equivalence relation becomes trivial and consequently not interesting. The following theorem shows that the definition of rigidity above is more natural in the context of Lie ideals. The current section studies the moduli space of Lie ideals under this natural action of the automorphism group of the ambient Lie algebra on the space of $k$-dimensional Lie ideals.
\end{rem}

\begin{thm}\label{theorem for Aut(g)-rigidity of ideals}
Let $\got{i}\lhd(\got{g},\mu_{\got{g}})$ be a Lie ideal. If $$H^0(\Pi)\colon  H^1_{\pi_{\got{g}/\got{i}}}(\got{g};\got{g}/\got{i})\to H^0(\got{i}\lhd\got{g}), \ [\phi]\mto\phi|_{\got{i}}$$is surjective, then $\got{i}\lhd\got{g}$ is a (topologically) rigid ideal with respect to the $\Aut(\got{g})$-action.
\end{thm}
\begin{proof}
The vector bundle $(E\to\Gr_{k}(\got{g}),\Aut(\got{g}))$ is an $\Aut(\got{g})$-vector bundle and the section $\sigma\colon \Gr_{k}(\got{g})\to E$ is $\Aut(\got{g})$-equivariant by Proposition \ref{aut_action}. 
The only assumption that remains to be checked, in order to apply Proposition \ref{proposition of openess of orbits}, is the non-degeneracy of the zero $\got{i}$ of $\sigma$. 
The map $T_{\Id_{\got{g}}}\alpha^{\got{i}}\colon \psi\mapsto \pi_{\got{g}/\got{i}}\circ\psi|_{\got{i}}$ factors as follows through the map $(\pi_{\got{g}/\got{i}})_*:Z^1(\got{g};\got{g})\to Z^1(\got{g};\got{g}/\got{i}), \ \psi\mto\pi_{\got{g}/\got{i}}\circ\psi$,
\begin{equation*}
\begin{tikzcd}
Z^1(\got{g};\got{g})\ar[d, "(\pi_{\got{g}/\got{i}})_*", dashed, swap] \ar["T_{\Id_{\got{g}}}\alpha^{\got{i}}", r] &  C^0(\got{g};\got{i}^*\otimes\got{g}/\got{i})  \ar[r, "\delta_{\got{g}\rhd\got{i}}^{\Hom}"] &  C^1(\got{g}; \got{i}^*\otimes\got{g}/\got{i})\\
Z^1(\got{g};\got{g}/\got{i}) \ar[ur, dashed, "\Pi", swap] & & 
\end{tikzcd}
\end{equation*}
with $\Pi$ as in Proposition  \ref{the lemma where we defined the cochain map from def. complex of can projection to the def. complex of ideals}: $$\Pi\colon C^{\bullet}(\got{g};\got{g}/\got{i})[1]\to C^{\bullet}(\got{g};\got{i}^*\otimes\got{g}/\got{i}), \ f\mto f|_{\wedge^{\bullet}\got{g}\wedge\got{i}},$$ restricted at degree 0 cocycles. The above sequence is exact if $\img(\Pi)\subset C^0(\got{g};\got{i}^*\otimes\got{g}/\got{i})$ is equal to $\Ker(\delta^{\Hom}_{\got{g}\rhd{\got{i}}})=Z^0(\got{g};\got{i}^*\otimes\got{g}/\got{i})$. This is exactly the surjectivity of $H^0(\Pi)$.
\end{proof}
The last theorem and Proposition \ref{prop. that the def. cohom. of an ideal fits in a commutative diagram of l.e.s.} have together the following corollary.
\begin{cor}
If $H^2(\got{g}/\got{i};\got{g}/\got{i})=0$, then $\got{i}\lhd\got{g}$ is (topologically) rigid under the $\Aut(\got{g})$-action.
\end{cor}

\begin{ex}
This shows that the radical ideal $\operatorname{Rad}(\got g)$ of a finite-dimensional Lie algebra $\got g$ is topologically rigid with respect to the $\Aut(\got{g})$-action, since the quotient $\got g/\operatorname{Rad}(\got g)$ is semi-simple and so 
\[ H^2\left(\frac{\got{g}}{\operatorname{Rad}(\got g)};\frac{\got{g}}{\operatorname{Rad}(\got g)}\right)=0,
\]
by Whitehead's (second) lemma, see e.g.~\cite{Jacobson-Nathan-book-on-Lie-algebras-1962}.
\end{ex}

\subsection{Stability of ideals in Lie algebras}

This section is inspired from the study in \cite{Crainic-Ivan-Schaetz} of the stability of Lie algebras, Lie algebra morphisms and Lie subalgebras. Similar techniques are developed to obtain a stability result for Lie ideals.

\begin{prop}(Stability of zeros \cite{Crainic-Ivan-Schaetz})\label{proposition of stability of zeros}
Let $E\to M$ and $F\to M$ be two vector bundles over the same base, $\sigma\in\Gamma(E)$ and $\tau\in\Gamma(\Hom(E,F))$ satisfying $\tau\circ\sigma=0$. In addition, let $x\in\sigma^{-1}(0)$ and assume that the following sequence is exact:
\begin{equation*}
\begin{tikzcd}
T_xM \ar[r, "T_x^{\ver}\sigma"] & E_x \ar[r, "\tau_x"] & F_{x}.
\end{tikzcd}
\end{equation*}
Then the following statements hold true:
\begin{enumerate}
\item $\sigma^{-1}(0)$ is locally a manifold around $x$ of dimension equal to $\dim\left(\Ker(T_x^{\ver}\sigma)\right)$.
\item For each open neighborhood $U$ of $x$ in $M$ there exist $C^0$-open subsets $V\subseteq \Gamma(E)$ and $W\subseteq \Gamma(\Hom(E,F))$ around $\sigma$ and $\tau$, respectively, such that 
for all  $\sigma'\in V$ and $\tau' \in W$ with $\tau'\circ\sigma'=0$, there exists $x'\in U$ such that $\sigma'(x')=0$.
\item In the situation of (2) there exists as well a $C^1$-open subset $V^1\subseteq V$ around $\sigma$ such that if $\sigma'\in V^1$, then the zero set of $\sigma'$ is also locally a manifold around $x'$, of the same dimension as $\sigma^{-1}(0)$.
\end{enumerate}
\end{prop}
In the following, $E\to \Gr_k(\got g)$ is the vector bundle already defined in \eqref{def_E} for the rigidity of a Lie ideal $\got{i}\lhd\got{g}$, and $\sigma$ is its section $\sigma\colon \Gr_k(\got{g})\to E, \ W\mto\sigma_W\in E\arrowvert_{W}$ with $\sigma_W(x,w)=\pi_{\got{g}/W}([x,w])$ for all $x\in \got g$ and all $w\in W$. 
The second vector bundle $\pi_F\colon F\to M$ is defined to be the vector bundle over $\Gr_k(\got{g})$ with fiber over $W\in\Gr_k(\got{g})$ defined by $F\arrowvert_W\eqdef\{W\}\times\Hom(\wedge^2\got{g};W^*\otimes\got{g}/W)$, i.e.~$F$ is the smooth vector bundle
\[ F=(\Gr_k(\got g)\times \wedge^2\got g^*)\otimes (\operatorname{Taut}_k(\got g))^*\otimes \frac{\Gr_k(\got g)\times\got g}{\operatorname{Taut}_k(\got g)} \longrightarrow \Gr_k(\got g).
\]
Like $E$, the vector bundle $F$ is a quotient vector bundle 
\[\begin{tikzcd}
{\Gr_k(\got g)\times (\wedge^2\got g^*\otimes\got g^*\otimes \got g)} &&& F\\
&& {\Gr_k(\got g)}
\arrow["P_F", from=1-1, to=1-4]
\arrow["{\pr_1}"', from=1-1, to=2-3]
\arrow["\pi_F", from=1-4, to=2-3]
\end{tikzcd}\]
with the smooth quotient map $P_F$ sending $(W,\eta)$ to 
\[ (W, \pi_{\got g/W}\circ \eta\arrowvert_{\wedge^2\got g\otimes W}).
\]

Consider the vector bundle morphism \[
\tilde \tau\colon \Gr_k(\got g)\times(\got g^*\otimes \got{gl}(\got g))\to\Gr_k(\got g)\times(\wedge^2\got g^*\otimes \got{gl}(\got g)) \]
over the identity $\id_{\Gr_k(\got g)}$, which is defined by 
\[ \tilde \tau(W, \phi)\mapsto (W, \delta_{\ad^\got g}\phi)
\]
for all $W\in \Gr_k(\got g)$ and all $\phi\in \got g^*\otimes \got{gl}(\got g)$.
Recall here that $\delta_{\ad^\got g}\colon \wedge^\bullet\got g^*\otimes\got{gl}(\got g)\to \wedge^{\bullet+1}\got g^*\otimes\got{gl}(\got g)$
is the differential defined in Section \ref{def_subrep_adjoint_subsec} by the representation $r\eqdef(\ad^{\got{g}})^*\left(\ad^{\got{gl}(\got{g})}\right)$ of $\got g$ on $\got{gl}(\got g)$:
\[ r\colon \got g\times \got{gl}(\got g)\to \got{gl}(\got g), \qquad (r_x(\phi))(y)=[x,\phi(y)]-\phi[x,y]
\]
for all $x, y\in\got g$ and all $\phi\in \got{gl}(\got g)=\got g^*\otimes \got g$.

The smooth vector bundle morphism
$$\tau:=P_F\circ \tilde \tau\circ S\colon E\to F,$$
with $S\colon E\to \Gr_k(\got g)\times (\got g^*\otimes\got g^*\otimes \got g)$ defined as in Appendix \ref{triv_E} by a choice of inner product on $\got g$, 
then sends $(W,\eta)\in E\arrowvert_{W}=\{W\}\times (\got g^*\otimes W^*\otimes \got g/W)$ to 
\[(W, \pi_{\got g/W}\circ (\delta_{\ad^\got g}(S_W(\eta))\arrowvert_{\wedge^2\got g\otimes W})) \in \{W\}\otimes (\wedge^2\got g^*\otimes W^*\otimes \got g/W)=F\arrowvert_{W}.\]
The form $\pi_{\got g/W}\circ (\delta_{\ad^\got g}(S_W(\eta))\arrowvert_{\wedge^2\got g\otimes W})=:\tau_W(\eta)\in \wedge^2\got g^*\otimes W^*\otimes \got g/W$ is explicitly
defined by 
\begin{align*}
\tau_W(\eta)(x,y,w)=&\pi_{\got{g}/W}([x, s_{W}(\eta(y,w))])-\eta(y,\pr_W[x,w])\\
&-\pi_{\got{g}/W}([y,s_{W}(\eta(x,w))])+\eta(x,\pr_W[y,w])-\eta([x,y],w)
\end{align*}
for $x,y\in\got{g}$ and $w\in W$.
Write again $\sigma_W$ for the second component $\pi_{\got g/W}\circ \mu_{\got g}\arrowvert_{\got g\otimes W}$ of the image under $\sigma \colon \Gr_k(\got g)\to E$ of $W\in\Gr_k(\got g)$. Then, using the Jacobi identity for $\mu_{\got g}$
\begin{equation}\label{tau_sigma}
\begin{split}
\tau_W(\sigma_W)(x,y,w)=&\pi_{\got{g}/W}[x,s_{W}\circ\pi_{\got{g}/W}[y,w]]-\pi_{\got{g}/W}[y,\pr_{W}[x,w]]\\
&-\pi_{\got{g}/W}[y,s_{W}\circ\pi_{\got{g}/W}[x,w]]+\pi_{\got{g}/W}[x,\pr_{W}[y,w]]-\pi_{\got{g}/W}[[x,y],w]\\
=&\pi_{\got{g}/W}[x,\pr_{\operatorname{orth}(W)}[y,w]]-\pi_{\got{g}/W}[y,\pr_{W}[x,w]]\\
&-\pi_{\got{g}/W}[y,\pr_{\operatorname{orth}(W)}[x,w]]+\pi_{\got{g}/W}[x,\pr_{W}[y,w]]-\pi_{\got{g}/W}[[x,y],w]\\
=&\pi_{\got{g}/W}([x,[y,w]]-[y,[x,w]]-[[x,y],w])=0
\end{split}
\end{equation}
for all $x,y\in \got g$ and all $w\in W$.
This shows that $\tau_W\circ\sigma_W=0$ for all $W\in\Gr_{k}(\got{g})$ and so that 
\[ \tau\circ \sigma=0.
\]

\begin{rem}\label{tau_to_sec_general}
Each element $\eta\in\wedge^2\got g^*\otimes \got g$ defines a linear map
\[ \delta_{\eta}\colon \got g^*\otimes \got g^*\otimes \got g\to \wedge^2\got g^*\otimes \got g^*\otimes \got g
\]
by 
\[ \delta_\eta(\nu)(x_1,x_2, x)= \eta(x_1, \nu(x_2,x))-\nu(x_2, \eta(x_1,x))-\eta(x_2, \nu(x_1,x))+\nu(x_1, \eta(x_2,x))-\nu(\eta(x_1,x_2),x)
\]
for all $\nu\in \got g^*\otimes \got g^*\otimes \got g$ and $x_1,x_2,x\in \got g$, and the map
\[ \wedge^2\got g^*\otimes \got g\to \Hom\left(\got g^*\otimes \got g^*\otimes \got g, \wedge^2\got g^*\otimes \got g^*\otimes \got g\right), \qquad \eta\to \delta_\eta
\]
is again linear. Hence it defines a linear map 
\[  \wedge^2\got g^*\otimes \got g\to \Gamma\left(\Hom\left(\Gr_k(\got g)\times(\got g^*\otimes \got g^*\otimes \got g), \Gr_k(\got g)\times(\wedge^2\got g^*\otimes \got g^*\otimes \got g)\right)\right), \qquad \eta\to \widetilde{\delta_\eta},
\]
where for $\psi\in  \Hom\left(\got g^*\otimes \got g^*\otimes \got g, \wedge^2\got g^*\otimes \got g^*\otimes \got g\right)$, the section \[\widetilde\psi\in \Gamma\left(\Hom\left(\Gr_k(\got g)\times(\got g^*\otimes \got g^*\otimes \got g), \Gr_k(\got g)\times(\wedge^2\got g^*\otimes \got g^*\otimes \got g)\right)\right)\] is the constant section defined by $\psi$.
The map  
\[\Tau\colon \wedge^2\got g^*\otimes \got g\to \Gamma\left(\Hom(E,F)\right), \qquad  \eta\mapsto P_F\circ \widetilde{\delta_\eta}\circ S
\]
is then also $\mathbb R$-linear since $P_F\colon \Gr_k(\got g)\times(\wedge^2\got g^*\otimes \got g^*\otimes \got g)\to F$ and $S\colon E\to \Gr_k(\got g)\times(\got g^*\otimes \got g^*\otimes \got g)$ are vector bundle morphisms over the identity on $\Gr_k(\got g)$.
\end{rem}

\begin{lem}\label{lemma for continuity of the product map used for stability of ideals}
Let $\Lambda\subset\Hom(\wedge^2\got{g};\got{g})$ be the space of Lie brackets on the vector space $\got{g}$. 
Consider the map
\begin{equation}\label{the product map is used for the stability of Lie ideals}
\Theta\colon	\Lambda\to\Gamma(E)\times\Gamma(\Hom(E,F)), \quad \mu\mto\bigleftpar\Sigma(\mu),\Tau(\mu)\bigrightpar,
\end{equation}
with $\Sigma\colon \wedge^2\got g^*\otimes \got g\to \Gamma(E)$ the linear map defined in Remark \ref{sigma_to_sec_general} and $\Tau\colon\wedge^2\got g^*\otimes\got g\to \Gamma(F)$ the linear map defined in 
Remark \ref{tau_to_sec_general}. 
Then $\Theta$ is continuous with respect to the product compact-open $C^k$-topology on $\Gamma(E)\times\Gamma(\Hom(E,F))$, for any $k\geq 0$.
\end{lem}
\begin{proof}
The maps $\Sigma\colon \wedge^2\got g^*\otimes \got g\to\Gamma(E)$ and $\phi\colon\wedge^2\got g^*\otimes \got g\to\Gamma(\Hom(E,F))$ are continuous as linear maps from finite dimensional Hausdorff topological vector spaces to topological vector space, see e.g.~\cite[Lemma 1.20]{Rudin-Functional-Analysis-book-2nd-edition-91'}.
The natural inclusion $i\colon \Lambda\to\wedge^2\got{g}^*\otimes\got{g}$ is continuous, since $\Lambda$ is equipped with the subspace topology, and so $\Sigma\arrowvert_{\Lambda}$ and $\Tau\arrowvert_{\Lambda}$ are continuous by the following commutative diagrams.
\[\begin{tikzcd}
\Lambda && {\Gamma(E)} && \Lambda && {\Gamma(\Hom(E,F))} \\
& {\wedge^2\got g^*\otimes \got g} &&&& {\wedge^2\got g^*\otimes \got g}
\arrow["\Sigma\arrowvert_{\Lambda}", from=1-1, to=1-3]
\arrow["i"{description}, hook, from=1-1, to=2-2]
\arrow["\Tau\arrowvert_{\Lambda}", from=1-5, to=1-7]
\arrow["i"{description}, hook, from=1-5, to=2-6]
\arrow["\Sigma"{description}, from=2-2, to=1-3]
\arrow["\Tau"{description}, from=2-6, to=1-7]
\end{tikzcd}\]
\end{proof}
Roughly speaking, an ideal $\got{i}\lhd(\got{g},\mu_{\got{g}})$ is (topologically) stable if for any Lie bracket $\mu'$ \emph{close enough} to $\mu_{\got{g}}$, there exists a Lie ideal $\got{i}'\lhd(\got{g},\mu')$ \emph{close enough} to $\got{i}\in\Gr_{k}(\got{g})$. 
The precise definition is the following.
\begin{defin}
An ideal $\got{i}\lhd(\got{g},\mu_{\got{g}})$ in a Lie algebra $(\got g, \mu_{\got g})$ is called a \textbf{(topologically) stable ideal} if for each neighborhood $U\subset\Gr_{k}(\got{g})$ of $\got{i}$ in $\Gr_k(\got g)$ there exists a neighborhood $V\subset\Lambda$ of $\mu_{\got{g}}$ in $\Lambda$ such that for every $\mu'\in V$ there exists $\got{i}'\in U$ with $\got{i}'\lhd(\got{g},\mu')$.
\end{defin}
\begin{thm}\label{theorem stability ideals}
Let $\got{i}\lhd(\got{g},\mu_{\got{g}})$ be an ideal. If $H^1(\got{i}\lhd\got{g})=0$, then $\got{i}$ is a (topologically) stable ideal. Furthermore, in this case, the space of $k$-dimensional Lie ideals $\mbf{I}_k(\got{g})\subset\Gr_k(\got{g})$ is locally around $\got{i}$, a manifold of dimension equal to $\dim\bigleftpar Z^0(\got{i}\lhd\got{g})\bigrightpar$. In particular, $Z^0(\got{i}\lhd\got{g})$ is the tangent space of $\mbf{I}_k(\got g)$ at $\got i$ and each infinitesimal deformation is a deformation class.
\end{thm}
\begin{proof}
Consider the map $\Theta$ defined in \eqref{the product map is used for the stability of Lie ideals}
\begin{equation*}
\Theta\colon \Lambda\to\Gamma(E)\times\Gamma(\Hom(E,F)), \quad \mu\mto\bigleftpar\Sigma(\mu),\Tau(\mu)\bigrightpar.
\end{equation*}
The same computation as \eqref{tau_sigma} but with $\mu\in\Lambda$ replacing  $[\cdot,\,\cdot]=\mu_{\got g}$ shows that the identity $\Tau(\mu)\circ\Sigma(\mu)=0$ is satisfied for all Lie brackets $\mu\in \Lambda$ on $\got g$.

It is easy to see using Section \ref{def_subrep_adjoint_subsec} that $\tau=\Tau(\mu_{\got g})$ restricted to \[E\arrowvert_{\got i}=\got g^*\otimes \got i^*\otimes \got g/\got i\longrightarrow  F\arrowvert_{\got i}=\wedge^2\got g^*\otimes \got i^*\otimes \got g/\got i\]
is the linear map $\delta_{\got{g}\rhd\got{i}}^{\Hom}$. 
Hence the sequence of Proposition \ref{proposition of stability of zeros} at the point $\got{i}\in\mbf{I}_k(\got{g})$ is given by
\begin{equation*}
\begin{tikzcd}
C^0(\got{g};\got{i}^*\otimes\got{g}/\got{i}) \ar[r, "\delta_{\got{g}\rhd\got{i}}^{\Hom}"] &  C^1(\got{g};\got{i}^*\otimes\got{g}/\got{i}) \ar[r, "\delta_{\got{g}\rhd\got{i}}^{\Hom}"] &  C^2(\got{g};\got{i}^*\otimes\got{g}/\got{i})
\end{tikzcd}
\end{equation*}
and  it is exact if and only if $H^1(\got{i}\lhd\got{g})=0$. Hence Proposition \ref{proposition of stability of zeros} can be applied here. 

Choose an open subset $U\subseteq \Gr_k(\got g)$ around $\got i\in \Gr_k(\got g)$. Then there exist $V\subseteq \Gamma(E)$ and $W\subseteq\Gamma(\Hom(E,F))$ $C^0$-open around $\sigma$ and $\tau$ as in (2) of Proposition \ref{proposition of stability of zeros}.
The continuity of the map \eqref{the product map is used for the stability of Lie ideals} guarantees that the pre-image $\Theta^{-1}(V\times W)$ of $V\times W$ under $\Theta$ in the $C^0$ compact-open topology is an open neighborhood $O_{\mu_{\got{g}}}\subset\Lambda$ around $\mu_{\got{g}}$. Let $\mu'\in O_{\mu_{\got{g}}}$ and consider $\sigma'\eqdef\Sigma(\mu')$ and $\tau'\eqdef\Tau(\mu')$. Recall that then $\sigma'\circ\tau'=0$, due to the Jacobi identity. Hence by (2) of Proposition \ref{proposition of stability of zeros}  there exists $\got{i}'\in U$ such that $\sigma'(\got{i}')=0$. \end{proof}

\begin{rem}\label{theorem stability ideals rem}
Note that this result can also be deduced from \cite[Theorem 3.20]{Singh25}: the differential graded algebra $(C^\bullet(\got g; \got g)[1], \llbracket\cdot\,,\cdot\rrbracket, 0)$ of Section \ref{sec_LA_MC} has the differential graded Lie subalgebra $C_{\got i}^\bullet(\got g;\got g)[1]$ defined in  Subsection \ref{rel_nij_ric}. A choice of complement $\got i^c$ for $\got i $ in $\got g$ defines splittings $\sigma_0\colon C^1(\got g; \got g)/C^1_{\got i}(\got g; \got g)\simeq \operatorname{Hom}(\got i, \got g/\got i)\to C^1(\got g; \got g)$ in degree $0$ and 
$\sigma_1\colon C^2(\got g; \got g)/C^2_{\got i}(\got g; \got g)\simeq C^1_\wedge(\got g; \got i^*\otimes \got g/\got i)\to C^2(\got g; \got g)$ in degree $1$. The Lie bracket $\mu_{\got g}$  is a Maurer-Cartan element of $(C^\bullet(\got g; \got g)[1], \llbracket\cdot\,,\cdot\rrbracket, 0)$ that lies in $C^2_{\got i}(\got g; \got g)$ since $\got i$ is an ideal in $(\got g, \mu_{\got g})$. By Proposition \ref{proposition/definition for/of the complex related to deform of ideals in Nij-Rich}, the cokernel complex$$\left(\frac{ C^{\bullet}(\got{g};\got{g})[1]}{ C^{\bullet}_{\got{i}}(\got{g};\got{g})[1]}, -\overline{\delta_{\rm ad}}=-\overline{\llbracket \mu_{\got g}, \cdot\rrbracket}\right)$$of the inclusion $ C^{\bullet}_{\got{i}}(\got{g};\got{g})[1]\hookrightarrow C^{\bullet}(\got{g};\got{g})[1]$ is isomorphic to $\left(C_{\wedge}^{\bullet}(\got{g};\got{i}^*\otimes\got{g}/\got{i}), \delta^{\rm Hom}_{\got g\rhd \got i}\right)$. Hence, according to Theorem 3.20 in \cite{Singh25}, if 
$H^1_\wedge(\got g, \got i^*\otimes\got g/\got i)=0$, 
then for every open neighborhood $V$ of $0\in  C^1(\got g; \got g)/C^1_{\got i}(\got g; \got g)\simeq \operatorname{Hom}(\got i, \got g/\got i)$ (hence for every open neighborhood of $\got i$ in $\Gr_k(\got g)$), there exists an open neighborhood $U$ of $\mu_{\got g}$ in $\Lambda$ 
such that for any $\mu 
\in U$
there exists a family $I\subseteq V$  diffeomorphic to an open neighborhood of $0\in Z^0_{\wedge}(\got g, \got i^*\otimes\got g/\got i)=Z^0(\got{i}\lhd\got{g})$, with 
$\mu^{\sigma_0(\phi)}\in C_{\got i}^2(\got g; \got g)$ for all $\phi\in I$. 
Using the splitting $\got g=\got i\oplus \got i^c$ of $\got g$, the bracket $\mu^{\sigma_0(\phi)}$ is explicitly given by 
\[ \mu^{\sigma_0(\phi)}(x,y)=\begin{pmatrix} 1& 0\\ \phi& 1\end{pmatrix}\mu\left(\begin{pmatrix} 1& 0\\ -\phi& 1\end{pmatrix}x, \begin{pmatrix} 1& 0\\ -\phi& 1\end{pmatrix}y\right)
\]
for all $x,y\in\got g=\got i\oplus \got i^c$. Since $\mu^{\sigma_0(\phi)}$ lies in $C^2_{\got i}(\got g;\got g)$, it has $\got i$ as an ideal. But this is equivalent to $\operatorname{graph}(-\phi)$ being an ideal of $\mu$.

Hence \cite[Theorem 3.20]{Singh25} implies that Theorem \ref{theorem stability ideals} holds with the weaker condition $H^1_ \wedge(\got g, \got{i}^*\otimes\got{g}/\got i)=0$.
\end{rem}

\begin{cor}
If $(\got{g},\mu_{\got{g}})$ is a semisimple Lie algebra, then every Lie ideal $\got{i}\lhd\got{g}$ is a (topologically) stable ideal.
\end{cor}
\begin{proof}
By Whitehead's (first) lemma, see e.g.~\cite{Jacobson-Nathan-book-on-Lie-algebras-1962}, follows that $H^1(\got{i}\lhd\got{g})=0$ and so the claim follows by Theorem \ref{theorem stability ideals}.
\end{proof}

\appendix
\section{Useful background on Grassmannians}\label{appendix_G_k}
This appendix collects useful structural results on Grassmann manifolds, their tangent spaces and their tautological bundles.
\subsection{Tangent spaces to the Grassmannian} \label{tangent_grk}
Let $V$ be a finite-dimensional vector space and choose $k\in \{0,\ldots, \dim V\}$. The $k$-Grassmannian of $V$ is the space 
\[ \Gr_k(V)=\{ W\subseteq V\mid W \text{ vector subspace of dimension } k\}.
\]
For each $W\in \Gr_k(V)$ and each choice of linear complement $W^c\subseteq V$ of $W$ in $V$, the 
map
\[ \Psi_{W,W^c}\colon W^*\otimes W^c\to \{W'\in \Gr_k(V)\mid W'\oplus W^c=V\}=:U_{W,W^c}, \qquad \phi\mapsto \grap(\phi)
\]
is a bijection with inverse 
\[ \Psi_{W,W^c}^{-1}\colon U_{W,W^c}\to W^*\otimes W^c, \qquad W'\mapsto \phi=\pr_{W^c}\arrowvert_{W'}\circ(\pr_{W}\arrowvert_{W'})^{-1},
\]
where $\pr_{W^c}\colon V\to W^c$ and $\pr_{W}\colon V\to W$ are the linear projections defined by the splitting $V=W\oplus W^c$ of $V$.

The set $\Gr_k(V)$ has a unique topology and a unique smooth structure such that for each $W\in \Gr_k(V)$ and each choice of complement $W^c$ as above, the map $\Psi^{W, W^c}$ is a smooth chart of $\Gr_k(V)$ centered at $W$. The smooth manifold $\Gr_k(V)$ has consequently the dimension $k(n-k)$. Choose again $W\in \Gr_k(V)$ and a smooth curve $\gamma\colon I\to \Gr_k(V)$ with $I$ an interval containing $0$, and with $\gamma(0)=W$. Choose a complement $W^c$ for $W$ in $V$. Then, possibly after shrinking the interval $I$ around $0$, the curve $\gamma$ has image in $U_{W,W^c}$. It is hence identified via $\Psi_{W,W^c}$ with a smooth curve $\phi\colon I\to W^*\otimes W^c$, and its tangent vector at $t=0$ is hence 
\[ \dot{\phi}(0)\in W^*\otimes W^c.
\]
Since $W^c$ is canonically isomorphic to $V/W$, this shows that 
\[ T_W\Gr_k(V)\simeq W^*\otimes V/W
\]
via the choice of chart centered at $W$. The following shows that this does not depend on the choice of complement $W^c$ for $W$. Consider two linear complements $W_1$ and $W_2$ for $W$ in $V$. Then for each $w_1\in W_1$ there exist $w\in W$ and $w_2\in W_2$ such that $w_1=w+w_2$. Setting $A(w_1)=w$ and $B(w_1)=w_2$ defines two linear maps $A=\pr^{W,W_2}_{W}\circ\iota_{W_1}\colon W_1\to W$ and $B=\pr^{W,W_2}_{W_2}\circ\iota_{W_1}\colon W_1\to W_2$. Here, $\iota_{W_1}$ is the inclusion of $W_1$ in $V$, and $\pr^{W,W_2}_{W}$ and $\pr^{W,W_2}_{W_2}$ are the linear projections from $V$ on $W$ and $W_2$ defined by the splitting $V=W\oplus W_2$. The map $B$ is invertible with inverse $\pr^{W,W_1}_{W_1}\circ\iota_{W_2}\colon W_2\to W_1$. Take $\phi\in W^*\otimes W_1$. Then for all $w\in W$
\begin{equation*}\label{grassmann1}
w+\phi(w)=w+A\phi(w)+B\phi(w)=\underset{\in W}{\underbrace{(\id_{W}+A\phi)(w)}}+\underset{\in W_2}{\underbrace{B\phi(w)}}.
\end{equation*}
If $\grap(\phi)\cap W_2=\{0\}$, then the map 
\[ \pr^{W, W_2}_W\arrowvert_{\grap\phi}\colon \grap(\phi)\to W, \qquad w+\phi(w)\to (\id_W+A\phi)(w)
\]
is invertible, and so the map 
\[ \id_W+A\phi=\pr^{W, W_2}_W\arrowvert_{\grap\phi}\circ I_\phi\colon W\to W
\]
is invertible as well, 
if $I_\phi\colon W\to \grap(\phi)$ is the isomorphism $w\mapsto w+\phi(w)$.

It is then easy to see that 
\[ \grap(\phi)=\grap\left(B\phi(\id_W+A\phi)^{-1}\right),
\]
with the linear map $B\phi(\id_W+A\phi)^{-1}\in W^*\otimes W_2$.

Now take a smooth curve $\gamma\colon I\to \Gr_k(V)$ with $\gamma(0)=W$. There exists $\epsilon>0$ such that 
$\gamma(-\epsilon,\epsilon)\subseteq \{W'\in \Gr_k(W)\mid W'\oplus W_1=V=W'\oplus W_2\}$. Without loss of generality $I=(-\epsilon, \epsilon)$.
Set $\phi_1\colon I\to W^*\otimes W_1$, 
\[ \phi_1=\Psi_{W,W_1}^{-1}\circ \gamma.
\] 
Then by the considerations above the image $\phi_2$ of $\gamma$ in the chart of $\Gr_k(V)$ defined by $W$ and $W_2$ is given by 
\[\phi_2(t)=\Psi_{W,W_2}^{-1}\circ \gamma(t) = B\cdot\phi_1(t)\cdot (I_W+A\phi_1(t))^{-1}
\]
for all $t\in I$, and 
\begin{equation*}
\begin{split}
\dot\phi_2(0)&= \left.\frac{d}{dt}\right\arrowvert_{t=0}\phi_2(t)=\left.\frac{d}{dt}\right\arrowvert_{t=0}B\cdot\phi_1(t)\cdot (I_W+A\phi_1(t))^{-1}\\
&=B\cdot \dot\phi_1(0)\cdot (I_W+A\phi_1(0))^{-1}+B\cdot \phi_1(0)\cdot \left.\frac{d}{dt}\right\arrowvert_{t=0} (I_W+A\phi_1(t))^{-1}\\
&=B\cdot \dot\phi_1(0)
\end{split}
\end{equation*}
since $\phi_1(0)=0$.

By definition of $B$, the canonical projection $\pi_{V/W}\colon V\to V/W$ does 
\[ \pi_{V/W}(\psi(w))=\pi_{V/W}(A\psi(w)+B\psi(w))=\pi_{V/W}(B\psi(w))
\]
for all $\psi\in W^*\otimes W_1$ and $w\in W$, i.e.
\[ \pi_{V/W}\circ \psi = \pi_{V/W}\circ B\circ \psi.
\]
In particular, 
\[ \pi_{V/W}\circ \dot \phi_2(0) = \pi_{V/W}\circ B\circ \dot \phi_1(0)=\pi_{V/W}\circ \dot \phi_1(0)
\]
as elements of $W^*\otimes V/W$.

\bigskip
Now consider more generally a curve $\alpha\colon I\to \operatorname{GL}(V)$ of isomorphisms of $V$ such that $\alpha(0)=\id_V$, and take $W\in \Gr_k(V)$.
Then $\gamma\colon I\to \Gr_k(V)$, $t\mapsto \alpha(t)(W)$ is a smooth curve starting at $W$.
Choose as before a complement $W^c$ to $W$ in $V$.
Then the curves 
\[ A:=\pr_{W}\circ\alpha\arrowvert_{W}\colon I\to W^*\otimes W \quad \text{ and } \quad B:= \pr_{W^c}\circ\alpha\arrowvert_{W}\colon I\to W^*\otimes W^c\]
are smooth and $A$ is invertible on an open neighborhood of $0$ in $I$, with $A(0)=\Id_{W}$ and $B(0)=0$. Without loss of generality $A$ is invertible on $I$.
Then for all $t\in I$ 
\[ \gamma(t)=\alpha_t(W)=\{A(t)(w)+B(t)(w)\mid w\in W\}=\{w+B(t)(A(t))^{-1}(w)\mid w\in W\}=\grap(B(t)(A(t))^{-1}).
\]
As an element of $T_W\Gr_k(W)\simeq W^*\otimes V/W$, the vector $\dot \gamma(0)$ is hence 
\begin{equation*}
\begin{split}
\dot \gamma(0)&=\pi_{V/W}\circ \left.\frac{d}{dt}\right\arrowvert_{t=0}(B(t)(A(t))^{-1})=\pi_{V/W}\circ \left(\dot B(0)(A(0))^{-1}-B(0)\left.\frac{d}{dt}\right\arrowvert_{t=0}(A(t))^{-1}\right)\\
&=\pi_{V/W}\circ \dot B(0)=
\pi_{V/W}\circ \dot\alpha(0)\arrowvert_{W}.
\end{split}
\end{equation*}

\subsection{The tautological bundle and several vector bundles constructed from it}\label{triv_E}
The tautological bundle $\pi\colon \Taut_k(V)\to\Gr_k(V)$ over $\Gr_k(V)$ is defined fibrewise by
\begin{equation*}
\Taut_k(V)\arrowvert_{W}\eqdef\{W\}\times W\subset \Gr_{k}(V)\times V,
\end{equation*}
and it is a smooth vector subbundle of the product vector bundle $\Gr_{k}(V)\times V$. That is, the vector space structure on $E_W$ is given by: $(W,w)+(W,w')=(W,w+w')$, while the projection $\pi\colon \Taut_k(V)\to\Gr_{k}(V)$ is given by $\pi(W,w)=W$. Let $U:=U_{W,W^{c}}\simeq W^*\otimes W^c$ be a neighborhood around $W\in\Gr_{k}(V)$ defined as above by a choice of complement $W^c$ for $W$ in $V$.
Then
\begin{equation*}
\pi^{-1}(U)\to U\times W, \quad (\grap(\phi),v) \mto (\phi,\pr_W(v))
\end{equation*}
for all $\phi\in W^*\otimes W^c$ and all $v\in \grap(\phi)$,
is a smooth trivialisation of $\Taut_k(V)$ around $W$, with smooth inverse
\[U\times W\to \pi^{-1}(U),  \quad (\phi, w)\mapsto (\grap(\phi), w+\phi(w)).
\]

The smooth annihilator \[\Taut_k^\circ(V):=\{(W, l)\in \Gr_k(V)\times V^*\mid l(w)=0 \text{ for all } w\in W\}\]
of $\Taut_k(V)$ as a subbundle of $\Gr_k(V)\times V$  is denoted by $\pi_{\Taut^\circ}\colon \Taut_ k^\circ(V)\to \Gr_k(V)$. It is a smooth subbundle of $\Gr_k(V)\times V^*$ because a smooth chart
$U$ of $\Gr_k(V)$ as above trivialises $\Taut_ k^\circ(V)$ via the map 
\[\pi_{\Taut^\circ}^{-1}(U)\to U\times W^\circ, \quad (\grap(\phi),l) \mto (\phi,l\arrowvert_{W^c})
\]
with smooth inverse 
\[U\times W^\circ\to \pi_{\Taut^\circ}^{-1}(U), \quad (\phi,l) \mto (\grap(\phi),l-\phi^*l).
\]
Here, the canonical identifications $W^\circ =(W^c)^*$ and $(W^c)^\circ=W^*$ are used.

The dual vector bundle $\pi_{\Taut^*}\colon \Taut_k^*(V)\to \Gr_k(V)$ is then as usual canonically isomorphic to the quotient 
\[ \Taut_k^*(V)\simeq \frac{\Gr_k(V)\times V^*}{\Taut_k^\circ (V)}\to \Gr_k(V).
\]
Again, the smooth chart $U$ of $\Gr_k(V)$ around $W$ trivialises $\Taut_k^*(V)$ via 
\[ \pi_{\Taut^*}^{-1}(U)\simeq \frac{U\times V^*}{\pi_{\Taut^\circ}^{-1}(U)}\longrightarrow U\times W^*
\simeq U\times (W^c)^\circ,
\]
\[ (\grap(\phi),  l+\grap(\phi)^\circ) \mapsto (\phi, l\arrowvert_{W}+l\arrowvert_{W^c}\circ \phi)
\]
with the smooth inverse 
\[(\grap(\phi),  l)\mapsto (\grap(\phi), l+\grap(\phi)^\circ).
\]

Finally, the quotient vector bundle \[
\overline{\pi}\colon \frac{\Gr_k(V)\times V}{\Taut_k(V)}\to \Gr_k(V)
\]
is locally trivialised by the smooth map
\[ \overline\pi^{-1}(U)\to U\times W^c, \quad (\grap(\phi), x+\grap(\phi))\mapsto (\phi, \pr_{W^c}(x)-\phi(\pr_W(x)))
\]
with the smooth inverse 
\[  U\times W^c\to \overline\pi^{-1}(U), \quad (\phi, x)\mapsto (\grap(\phi), x+\grap(\phi)).
\]

The vector bundle 
\[E:=(\Gr_k(V)\times V^*)\otimes (\operatorname{Taut}_k(V))^*\otimes \frac{\Gr_k(V)\times V}{\operatorname{Taut}_k(V)} \longrightarrow \Gr_k(V)
\]
comes with the smooth vector bundle projection 
\[\begin{tikzcd}
{\Gr_k(V)\times (V^*\otimes V^*\otimes V)} &&& E \\
&& {\Gr_k(V)}
\arrow["P", from=1-1, to=1-4]
\arrow["{\pr_1}"', from=1-1, to=2-3]
\arrow["\pi_E", from=1-4, to=2-3]
\end{tikzcd}\]
sending $(W,\phi)$ to 
\[ (W, \pi_{V/W}\circ \phi\arrowvert_{V\otimes W}).
\]
$E$ is trivialised over the open subset $U=U_{W,W^c}$ of $\Gr_k(V)$ by 
\[ \Phi\colon U\times (V^*\otimes W^*\otimes V/W) \to E\arrowvert_{U},
\]
\[(\grap(\phi), \theta\otimes (\eta\arrowvert_W+\eta\arrowvert_{W^c}\circ \phi)\otimes (\pr_{W^c}(x)-\phi(\pr_W(x))))\,\mapsto \, 
(\phi, \theta\otimes (\eta+\grap(\phi)^\circ)\otimes (x+\grap(\phi))).
\]
Conversely, the inverse of this smooth vector bundle isomorphism sends $(\grap(\phi),\omega)$ with $\omega\in V^*\otimes \grap(\phi)^*\otimes V/\grap(\phi)$ to $(\phi,\tilde\omega)$ with $\tilde\omega\in V^*\otimes W^*\otimes V/W$ defined by 
\begin{equation}\label{triv_E_2}
\tilde \omega(v,w)=\pr_{W^c}(\omega(v, w+\phi(w)))-\phi(\pr_W(\omega(v, w+\phi(w))))
\end{equation}
for all $v\in V$ and $w\in W$.
Note that the isomorphism $\Phi\colon U\times (V^*\otimes W^*\otimes V/W) \to E\arrowvert_{U}$ is the identity on $E\arrowvert_{W}=\{W\}\times (V^*\otimes W^*\otimes V/W)$. 

\medskip

A choice of inner product $\langle\cdot\,,\cdot\rangle$ on $V$ defines a map
$\operatorname{orth}\colon \Gr_k(V)\to \Gr_{n-k}(V)$ (where $n=\dim V$) sending a $k$-dimensional subspace of $V$ to its orthogonal complement.
For $\phi\in W^*\otimes\operatorname{orth}(W)$, the adjoint map $\phi^t\in (\operatorname{orth}(W))^*\otimes W$ with respect to $\langle\cdot\,,\cdot\rangle$ is defined as usual by 
\[ \langle \phi(w), u\rangle=\langle w,\phi^t(u)\rangle
\]
for all $w\in W$ and $u\in \operatorname{orth}(W)$. It is then easy to see that in the coordinates on $\Gr_k(V)$ and  $\Gr_{n-k}(V)$ defined by the splitting $W\oplus\operatorname{orth}(W)$, the map $\operatorname{orth}$ sends $\phi\in W^*\otimes\operatorname{orth}(W)$ to $-\phi^t\in (\operatorname{orth}(W))^*\otimes W$. Hence, it is a smooth map.

It defines the vector bundle morphism
\[ s\colon \frac{\Gr_k(V)\times V}{\Taut_k(V)}\to \Gr_k(V)\times V, \qquad (W, x+W)\mapsto (W, \underset{=:s_W(x+W)}{\underbrace{\pr_{\operatorname{orth}(W)}(x)}}),
\]
where for each $W\in \Gr_k(V)$ the map $\pr_{\operatorname{orth}(W)}\colon V\to V$ is the projection of $V$ on $\operatorname{orth}(W)$ defined by the splitting $V=W\oplus \operatorname{orth}(W)$ of $V$.
The map $s$ is smooth since 
\[\begin{tikzcd}
{\Gr_k(V)\times V} && {\Gr_k(V)\times V} & {(W,v)} && {(W,\pr_{\operatorname{orth}(W)}(v))} \\
& {\frac{\Gr_k(V)\times V}{\Taut_k(V)}} &&& {(W,v+W)}
\arrow["{\tilde s}", from=1-1, to=1-3]
\arrow[from=1-1, to=2-2]
\arrow["{\tilde s}", maps to, from=1-4, to=1-6]
\arrow[from=1-4, to=2-5]
\arrow["s"{description}, from=2-2, to=1-3]
\arrow["s"{description}, maps to, from=2-5, to=1-6]
\end{tikzcd}\]
commutes and the top map $\tilde s$ is clearly a smooth vector bundle morphism, while the left projection is a fibration of smooth vector bundles (hence a smooth surjective submersion).
The vector bundle morphism $s$ splits the short exact sequence of vector bundles 
\[0\longrightarrow \Taut_k(V)\longrightarrow \Gr_k(V)\times V \longrightarrow \frac{\Gr_k(V)\times V}{\Taut_k(V)}\longrightarrow 0
\]
over the identity on $\Gr_k(V)$.  Similarly, 
the smooth vector bundle morphism 
\[ \pr\colon \Gr_k(V)\times V \longrightarrow\Taut_k(V), \qquad (W, x)\to (W, \pr_W(x))
\]
is defined by the projections
$\pr_{W}\colon V\to W$ of $V$ on $W$ defined by the same splittings $V=W\oplus \operatorname{orth}(W)$ of $V$.
The vector bundle morphisms $s$ and $\pr$ then define together the smooth splitting 
\[ S\colon E\to \Gr_k(V)\times (V^*\otimes V^*\otimes V)
\] 
of $P\colon \Gr_k(V)\times (V^*\otimes V^*\otimes V)\to E$ 
by 
\[ S\colon (W, \eta)\mapsto (W, \underset{=:S_W(\eta)\in V^*\otimes V^*\otimes V}{\underbrace{s_W\circ \eta\circ (\id_{V}\otimes \pr_W)}}).
\]

\newpage

\bibliographystyle{plain}
\bibliography{ErJo26bibl}

\end{document}